\documentclass[11pt, twoside]{article}

\usepackage[english]{babel}

\usepackage{bbm}
\usepackage{amssymb}
\usepackage{amsfonts}
\usepackage{amsmath}
\usepackage{amsthm}
\usepackage{color}
\usepackage{mathrsfs}
\usepackage{txfonts}
\usepackage{bbm}
\usepackage{enumerate}
\usepackage{anysize}
\usepackage{indentfirst}

\usepackage{latexsym}

\usepackage[colorlinks=true,
  linkcolor=blue,
  citecolor=red,
  urlcolor=magenta,
  ]{hyperref}

\allowdisplaybreaks

\newtheorem{theorem}{Theorem}[section]
\newtheorem{lemma}[theorem]{Lemma}
\newtheorem{corollary}[theorem]{Corollary}
\newtheorem{proposition}[theorem]{Proposition}

\theoremstyle{definition}
\newtheorem{remark}[theorem]{Remark}
\newtheorem{definition}[theorem]{Definition}

\newcounter{assum}

\renewcommand{\appendix}{\par
\setcounter{section}{0}%
\setcounter{subsection}{0}%
\setcounter{subsubsection}{0}%
\gdef\thesection{\@Alph\c@section}%
\gdef\thesubsection{\@Alph\c@section.\@arabic\c@subsection}%
\gdef\theHsection{\@Alph\c@section.}%
\gdef\theHsubsection{\@Alph\c@section.\@arabic\c@subsection}%
\csname appendixmore\endcsname
}

\numberwithin{equation}{section}

\begin{document}

\arraycolsep=1pt

\title{\bf\Large Sobolev--Morrey Spaces and Divergence-Form Degenerate
Second-Order Elliptic Equations on Domains
with Higher Co-Dimensional Boundaries
\footnotetext{\hspace{-0.35cm} 2020
\emph{Mathematics Subject Classification}.
Primary 46E35; Secondary 35J25, 35J70, 28A75, 42B35, 42B37, 46B70.
\endgraf \emph{Key words and phrases}.
weighted Sobolev--Morrey space,
higher co-dimensional boundary,
trace theorem,
extension theorem,
degenerate elliptic equation.
\endgraf
This project is partially supported by the National Natural
Science Foundation of China (Grant Nos. 12431006 and 12371093), the Beijing
Natural Science Foundation (Grant No.1262011),
the Fundamental Research Funds for the Central Universities
(Grant No. 2233300008), and Longyuan Young Talents of Gansu Province.}}
\author{Weiyi Kong, Yoshihiro Sawano, Dachun Yang\footnote{Corresponding author,
E-mail: \texttt{dcyang@bnu.edu.cn}/{\color{red}\today}/Final Version.},
\ \ Sibei Yang and Wen Yuan}
\date{}
\maketitle

\vspace{-0.6cm}

\begin{center}
\begin{minipage}{13cm}
{\small {\bf Abstract}\quad In this article,
we study the weighted homogeneous Sobolev--Morrey
spaces on domains in $\mathbb{R}^n$ with higher
co-dimensional boundaries. Precisely, we systematically establish a real-variable
theory of these spaces, including completeness, embedding theorems, Riesz
potential characterizations, continuity, trace and extension theorems, and
complex interpolation. Applying the boundedness of the trace and the
extension operators, we obtain sharp weighted a priori estimates
for solutions to the Dirichlet problem of divergence-form
degenerate second-order elliptic equations on such domains in
weighted Lebesgue spaces. The absence of a boundary manifold
structure of these domains poses some essential difficulties,
which are overcome by  using some tools, such as
the intrinsic properties of distance weights and the geometric structure
of domains, different from those available in Lipschitz domains.
}
\end{minipage}
\end{center}

\vspace{0.2cm}

\tableofcontents

\vspace{0.2cm}

\section{Introduction}
The restriction (trace) and the extension of function spaces,
particularly the boundedness of these operators, play a
crucial role in the regularity theory of elliptic equations on domains. The
regularity theory of elliptic equations on domains, in turn, constitutes an
important topic in partial differential equations (for short, PDEs). In the classical framework, the
domain is usually assumed to have a Lipschitz boundary or, more generally, a
boundary of co-dimension one (see, for example, \cite{bw(am-2010), d(amj-1986),
DEK18, HaLi11, K1994, Shen05}). Recently, David, Feneuil, and Mayboroda \cite{DFM21}
introduced the domains with \emph{higher co-dimensional boundaries},
systematically studied the corresponding weighted Sobolev space
$\dot{W}^{1,2}$ and its trace space,
established the boundedness of the associated trace and extension operators,
and applied these real-variable tools to investigate the weighted
$\dot{W}^{1,2}$ regularity of solutions to the Dirichlet
problem of divergence-form degenerate second-order elliptic
equations on such domains.
We refer to \cite{bmp(2603.22067),dfm(amses-2019), dfm(a-2023), dm(imrn-2023), fmz(rmi-2021)}
for further studies on the restriction and the
extension of function spaces and their applications to elliptic equations
on such domains.

On the other hand, to obtain more refined or finer regularity of solutions, Morrey spaces,
originally introduced by Morrey \cite{m(tams-1938)}, serve as a natural
extension of classical Lebesgue spaces and are of fundamental importance to the
regularity theory of elliptic equations. In the study of PDEs, one typically
needs to control the derivatives of functions as well. This leads to the
development of Sobolev--Morrey spaces. Nowadays, Morrey and Sobolev--Morrey spaces
have become important tools and have widespread impact in harmonic analysis
(see, for example, \cite{Adams-Book,hlms(amses-2025),hms(jfaa-2020),h(aasfm-2015),h(jga-2023),
lx(jfa-2018), SDH20-1, s(emj-2012), s(emj-2013)})
and in PDEs (see, for example, \cite{ax(jlms-2025), ax(ma-2025),
bb(na-2020), bdl(jde-2016), bdl(na-2019), lx(jmpa-2020),
lx(am-2025), NNS16, sty(amses-2018)}).

Furthermore, when dealing with elliptic equations on domains with irregular
boundaries, it becomes necessary to study weighted Morrey and weighted
Sobolev--Morrey spaces. Generally, there exist two types of weighted Morrey
spaces: the Komori--Shirai type \cite{Komori} and the Samko type \cite{Samko}.
Motivated by the classical study of elliptic equations in weighted Morrey
spaces, we adopt the Samko-type formulation in this article [see
\eqref{20260604.2231}]. We refer to \cite{DuRo20-1, DuRo20-2,
DuRo21, h(pems-2021), n(mn-1994), n(sm-2006), n(sm-2008),
Nakamura16, NST18, Sawano-Book II} for more studies on weighted
Morrey spaces.

Let $n\ge2$. Throughout this article, we \emph{always} assume that
$\Gamma\subset\mathbb{R}^n$ is an Ahlfors--David-regular set with dimension
$d\in(0,n-1)$; that is, $\Gamma$ is closed and there exists a constant
$C_0\in[1,\infty)$ such that, for any $x\in\Gamma$ and $r\in(0,\infty)$,
\begin{align*}
C_0^{-1}r^d\leq\mathcal{H}^d(\Gamma(x,r))\leq C_0 r^d.
\end{align*}
Here, and thereafter, for any $s\in(0,n]$, $x\in\mathbb{R}^n$, and
$r\in(0,\infty)$, $\mathcal{H}^s$ denotes the $s$-dimensional Hausdorff
measure,
$$\Gamma(x,r):=\Gamma\cap B(x,r),$$
and $B(x,r):=\{y\in\mathbb{R}^n:|y-x|<r\}$ is the open \emph{ball} in
$\mathbb{R}^n$ centered at $x$ with radius $r$.
Such a set $\Gamma$ is also called a
\emph{$d$-set} (see, for example, \cite{jw84,t97}).
The function spaces and their applications on $d$-sets also attract a lot of
attention (see, for example, \cite{hp(mn-2008), hs(aasfm-2011),
s(psim-2023), t97, y(sm-2003), ysy(jfaa-2025)}). Let
$$\Omega:=\mathbb{R}^n\setminus\Gamma.$$
Since $d<n-1$, such a domain $\Omega$ is referred to as a \emph{domain with a
higher co-dimensional boundary}. For any $x\in\mathbb{R}^n$, we define the
\emph{distance function}
$$\delta(x):=\operatorname{dist}(x,\Gamma):=\inf_{y\in\Gamma}|x-y|$$
and the \emph{distance weight}
\begin{align}\label{dw}
w(x):=[\delta(x)]^{d+1-n}.
\end{align}
Obviously, if $x\in\Gamma$, then $\delta(x)=0$ and $w(x)=\infty$.
Throughout this article, we always use the \emph{symbol $\Omega$} to denote a
domain with a higher co-dimensional boundary; i.e., the $d$-set $\Gamma$.

There exist two typical examples of such $d$-sets $\Gamma$. The first one is
the \emph{flat lower-dimensional boundary} $\mathcal{F}:=\mathbb{R}^{n-2}\times
\{(0,0)\}$ with $n\in\mathbb{N}\cap[3,\infty)$, which is an $(n-2)$-set.
Another canonical example is the \emph{Cantor dust} in
$\mathbb{R}^n$ (see, for example, \cite{F2014}). Precisely, let
$\mathcal{A}_0:=[0,2]^n$ and $a\in(0,\frac{1}{2^\frac{n}{n-1}})$. For any
$j\in\mathbb{N}$, define $\mathcal{A}_j:=\bigcup_{k\in\{0,1\}^n}[a\mathcal{A}_{j-1}+2(1-a)k]$ and let
$\mathcal{A}:=\bigcap_{j\in\mathbb{Z}_+}\mathcal{A}_j$. This set $\mathcal{A}$
is called the \emph{Cantor dust} and is a $d$-set,
where $d\in(0,n-1)$ satisfies $2^na^d=1$.

In this article, motivated by the work of David et al. \cite{DFM21} on the
real-variable theory of the weighted Sobolev space $\dot{W}^{1,2}(\Omega,w)$,
we study the weighted Sobolev--Morrey space
$\dot{W}^1\mathcal{M}^p_q(\Omega,w)$,
where $w$ is as in \eqref{dw}, $1<q\le p<\infty$,
and $\mathcal{M}^p_q(\Omega,w)$ denotes the Morrey space
in \eqref{20260604.2231}. More precisely, we systematically establish
a real-variable theory for these spaces, including completeness,
embedding theorems, Riesz potential characterizations, continuity, trace and
extension theorems, and complex interpolation. Note that
$\dot{W}^{1,2}(\Omega,w)=\dot{W}^1\mathcal{M}^2_2(\Omega,w)$.
Thus, these spaces are new when $1<q\le p<\infty$ and $p\neq 2\neq q$.
Even in the special case $p\neq2$ of $\dot{W}^{1,p}(\Omega,w)$
(which is precisely $\dot{W}^1\mathcal{M}^p_p(\Omega,w)$),
our results are also new.
Furthermore, applying the real-variable theory of
$\dot{W}^1\mathcal{M}^p_q(\Omega,w)$,
we obtain sharp weighted a priori estimates for solutions to the Dirichlet
problem of divergence-form degenerate second-order elliptic equations on
$\Omega$ in weighted Lebesgue spaces.

It is worth mentioning that a systematic study of degenerate elliptic
equations can be traced back to Fabes et al. \cite{FBS82}, where the analysis was
developed on a class of well-behaved domains. David et al. \cite{DFM21}
first considered divergence-form degenerate elliptic equations on $\Omega$.
Although many phenomena proved similar to
those in the classical case of bounded $C^\infty$ domains
as treated in \cite{HaLi11}, David et al. \cite{DFM21}
employed many new tools.

Indeed, compared with the classical case of domains with Lipschitz
boundaries, the present setting poses substantial analytical challenges.
The main difficulty is that the boundary $\Gamma$ has no manifold structure
and hence cannot be locally flattened or parameterized by boundary charts.
Consequently, boundary estimates cannot be reduced to the standard
half-space model, and the usual arguments based on normal directions,
surface coordinates, or classical boundary traces are no longer available.
To overcome these difficulties, we exploit the intrinsic geometry of the
$d$-set $\Gamma$, the quantitative properties of the distance
weight $w$, and a Whitney decomposition adapted to $\Omega$.
These tools allow us to establish Poincar\'e-type inequalities,
trace and extension theorems, and ultimately sharp weighted a
priori estimates in weighted Lebesgue spaces.

The organization of the remainder of this article is as follows.

In Section \ref{section2}, we prepare the necessary geometric and analytical
foundations regarding the distance weight $w$ and boundary traces.
Specifically, in Subsection \ref{subsection2.1}, we
examine the quantitative behavior of ball averages of $w$ with
the sharp range of $p$ (see Lemma
\ref{lem:250104-11}). Furthermore, we characterize the weight classes which $w$
belongs to (see Remark \ref{rmk:20260508.2141} and Proposition
\ref{prop:260117-1}). In Subsection \ref{subsection2.2}, we present various
Poincar\'e inequalities for functions in Sobolev spaces (see Lemma
\ref{DFM21} and Propositions \ref{prop:Poincare} and \ref{prop:Poincare2}) and
derive the existence and the differentiation properties of the boundary trace for
these functions (see Proposition \ref{prop:260126-1}).

In Section \ref{section3}, we introduce and investigate the weighted Morrey
spaces $\mathcal{M}^p_q(\Omega,w)$ adapted to $\Gamma$. Specifically, in
Subsection \ref{subsection3.1} we introduce these spaces and
present their fundamental structural properties (see Propositions
\ref{lem:250919-2} and \ref{prop:250926-1}
and Theorem \ref{prop:250104-11}). We mention that Proposition
\ref{prop:250926-1} and Theorem \ref{prop:250104-11} are sharp.
In Subsection \ref{subsection3.2}, we state the boundedness results of classical
operators on $\mathcal{M}^p_q(\Omega,w)$ (see Propositions
\ref{thm:260126-1}--\ref{prop:20260420.1836}). In Subsection
\ref{subsection3.3}, by showing that $\mathcal{M}^p_q(\Omega,w)$
can be embedded into the space $\mathcal{S}'(\mathbb{R}^n)$
of tempered distributions, we obtain its Littlewood--Paley
characterization (see Propositions \ref{lem:250919-31} and
\ref{prop:LP-Morrey}), which serves as a crucial tool for the Riesz potential
characterization in Section \ref{section4} below. Finally, the complex
interpolation for $\mathcal{M}^p_q(\Omega,w)$ is presented in
Subsection \ref{subsection3.4} (see Proposition \ref{thm:interp-morrey}).

In Section \ref{section4}, we systematically
develop a real-variable theory
of weighted Sobolev--Morrey spaces $\dot{W}^1\mathcal{M}^p_q(\Omega,w)$
adapted to $\Gamma$. Specifically, Subsection
\ref{subsection4.1} is devoted to showing their completeness (see Lemma
\ref{lem:250919-3}). In Subsection \ref{subsection4.2}, we establish the Riesz
potential characterization of $\dot{W}^1\mathcal{M}^p_q(\Omega,w)$, which
subsequently yields the corresponding Sobolev--Morrey embedding theorem
(see Theorem \ref{cor:260117-1} and Corollary \ref{cor:260605}). Subsection
\ref{subsection4.3} examines the convergence of integral averages at infinity
under the lower critical case (see Lemma \ref{lem:250916-1}), while Subsection
\ref{subsection4.4} addresses the pointwise continuity and quantitative
H\"older estimates under the upper critical case (see Lemmas
\ref{lem:260126-11} and \ref{lem:local-continuity} and Theorem \ref{thm:5.1}
with the range of $p$ being sharp).
In Subsections \ref{subsection4.5} and \ref{subsection4.6}, we turn our
attention to the boundary behavior on $\Gamma$, where we introduce the trace
space $Q^p_q(\Gamma)$ and establish the mapping properties of both the trace
operator $T$ and the extension operator $E$ (see Theorems
\ref{thm:local-trace} and \ref{thm:Eg-membership}). Finally, in Subsection
\ref{subsection4.7}, we establish the complex interpolation identities for
both $\dot{W}^1\mathcal{M}^p_q(\Omega,w)$ and $Q^p_q(\Gamma)$ (see Theorem
\ref{thm:interp-sobolev-morrey} and Corollary \ref{cor:interp-Qpq}).

In Section \ref{section5}, as an application of the function space theory
developed in the preceding sections, we establish the weighted a priori
estimates in a sharp range for solutions to the Dirichlet problem
of divergence-form degenerate second-order elliptic
equations on $\Omega$. Specifically, in Subsection \ref{subsection5.1},
we recall the definition of solutions and prove the reverse H\"older
inequalities for local solutions (see Lemma \ref{lem:7.1}). Combining this,
a real-variable lemma of Gehring type (see Lemma \ref{lem:Shen}), the
weighted $\dot{W}^{1,2}$ estimates established by David et al. \cite[Lemma
9.1]{DFM21}, and the real-variable tools obtained in Sections
\ref{section3} and \ref{section4} (especially the trace and the extension
theorems), we obtain the regularity properties of solutions to weighted Sobolev
and weighted Sobolev--Morrey settings in a sharp range in Subsection
\ref{subsection5.2} (see Theorem \ref{thm:weighted-dirichlet}, Corollaries
\ref{thm:weighted-dirichlet-cor} and \ref{cor:7.1},
and Remark \ref{rmk:20260521.1442}).

We end this introduction by making some notational conventions. Throughout
this article, let $\mathbb{N}:=\{1,\,2,\,\dots\}$ and
$\mathbb{Z}_+:=\mathbb{N}\cup\{0\}$. For any $s\in\mathbb{R}$, the \emph{symbol}
$\lceil s\rceil$ denotes the smallest integer not less than $s$. For any given
$p\in[1,\infty]$, we denote by $p'$ its \emph{conjugate exponent}; i.e.,
$\frac1p+\frac{1}{p'}=1$. We always denote by $C$ a \emph{positive constant}
which is independent of the main parameters involved, but it may vary from line
to line. The \emph{notation} $f\lesssim g$ means that $f\le Cg$.
If $f\lesssim g$ and $g\lesssim f$, then we write $f\sim g$.
If $f\le Cg$ and $g=h$ or $g\le h$, we then write $f\lesssim g=h$ or $f\lesssim
g\le h$. For any measurable set $A\subset\mathbb{R}^n$, we
denote by $|A|$ its Lebesgue measure. For any $a\in(0,\infty)$ and any ball
$B:=B(x,r)\subset\mathbb{R}^n$ with $x\in\mathbb{R}^n$ and $r\in(0,\infty)$,
let $aB:=B(x,ar)$ be the ball centered at $x$ with radius $ar$.
Assume that $\mathcal{X}\subset\mathbb{R}^n$ is a measurable.
For any given $p\in[1,\infty]$, let
$L_\mathrm{loc}^p(\mathcal{X})$ denote the set of all
$p$-locally integrable functions on $\mathcal{X}$. Similarly, for any given nonnegative locally integrable function $\omega$ on
$\mathbb{R}^n$, we denote by $L_\mathrm{loc}^p(\mathcal{X},
\omega)$ the set of all $p$-locally integrable functions on
$\mathcal{X}$ with respect to the measure $\omega(x)\,dx$.
In addition, the weighted \emph{Lebesgue space
$L^p(\mathcal{X,\omega}$)} is defined to be the
space of all measurable functions $f$ on $\mathcal{X}$ such that
\begin{align*}
\|f\|_{L^p(\mathcal{X,\omega})}:=
\begin{cases}
\displaystyle
\left[\int_\mathcal{X}|f(x)|^p\omega(x)\,dx\right]^\frac1p
&\mathrm{if}\ p\in[1,\infty),\\
\displaystyle\mathop\mathrm{ess\,sup}_{x\in\mathcal{X}}|f(x)|&\mathrm{if}\ p=\infty\\
\end{cases}
\end{align*}
is finite. If $\omega\equiv1$, then we simply write
$L^p(\mathcal{X},\omega)$ as $L^p(\mathcal{X})$.
For any given Borel measure $\mu$ on $\mathbb{R}^n$, any
$\mu$-measurable subset $A\subset\mathbb{R}^n$ with
$\mu(A)\in(0,\infty)$, and any $\mu$-locally integrable function
$f$ on $\mathbb{R}^n$, let
\begin{align*}
f_A:=\fint_Af(x)\,d\mu(x):=\frac{1}{\mu(A)}\int_Af(x)\,d\mu(x).
\end{align*}
For any given open subset $U\subset\mathbb{R}^n$,
denote by $C_\mathrm{c}^\infty(U)$ the \emph{set} of all
infinitely differentiable functions on $U$ with compact support.
For any given normed spaces $\mathcal X$ and $\mathcal Y$ with the
corresponding norms $\|\cdot\|_{\mathcal X}$ and
$\|\cdot\|_{\mathcal Y}$, the \emph{notation} ${\mathcal X}\hookrightarrow{\mathcal Y}$
means that, if $f\in \mathcal X$, then $f\in\mathcal Y$ and
$\|f\|_{\mathcal Y}\lesssim \|f\|_{\mathcal X}$ with the implicit positive constant
independent of $f$. Finally, in all proofs, we consistently retain the
notation introduced in the original theorem (or related statement).

\section{General Remarks}\label{section2}

This section consists of two subsections. In Subsection \ref{subsection2.1},
we study the analytical and the geometric properties of the distance weight $w$.
Specifically, we examine the quantitative behavior of ball averages of $w$.
Furthermore, we also prove that $w$ belongs to the Muckenhoupt class
$A_q(\mathbb{R}^n)$ for $q\in[1,\infty]$ and we provide a characterization for
$w$ belonging to the weight class $\mathcal{B}_{p,q}(\mathbb{R}^n)$ introduced
by Nakamura \cite[Definition 1.1]{Nakamura16}.
In Subsection \ref{subsection2.2}, we establish the Poincar\'e inequality for
functions in the Sobolev space and derive the existence and
differentiation properties of the boundary trace for these functions.

\subsection{Analytical and Geometric Properties of Distance Weight $w$}
\label{subsection2.1}

In this subsection, we discuss the analytical and the geometric properties of
$w$. We begin by presenting the geometric properties which $w$ satisfies,
generalizing the corresponding results in \cite[Lemma 2.3 and (2.13)]{DFM21}.
More precisely, (i) and (ii) of Lemma \ref{lem:250104-11} relax the
restriction $p=1$ in \cite[Lemma 2.3]{DFM21} to the sharp range
$p\in(-\infty,0)\cup(0,\frac{n-d}{n-d-1})$, while Lemma
\ref{lem:250104-11}(iii) extends the diagonal case $q_0=q_1=1$ in
\cite[(2.13)]{DFM21} to the off-diagonal case $1\leq q_0\leq q_1<\infty$.
Lemma \ref{lem:250104-11} is of independent interest.

\begin{lemma}\label{lem:250104-11}
Let $1\leq q_0\leq q_1<\infty$,
$p\in(-\infty,0)\cup(0,\frac{n-d}{n-d-1})$,
and $w$ be as in \eqref{dw}.
Then the following assertions hold.
\begin{enumerate}
\item[\rm(i)] For any $x\in\mathbb{R}^n$ and $r\in(0,\infty)$ satisfying
$\delta(x)\geq2r$,
\begin{align}\label{eq:general2.1}
\left\{\fint_{B(x,r)}[w(z)]^p\,dz\right\}^\frac1p\sim w(x),
\end{align}
where the positive equivalence constants are independent of $x$ and $r$.

\item[{\rm(ii)}] For any $x\in\mathbb{R}^n$ and $r\in(0,\infty)$ satisfying
$\delta(x)\leq2r$,
\begin{align}\label{eq:general2.2}
\left\{\fint_{B(x,r)}[w(z)]^p\,dz\right\}^\frac1p\sim r^{d+1-n},
\end{align}
where the positive equivalence constants are independent of $x$ and $r$.

\item[\rm(iii)] For any $x\in\mathbb{R}^n$,  $r\in(0,\infty)$, and any
measurable function $f$ on $B(x,r)$,
\begin{align}\label{20260508.2150}
\left[\fint_{B(x,r)}|f(y)|^{q_0}\,dy\right]^\frac{1}{q_0}
\lesssim\left[\frac{1}{w(B(x,r))}
\int_{B(x,r)}|f(y)|^{q_1}w(y)\,dy\right]^\frac{1}{q_1},
\end{align}
where the positive implicit constant is
independent of $x$, $r$, and $f$.
\end{enumerate}
\end{lemma}

\begin{proof}
We first prove (i). Let $x\in\mathbb{R}^n$ and $r\in(0,\infty)$ be such that
$\delta(x)\geq2r$. Then, for any $y\in B(x,r)$, $\frac12\delta(x)\leq\delta(y)
\leq\frac32\delta(x)$, which further implies that $(\frac{2}{3})^{n-d-1}w(x)
\leq w(y)\leq 2^{n-d-1}w(x)$. This immediately yields \eqref{eq:general2.1}.

Next, we show (ii). We first establish the lower bound in \eqref{eq:general2.2}.
Observe that, for any $y\in B(x,r)$, $\delta(y)\leq3r$. Thus,
\begin{align*}
r^{d+1-n}\lesssim\left\{\fint_{B(x,r)}[\delta(z)]^{p(d+1-n)}\,dz
\right\}^\frac1p=\left\{\fint_{B(x,r)}[w(z)]^p\,dz\right\}^\frac1p.
\end{align*}
We then prove the upper bound in \eqref{eq:general2.2}. To do this,
we consider the following two cases for $p$.

\emph{Case (1)} $p\in(-\infty,0)$. In this case, for any $k\in\mathbb{Z}_+$, let
\begin{align*}
Z_k:=\left\{y\in B(x,r):2^{-k-1}r<\delta(y)\le 2^{-k}r\right\}
\mbox{\ \ and\ \ }E_k:=\left\{y\in B(x,r):\delta(y)\leq2^{-k}r\right\}.
\end{align*}
It was shown in \cite[p.\,11]{DFM21} that, for any $k\in\mathbb{Z}_+$,
$|Z_k|\lesssim2^{kd}(2^{-k}r)^n$, which further implies that,
for any $k\in\mathbb{Z}_+$,
\begin{align*}
|E_k|\lesssim\sum_{j=k}^\infty2^{jd}\left(2^{-j}r\right)^n
\lesssim2^{k(d-n)}r^n.
\end{align*}
Therefore, there exists $K\in\mathbb{N}$ such that
$|E_K|\lesssim2^{K(d-n)}r^n\leq\frac12|B(x,r)|$
and hence $|B(x,r)\setminus E_K|>\frac12|B(x,r)|$. Consequently,
\begin{align*}
\left\{\fint_{B(x,r)}[w(z)]^p\,dz\right\}^\frac1p\lesssim
\left\{\fint_{B(x,r)\setminus E_K}[\delta(z)]^{p(d+1-n)}\,dz
\right\}^\frac1p\lesssim r^{d+1-n}.
\end{align*}

\emph{Case (2)} $p\in(0,\frac{n-d}{n-d-1})$. In this case, if $x\in\Gamma$,
then, from the proven conclusion that $|Z_k|\lesssim2^{kd}(2^{-k}r)^n$
for any $k\in\mathbb{Z}_+$, we deduce that
\begin{align*}
\int_{B(x,r)}[w(z)]^p\,dz&=\sum_{k\in\mathbb{Z}_+}\int_{Z_k}
[\delta(z)]^{p(d+1-n)}\,dz\lesssim\sum_{k\in\mathbb{Z}_+}
2^{kd}\left(2^{-k}r\right)^{n+p(d+1-n)}\\
&=r^{n+p(d+1-n)}\sum_{k\in\mathbb{Z}_+}2^{k[(d-n)(1-p)-p]}\sim r^{n+p(d+1-n)}
\end{align*}
and hence
\begin{align}\label{20260508.1847}
\left[\fint_{B(x,r)}[w(z)]^p\,dz\right]^\frac1p\lesssim r^{d+1-n}.
\end{align}
Otherwise, $x\in\Omega$ and $\delta(x)\leq2r$. In this case, there exists
$\xi_x\in\Gamma$ such that $\delta(x)=|x-\xi_x|$ and hence $B(x,r)\subset
B(\xi_x,3r)$. Applying this and \eqref{20260508.1847} with $x$ therein
replaced by $\xi_x$, we find that
\begin{align*}
\left[\fint_{B(x,r)}[w(z)]^p\,dz\right]^\frac1p\lesssim
\left[\fint_{B(\xi_x,3r)}[w(z)]^p\,dz\right]^\frac1p\lesssim r^{d+1-n}.
\end{align*}
This finishes the proof of (ii).

Finally, we show (iii). Let $x\in\mathbb{R}^n$ and $r\in(0,\infty)$, and
let $f$ be a measurable function on $B(x,r)$,
We consider the following two cases for $q_0$ and $q_1$.

\emph{Case (1)} $q_0=q_1=:q$. In this case, from Cavalieri's principle and
the doubling condition for $w$ (see Remark \ref{rmk:20260508.2141}(ii) below),
it follows that
\begin{align*}
\fint_{B(x,r)}|f(y)|^q\,dy&=\frac{1}{|B(x,r)|}\int_{B(x,r)}qt^{q-1}
|\{y\in B:|f(y)|>t\}|\,dt\\
&\lesssim\frac{1}{w(B(x,r))}\int_{B(x,r)}qt^{q-1}w(\{y\in B:|f(y)|>t\})\,dt\\
&=\frac{1}{w(B(x,r))}\int_{B(x,r)}|f(y)|^{q_1}w(y)\,dy
\end{align*}
and hence \eqref{20260508.2150} holds.

\emph{Case (2)} $q_0<q_1$. In this case, define $p:=\frac{q_1}{q_1-q_0}$.
Then $\frac1p+\frac{q_0}{q_1}=1$. Using this, H\"older's inequality, (i),
and (ii), we conclude that
\begin{align*}
\left[\fint_{B(x,r)}|f(y)|^{q_0}\,dy\right]^\frac{1}{q_0}
&=\left\{\fint_{B(x,r)}|f(y)|^{q_0}[w(y)]
^\frac{q_0}{q_1}[w(y)]^{-\frac{q_0}{q_1}}\,dy\right\}^\frac{1}{q_0}\\
&\leq\left\{\fint_{B(x,r)}|f(y)|^{q_1}w(y)\,dy\right\}^\frac{1}{q_1}
\left\{\fint_{B(x,r)}[w(y)]^{-\frac{q_0p}{q_1}}\,dy\right\}^\frac{1}{pq_0}\\
&\sim\left\{\fint_{B(x,r)}|f(y)|^{q_1}w(y)\,dy\right\}^\frac{1}{q_1}
\left[\fint_{B(x,r)}w(y)\,dy\right]^{-\frac{1}{q_0}}\\
&=\left[\int_{B(x,r)}|f(y)|^{q_1}w(y)\,dy\right]^\frac{1}{q_1}
\end{align*}
and hence \eqref{20260508.2150} holds. This finishes the proof of (iii) and
Lemma \ref{lem:250104-11}.
\end{proof}

\begin{remark}\label{rmk:20260508.2141}
We use the same notation as in Lemma \ref{lem:250104-11}.
\begin{enumerate}
\item[\rm(i)] If $p:=1$, then (i)--(iii) of Lemma \ref{lem:250104-11}
reduce, respectively, to \cite[Lemma 2.3(i)]{DFM21},
\cite[Lemma 2.3(ii)]{DFM21}, and \cite[(2.13)]{DFM21}.
Moreover, the restriction $p<\frac{n-d}{n-d-1}$ here is sharp. To prove this,
let $x\in\Gamma$ and $r\in(0,\infty)$.
By Tonelli's theorem, we find that
\begin{align}\label{20260629.1039}
\int_{B(x,r)}[w(z)]^p\,dz&=\int_{B(x,r)}[\delta(z)]^{-p(n-d-1)}\,dz\notag\\
&=p(n-d-1)\int_{B(x,r)}\int_0^\infty t^{-p(n-d-1)-1}
\boldsymbol{1}_{\{t\in(0,\infty):t>\delta(z)\}}\,dt\,dz\notag\\
&=p(n-d-1)\int_0^\infty t^{-p(n-d-1)-1}
|\{z\in B(x,r):\delta(z)<t\}|\,dt.
\end{align}
We claim that, for any $t\in(0,2r)$,
\begin{align*}
r^dt^{n-d}\lesssim|\{z\in B(x,r):\delta(z)<t\}|,
\end{align*}
where the implicit positive constant is independent of $r$ and $t$.
Indeed, it is easy to show that there exist a maximal $N\in\mathbb{N}$
and points $\{y_j\}_{j\in\mathbb{N}\cap[1,N]}\subset\Gamma(x,\frac{r}{2})$
such that $|y_i-y_j|\geq t$ for any $i,j\in\mathbb{N}\cap[1,N]$ with $i\neq j$.
Since $N$ is maximal, it follows that
\begin{align*}
\Gamma\left(x,\frac{r}{2}\right)\subset\bigcup_{j=1}^N\Gamma\left(y_j,t\right).
\end{align*}
From this and the assumption that $\Gamma$ is a $d$-set, we deduce that
\begin{align*}
r^d\sim\mathcal{H}^d\left(\Gamma\left(x,\frac{r}{2}\right)\right)
\leq\sum_{j=1}^N\mathcal{H}^d\left(\Gamma\left(y_j,t\right)\right)
\sim Nt^d,
\end{align*}
which further implies that $(\frac{r}{t})^d\lesssim N$. Moreover, for any
$j\in\mathbb{N}\cap[1,N]$ and $z\in B(y_j,\frac{t}{4})$,
$|z-x|\leq|z-y_j|+|y_j-x|<\frac{t}{4}+\frac{r}{2}<r$ and
$\delta(z)\leq|z-y_j|<\frac{t}{4}<t$. Thus,
\begin{align*}
\bigcup_{j=1}^NB\left(y_j,\frac{t}{4}\right)
\subset\{z\in B(x,r):\delta(z)<t\}.
\end{align*}
Combining this, the proven conclusion that $(\frac{r}{t})^d\lesssim N$,
and the fact that $\{B(y_j,\frac{t}{4})\}_{j\in\mathbb{N}\cap[1,N]}$
are pairwise disjoint, we conclude that
\begin{align*}
|\{z\in B(x,r):\delta(z)<t\}|\geq
\left|\bigcup_{i=1}^NB\left(y_j,\frac{t}{4}\right)\right|\sim Nt^n
\gtrsim r^dt^{n-d},
\end{align*}
and hence the above claim holds. From the above claim and
\eqref{20260629.1039}, we infer that, if $p\geq\frac{n-d}{n-d-1}$, then
\begin{align*}
\int_{B(x,r)}[w(z)]^p\,dz\gtrsim r^d\int_0^{2r}t^{n-d-p(n-d-1)-1}\,dt=\infty,
\end{align*}
which contradicts Lemma \ref{lem:250104-11}(ii). Therefore, the
restriction $p<\frac{n-d}{n-d-1}$ is sharp.

\item[\rm(ii)] Let $q\in[1,\infty]$. Recall that the
\emph{Muckenhoupt class $A_q(\mathbb{R}^n)$} is defined to be the
set of all nonnegative locally integrable functions $\omega$ on $\mathbb{R}^n$
such that
\begin{align*}
[\omega]_{A_q(\mathbb{R}^n)}:=
\begin{cases}
\displaystyle
\sup_B\fint_B\omega(x)\,dx\left[\mathop\mathrm{ess\,inf}_{x\in B}
\omega(x)\right]^{-1}<\infty &\mathrm{if}\ q=1,\\
\displaystyle
\sup_B\fint_B\omega(x)\,dx\left\{
\fint_B[\omega(x)]^\frac{1}{1-q}\,dx\right\}^{q-1}<\infty
&\mathrm{if}\ q\in(1,\infty),\\
\displaystyle
\sup_B\fint_B\omega(x)\,dx\exp\left\{\fint_B\ln\frac{1}{\omega(x)}\,dx
\right\}<\infty &\mathrm{if}\ q=\infty,
\end{cases}
\end{align*}
where the suprema are taken over all balls $B\subset\mathbb{R}^n$.
As a direct consequence of (i) and (ii) of
Lemma \ref{lem:250104-11}, we find that $w$ belongs to the Muckenhoupt
class $A_q(\mathbb{R}^n)$, where $q\in[1,\infty]$.
This further implies that $w$ has the doubling property; i.e.,
there exists a positive constant $C$ such that, for any measurable set
$B\subset\mathbb{R}^n$ and any measurable subset $A\subset B$,
$\frac{|A|}{|B|}\leq C\frac{w(A)}{w(B)}$.
\end{enumerate}
\end{remark}

Let $1\leq q\leq p<\infty$. Recall that the weight class
$\mathcal{B}_{p,q}(\mathbb{R}^n)$, introduced by Nakamura
\cite[Definition 1.1]{Nakamura16}, is defined to be the set of all nonnegative
locally integrable functions $\omega$ on $\mathbb{R}^n$ such that there exists
a positive constant $C$ satisfying, for any $x\in\mathbb{R}^n$
and $0<r\leq R<\infty$,
\begin{align*}
|B(x,r)|^{\frac1p-\frac1q}[\omega(B(x,r))]^\frac1q
\leq C|B(x,R)|^{\frac1p-\frac1q}[\omega(B(x,R))]^\frac1q.
\end{align*}
Roughly speaking, the class $\mathcal{B}_{p,q}(\mathbb{R}^n)$ is introduced to
ensure that the weighted Morrey norm of the characteristic function
$\boldsymbol{1}_{B(x,r)}$ can be estimated efficiently for any
$x\in\mathbb{R}^n$ and $r\in(0,\infty)$ (see Proposition \ref{prop:250926-1}
below). This condition also plays an important role in showing the weighted
boundedness of classical operators on weighted Morrey spaces
(see Subsection \ref{subsection3.2}).
We characterize the condition for $w\in B_{p,q}(\mathbb{R}^n)$.

\begin{proposition}\label{prop:260117-1}
Let $1\leq q\leq p<\infty$ and $w$ be as in \eqref{dw}.
Then $w\in\mathcal{B}_{p,q}(\mathbb{R}^n)$ if and only if
\begin{align}\label{eq:cond-pq}
\frac{n}{p}-\frac{n-d-1}{q}\geq0.
\end{align}
\end{proposition}

\begin{proof}
We first prove the sufficiency. Let $x\in\mathbb{R}^n$ and $0<r\leq R<\infty$.
We consider the following three cases for $x$, $r$, and $R$.

\emph{Case (1)} $\delta(x)<2r$. In this case, from Lemma
\ref{lem:250104-11}(ii) and \eqref{eq:cond-pq}, we infer that
\begin{align*}
|B(x,r)|^{\frac1p-\frac1q}[w(B(x,r))]^\frac1q
\sim r^{\frac{n}{p}-\frac{n-d-1}{q}}\leq R^{\frac{n}{p}-\frac{n-d-1}{q}}
\sim|B(x,R)|^{\frac1p-\frac1q}[w(B(x,R))]^\frac1q.
\end{align*}

\emph{Case (2)} $\delta(x)>2R$. In this case, by Lemma \ref{lem:250104-11}(i), we obtain
\begin{align*}
|B(x,r)|^{\frac1p-\frac1q}[w(B(x,r))]^\frac1q
&\sim|B(x,r)|^\frac1p[w(x)]^\frac1q\\
&\lesssim|B(x,R)|^\frac1p[w(x)]^\frac1q\sim
|B(x,R)|^{\frac1p-\frac1q}[w(B(x,R))]^\frac1q.
\end{align*}

\emph{Case (3)} $\delta(x)\geq2r$ and $\delta(x)\leq2R$. In this case, using (i) and (ii) of Lemma \ref{lem:250104-11} and
\eqref{eq:cond-pq}, we conclude that
\begin{align*}
|B(x,r)|^{\frac1p-\frac1q}[w(B(x,r))]^\frac1q
&\sim|B(x,r)|^\frac1p[w(x)]^\frac1q\lesssim r^{\frac{n}{p}-\frac{n-d-1}{q}}\\
&\leq R^{\frac{n}{p}-\frac{n-d-1}{q}}\sim
|B(x,R)|^{\frac1p-\frac1q}[w(B(x,R))]^\frac1q.
\end{align*}
This finishes the proof of the sufficiency.

Next, we show the necessity. Let $x\in\Gamma$. Then, for any $0<r\leq R<\infty$,
\begin{align*}
r^{\frac{n}{p}-\frac{n-d-1}{q}}\sim|B(x,r)|^{\frac1p-\frac1q}[w(B(x,r))]^\frac1q
\lesssim|B(x,R)|^{\frac1p-\frac1q}[w(B(x,R))]^\frac1q
\sim R^{\frac{n}{p}-\frac{n-d-1}{q}}.
\end{align*}
This further implies that \eqref{eq:cond-pq} and hence the necessity hold,
which completes the proof of Proposition \ref{prop:260117-1}.
\end{proof}

We also present a growth property of $w$ and further provide a necessary and
sufficient condition for $w$ to satisfy the integral condition introduced by
Nakamura \cite[Theorem 1.4]{Nakamura16}. This integral condition plays a key
role in studying the boundedness of Hardy--Littlewood maximal operators and
Riesz transforms on weighted Morrey spaces
(see Subsection \ref{subsection3.2}).

\begin{proposition}\label{prop:260117-2}
Let $w$ be as in \eqref{dw}. Then
there exists a positive constant $C$ such that, for any $x\in\mathbb{R}^n$,
$r\in(0,\infty)$, and $k\in\mathbb{Z}_+$,
$w(B(x,r))\leq C2^{-k(d+1)}w(B(x,2^kr))$. In particular, let
$1<q\leq p<\infty$. Then, for any ball $B\subset\mathbb{R}^n$,
\begin{align}\label{20260511.1639}
\int_1^\infty\frac{1}{|sB|^{\frac1p-\frac1q}[w(sB)]^\frac1q}\frac{ds}{s}
\lesssim\frac{1}{|B|^{\frac1p-\frac1q}[w(B)]^\frac1q}
\end{align}
holds with the implicit positive constant independent of $B$ if and only if
\begin{align}\label{20260623.2235}
\frac{n}{p}-\frac{n-d-1}{q}>0.
\end{align}
\end{proposition}

\begin{proof}
Let $x\in\mathbb{R}^n$, $r\in(0,\infty)$, and $k\in\mathbb{Z}_+$. From
\cite[(2.8)]{DFM21}, we infer that, for any $k\in\mathbb{Z}_+$,
\begin{align}\label{20260511.1650}
w(B(x,r))\lesssim2^{-k(d+1)}w\left(B\left(x,2^kr\right)\right).
\end{align}

Now, let $1<q\leq p<\infty$ satisfy \eqref{20260623.2235}. We show
that \eqref{20260511.1639} holds. Indeed, from \eqref{20260511.1650}
and \eqref{20260623.2235}, we deduce that, for any ball $B\subset\mathbb{R}^n$,
\begin{align*}
\int_1^\infty\frac{1}{|sB|^{\frac1p-\frac1q}[w(sB)]^\frac1q}\frac{ds}{s}
&=\sum_{k\in\mathbb{Z}_+}\int_{2^k}^{2^{k+1}}
\frac{1}{|sB|^{\frac1p-\frac1q}[w(sB)]^\frac1q}\frac{ds}{s}\\
&\leq\sum_{k\in\mathbb{Z}_+}\frac{1}
{|2^{k+1}B|^{\frac1p-\frac1q}[w(2^kB)]^\frac1q}\\
&=\frac{1}{|B|^{\frac1p-\frac1q}}\sum_{k\in\mathbb{Z}_+}
2^{(k+1)(\frac{n}{q}-\frac{n}{p})}\frac{1}{[w(2^kB)]^\frac1q}\\
&\lesssim\frac{1}{|B|^{\frac1p-\frac1q}[w(B)]^\frac1q}\sum_{k\in\mathbb{Z}_+}
2^{-k(\frac{n}{p}-\frac{n-d-1}{q})}
\sim\frac{1}{|B|^{\frac1p-\frac1q}[w(B)]^\frac1q}
\end{align*}
and hence \eqref{20260511.1639} holds.

Conversely, let $1<q\leq p<\infty$ satisfy $\frac{n}{p}-\frac{n-d-1}{q}\leq0$.
Fix $x\in\Gamma$ and $r\in(0,\infty)$.
Using Lemma \ref{lem:250104-11}(ii), we obtain
\begin{align*}
\int_1^\infty\frac{1}{|B(x,sr)|^{\frac1p-\frac1q}[w(B(x,sr))]
^\frac1q}\frac{ds}{s}\sim r^{-(\frac{n}{p}-\frac{n-d-1}{q})}\int_1^\infty
s^{-(\frac{n}{p}-\frac{n-d-1}{q})-1}\,ds=\infty.
\end{align*}
Therefore, if \eqref{20260511.1639} holds for any ball $B\subset\mathbb{R}^n$,
then \eqref{20260623.2235} holds.
This finishes the proof of Proposition \ref{prop:260117-2}.
\end{proof}

\subsection{Poincar\'e Inequalities and Trace Operators}
\label{subsection2.2}

In this subsection, we establish the Poincar\'e inequality for suitable
functions and derive the existence and differentiation properties of the
boundary trace for them. Before stating the Poincar\'e inequality, we make the
following observation, which indicates that a random line segment in
$\mathbb{R}^n$ intersects $\Gamma$ with probability zero. In what follows, for
any $x,y\in\mathbb{R}^n$, denote by $[x,y]$ the closed \emph{segment}
joining $x$ and $y$.

\begin{lemma}\label{lem:2.10}
Let $n\ge2$, $d\in(0,n-1)$, and $\Gamma\subset\mathbb{R}^n$ be a $d$-set. Let
\begin{align*}
G:=\left\{(x,y) \in {\mathbb R}^n \times {\mathbb R}^n: [x,y]
\cap \Gamma \ne \emptyset\right\}.
\end{align*}
Then ${\mathcal H}^{2n}(G)=0$.
\end{lemma}

\begin{proof}
Since $\mathcal{H}^n(\Gamma)=0$, it suffices to show that
$\mathcal{H}^{2n}(F)=0$, where
\begin{align*}
F:=\left\{(x,y)\in\mathbb{R}^n\times\mathbb{R}^n:x,y\in\mathbb{R}^n
\setminus\Gamma,\ x \neq y,\mbox{\ and\ }[x,y]\cap\Gamma\neq\emptyset\right\}.
\end{align*}
To do this, let us construct a surjection $\sigma$ as follows:
\begin{align*}
\sigma:\left\{\begin{array}{rll}
D(\sigma)&\longrightarrow F,\\
(z,t,s,\rho)&\longmapsto(z-t\rho,z+s\rho),
\end{array}\right.
\end{align*}
where $D(\sigma):=\Gamma\times(0,\infty)\times(0,\infty)\times\mathbb{S}^{n-1}$
and $\mathbb{S}^{n-1}$ denotes the unit sphere in $\mathbb{R}^n$.
Then $\sigma$ is locally Lipschitz. In addition, there exists an increasing
sequence of compact subsets $\{D_k\}_{k\in\mathbb{N}}$ in $F$ such that
$D(\sigma)=\bigcup_{k\in\mathbb{N}}D_k$. Combining these, \cite[Proposition
2.3]{F2014} with $\alpha:=1$ therein, the countable stability of the Hausdorff
dimension $\dim_{\mathcal{H}}$, and \cite[Corollary 7.4]{F2014}, we conclude that
\begin{align*}
\dim_{\mathcal{H}}(F)&=\dim_{\mathcal{H}}\left(\sigma
\left(\bigcup_{k\in\mathbb{N}}D_k\right)\right)
=\dim_{\mathcal{H}}\left(\bigcup_{k\in\mathbb{N}}\sigma(D_k)\right)\\
&=\sup_{k\in\mathbb{N}}\dim_{\mathcal{H}}(\sigma(D_k))
\leq\sup_{k\in\mathbb{N}}\dim_{\mathcal{H}}(D_k)\\
&=\dim_{\mathcal{H}}(D(\sigma))
=d+n+1<2n.
\end{align*}
Thus, $\mathcal{H}^{2n}(F)=0$. This finishes the proof of Lemma \ref{lem:2.10}.
\end{proof}

Now, we present the Poincar\'e inequality, which is exactly
\cite[(3.7) and (3.8)]{DFM21}. Applying Lemma \ref{lem:2.10}, we provide
an alternative proof and extend its validity from the weighted Sobolev space
$\dot{W}^{1,2}(\Omega,w)$ in Definition \ref{def5.1} below to the \emph{local
Sobolev space $$W^{1,1}_{\mathrm{loc}}(\Omega):=
\left\{f\in L^1_{\mathrm{loc}}(\Omega):\nabla f\in
L^1_{\mathrm{loc}}(\Omega)\right\}.$$}
In what follows, $r\to0^+$ means that $r\in(0,\infty)$ and $r\to0$.

\begin{lemma}\label{DFM21}
Let $B:=B(x,r)\subset\mathbb{R}^n$ be a ball and $f\in
W^{1,1}_{\mathrm{loc}}(\Omega)$.
\begin{enumerate}
\item[{\rm(i)}] It holds that
\begin{equation}\label{eq:250906-12}
\fint_B\fint_B |f(y)-f(z)|\,dydz\lesssim r\fint_B |\nabla f(\xi)|\,d\xi
\end{equation}
with the implicit positive constant independent of $f$ and $B$.
\item[{\rm(ii)}]For almost every $z\in B$,
\begin{equation}\label{eq:250906-11}
\fint_B |f(y)-f(z)|\,dy\lesssim\int_B \frac{|\nabla f(\xi)|}{|z-\xi|^{n-1}}\,d\xi
\end{equation}
with the implicit positive constant independent of $f$, $B$, and $z$.
\end{enumerate}
\end{lemma}

\begin{proof}
We first prove (i). Fix $\epsilon\in(0,1)$ and define
\begin{align*}
E_\epsilon:=\{(y,z)\in B(x,(1-\epsilon)r)\times B(x,(1-\epsilon)r):
\operatorname{dist}([y,z],\Gamma)>\epsilon\}\subset\mathbb{R}^{2n}.
\end{align*}
From Lemma \ref{lem:2.10}, we infer that
\begin{align}\label{20260414.1427}
\mathcal{H}^{2n}\left(B(x,r)\times B(x,r)\setminus
\bigcup_{\epsilon\in(0,1)}E_\epsilon\right)
=\mathcal{H}^{2n}(\{(y,z)\in B(x,r)\times B(x,r):
[y,z]\cap\Gamma\neq\emptyset\})=0.
\end{align}
Let $\eta\in C^\infty_{\mathrm{c}}(B(\bf{0},\epsilon))$ satisfy
$0\leq\eta\leq1$ and $\|\eta\|_{L^1(\mathbb{R}^n)}=1$.
For any given $\delta\in(0,\epsilon)$, define
$\eta_\delta(\cdot):=\delta^{-n}\eta(\frac{\cdot}{\delta})$ and
$f_\delta:=f*\eta_\delta$. Note that, for any $y\in F_\delta:=\{x\in\Omega:
\operatorname{dist}(x,\Gamma)>\delta\}$ and $z\in B(\mathbf{0},\delta\epsilon)$,
$\operatorname{dist}(y-z,\Gamma)\geq\operatorname{dist}(y,\Gamma)-|z|>\delta(1-\epsilon)>0$. Hence
$f_\delta$ is well defined on $F_\delta$. Moreover, $f_\delta\in
C^\infty(F_\delta)$.

For any $(y,z)\in E_\epsilon$, the line segment $[y,z]\subset F_\varepsilon$.
Therefore, $f_\delta$ satisfies the fundamental theorem of calculus:
\begin{align*}
f_\delta(y)-f_\delta(z)=\int_0^1(y-z)\cdot\nabla f_\delta(z+t(y-z))\,dt.
\end{align*}
This, together with the Cauchy--Schwarz inequality, further implies that
\begin{align}\label{20260413.2250}
\iint_{E_\epsilon}|f_\delta(y)-f_\delta(z)|\,dy\,dz
\leq\iint_{E_\epsilon}\int_0^1|y-z|\,|\nabla f_\delta(z+t(y-z))|\,dt\,dy\,dz.
\end{align}

For any $z\in\mathbb{R}^n$, let
\begin{align*}
G(z):=\begin{cases}
\nabla f(z) &\mathrm{if}\ z\in\Omega,\\
0 &\mathrm{if}\ z\in\Gamma
\end{cases}
\mbox{\ \ and\ \ }G_\delta:=G*\eta_\delta.
\end{align*}
Then, for any $(y,z)\in E_\epsilon$, $\nabla f_\delta=\nabla
f*\eta_\delta=G_\delta$ on $[y,z]$. By this and \eqref{20260413.2250}, we obtain
\begin{align*}
\iint_{E_\epsilon}|f_\delta(y)-f_\delta(z)|\,dy\,dz
\leq\iint_{E_\epsilon}\int_0^1|y-z|\,|G_\delta(z+t(y-z))|\,dt\,dy\,dz.
\end{align*}
Letting $\delta\to0^+$ and using Fatou's lemma and the Lebesgue
dominated convergence theorem, we conclude that
\begin{align}\label{eq:250906-12x}
\iint_{E_\epsilon}|f(y)-f(z)|\,dy\,dz
\leq\iint_{B(x,(1-\epsilon)r)\times B(x,(1-\epsilon)r)}
\int_0^1 |y-z|\,|G(z+t(y-z))|\,dt\,dy\,dz.
\end{align}

Now, we estimate the right-hand side of \eqref{eq:250906-12x}. For any
$t\in[0,1]$, $y\in\mathbb{R}^n$, and $z\in B(x,(1-\epsilon)r)$, let
$\xi:=z+t(y-z)$. Then $y\in B(x,(1-\epsilon)r)$ if and only if
$\xi\in B(z+t(x-z),t(1-\epsilon)r)$. From this change of variable and Tonelli's
theorem, we deduce that
\begin{align*}
&\iint_{B(x,(1-\epsilon)r)\times B(x,(1-\epsilon)r)}
\int_0^1 |y-z|\,|G(z+t(y-z))|\,dt\,dy\,dz\\
&\quad=\int_{B(x,(1-\epsilon)r)}\int_0^1
\int_{B(z+t(x-z),t(1-\epsilon)r)}\frac{|\xi-z|}{t^{n+1}}|G(\xi)|\,d\xi\,dt\,dz\\
&\quad\leq\int_{B(x,(1-\epsilon)r)}\int_{B(x,(1-\epsilon)r)}
\int_{\frac{|\xi-z|}{2(1-\epsilon)r}}^1
\frac{|\xi-z|}{t^{n+1}}|G(\xi)|\,dt\,d\xi\,dz\\
&\quad\lesssim(1-\epsilon)^nr^n\int_{B(x,(1-\epsilon)r)}\int_{B(x,(1-\epsilon)r)}
\frac{|G(\xi)|}{|z-\xi|^{n-1}}\,dz\,d\xi\\
&\quad\leq(1-\epsilon)^nr^n\int_{B(x,(1-\epsilon)r)}|G(\xi)|
\int_{B(\xi,2(1-\epsilon)r)}\frac{1}{|z-\xi|^{n-1}}\,dz\,d\xi\\
&\quad\lesssim(1-\epsilon)^{n+1}r^{n+1}
\int_{B(x,(1-\epsilon)r)}|\nabla f(\xi)|\,d\xi,
\end{align*}
which further implies that
\begin{align*}
&\fint_{B(x,(1-\epsilon)r)}\fint_{B(x,(1-\epsilon)r)}
\int_0^1 |y-z|\,|G(z+t(y-z))|\,dt\,dy\,dz\\
&\quad\lesssim(1-\epsilon)r\fint_{B(x,(1-\epsilon)r)}|\nabla f(\xi)|\,d\xi.
\end{align*}
Combining this, \eqref{eq:250906-12x}, and \eqref{20260414.1427}, and letting
$\epsilon\to0^+$, we find that \eqref{eq:250906-12} holds,
thereby completing the proof of (i).

Next, we show (ii). Let $z\in B$ be a Lebesgue point of $f$.
For any $j\in\mathbb{Z}_+$, define
\begin{align*}
B_j:=B\left(\left(1-2^{-j}\right)z+2^{-j}x,2^{-j}r\right).
\end{align*}
Then $B_{j+1}\subset B_j$ with equivalent measure and
$\bigcap_{j=0}^\infty B_j=\{z\}$. Using \eqref{eq:250906-12} on $B_j$,
we conclude that
\begin{align}\label{eq:250913-1}
\left|f_{B_j}-f_{B_{j+1}}\right|&\leq
\fint_{B_{j+1}}\fint_{B_j}|f(y)-f(v)|\,dy\,dv\notag\\
&\sim\fint_{B_j}\fint_{B_j}|f(y)-f(v)|\,dy\,dv
\lesssim2^{-j}r\fint_{B_j}|\nabla f(\xi)|\,d\xi.
\end{align}
We claim that
\begin{align}\label{eq:250913-2}
\sum_{j\in\mathbb{Z}_+}\left(2^{-j}r\right)^{1-n}\mathbf{1}_{B_j}(\xi)
\lesssim\frac{\boldsymbol{1}_B(\xi)}{|z-\xi|^{n-1}}.
\end{align}
Indeed, for any $\xi\in B$, let $J:=\max\{j\in\mathbb{N}:\xi\in B_j\}$. Then
$|z-\xi|\sim2^{-J}r$ and hence
\begin{align*}
\sum_{j\in\mathbb{Z}_+}\left(2^{-j}r\right)^{1-n}\mathbf{1}_{B_j}(\xi)
=\sum_{j=0}^J\left(2^{-j}r\right)^{1-n}\lesssim2^{(n-1)J}r^{1-n}\sim
2^{-j}r\fint_{B_j}|\nabla f(\xi)|\,d\xi.
\end{align*}
Thus, the above claim holds. In addition, since $z$ is a Lebesgue point of $f$,
it follows that $f(z)=\lim_{j\to\infty}f_{B_j}$. Combining this,
\eqref{eq:250906-12} on $B$, \eqref{eq:250913-1}, and \eqref{eq:250913-2},
we find that
\begin{align*}
\fint_B|f(y)-f(z)|\,dy&\leq\fint_{B_0}|f(y)-f_{B_1}|\,dy+|f_{B_1}-f(z)|\\
&\lesssim\fint_B\fint_B|f(y)-f(v)|\,dy\,dv
+\sum_{j\in\mathbb{N}}\left|f_{B_j}-f_{B_{j+1}}\right|\\
&\lesssim r^{1-n}\int_B|\nabla f(\xi)|\,d\xi
+\int_B\sum_{j\in\mathbb{Z}_+}\left(2^{-j}r\right)^{1-n}\mathbf{1}_{B_j}(\xi)
|\nabla f(\xi)|\,d\xi\\
&\lesssim\int_B \frac{|\nabla f(\xi)|}{|z-\xi|^{n-1}}\,d\xi,
\end{align*}
which further implies that \eqref{eq:250906-11} holds. This finishes the proof
of (ii) and hence Lemma \ref{DFM21}.
\end{proof}

Next, using Lemma \ref{DFM21}, we establish the existence and differentiation
properties of boundary trace for functions in $W^{1,1}_\mathrm{loc}(\Omega)$.

\begin{proposition}\label{prop:260126-1}
Let $f\in W^{1,1}_\mathrm{loc}(\Omega)$.
\begin{enumerate}
\item[\rm(i)] If $x\in\Gamma$ satisfies
\begin{align}\label{eq:250928-1}
\int_{B(x,r)}\frac{|\nabla f(\xi)|}{|x-\xi|^{n-1}}\,d\xi<\infty
\end{align}
for some $r\in(0,\infty)$, then the limit
\begin{align*}
Tf(x):=\lim_{r\to0^+}\fint_{B(x,r)}f(y)\,dy
\end{align*}
exists and is finite.

\item[\rm(ii)] If $\nabla f\in L^1_\mathrm{loc}(\Omega,w)$, then
$Tf(x)$ exists for $\mathcal{H}^d$-almost every $x\in\Gamma$.

\item[\rm(iii)] If $x\in\Gamma$ satisfies \eqref{eq:250928-1},
then the Lebesgue differentiation property
\begin{align*}
\lim_{r\to0^+}\fint_{B(x,r)}|Tf(x)-f(y)|\,dy=0
\end{align*}
holds.
\end{enumerate}
\end{proposition}

We point out that, under the same notation as in Proposition
\ref{prop:260126-1}, the mapping $f\longmapsto Tf$ is called the
\emph{trace operator}.

\begin{proof}[Proof of Proposition \ref{prop:260126-1}]
We first prove (i). Let $x\in\Gamma$ satisfy \eqref{eq:250928-1}
for some $r\in(0,\infty)$. Using Lemma \ref{DFM21}(i), \eqref{eq:250913-2}
with $B_j$ therein replaced by $B(x,2^{1-j}r)$, and \eqref{eq:250928-1},
we find that
\begin{align*}
\sum_{j\in\mathbb{N}}\left|f_{B(x,2^{-j}r)}-f_{B(x,2^{1-j}r)}\right|
&\lesssim\sum_{j\in\mathbb{N}}\fint_{B(x,2^{1-j}r)}\fint_{B(x,2^{1-j}r)}
|f(y)-f(z)|\,dy\,dz\\
&\lesssim\int_{B(x,r)}\sum_{j\in\mathbb{N}}\left(2^{1-j}r\right)^{1-n}
\boldsymbol{1}_{B(x,2^{1-j}r)}(\xi)|\nabla f(\xi)|\,d\xi\\
&\lesssim\int_{B(x,r)}\frac{|\nabla f(\xi)|}{|x-\xi|^{n-1}}\,d\xi<\infty,
\end{align*}
which further implies that the sequence $\{f_{B(x,2^{-j}r)}\}_{j\in\mathbb{N}}$
is a Cauchy sequence. This guarantees the existence $Tf(x)$ and finishes the
proof of (i).

Now, we show (ii). We claim that, for any given $a\in(0,n-d)$ and for any
$x\in\Gamma$, $r\in(0,\infty)$, and $y\in B(x,2r)$,
\begin{align}\label{20260511.1829}
\int_{\Gamma(x,r)}\frac{d\mathcal{H}^d(z)}{|z-y|^{n-a}}
\lesssim[\delta(y)]^{d+a-n},
\end{align}
where the implicit positive constant is independent of $x$, $y$, and $r$.
Indeed, let $a\in(0,n-d)$, $x\in\Gamma$, $r\in(0,\infty)$, and $y\in B(x,2r)$.
Since $\delta(y)\leq|y-x|<2r$, it follows that there exists
$\xi_y\in\Gamma(x,4r)$ satisfying $\delta(y) = |y-\xi_y|$.
In addition, by the fact that $\Gamma$ is a $d$-set, we find that,
for any $k\in\mathbb{N}$, $\mathcal{H}^d(\Gamma(y,2^k\delta(y)))\sim
[2^k\delta(y)]^d$. This, together with the observation that
$\Gamma(x,r)\subset\Gamma(\xi_y,5r)$, further implies that
\begin{align*}
\int_{\Gamma(x,r)}\frac{d\mathcal{H}^d(z)}{|z-y|^{n-a}}
&\leq\int_{\Gamma(\xi_y,5r)}\frac{d\mathcal{H}^d(z)}{|z-y|^{n-a}}
\leq\sum_{k\in\mathbb{N}}\int_{\Gamma(y,2^k\delta(y)\setminus\Gamma(y,2^{k-1}
\delta(y))}\frac{d\mathcal{H}^d(z)}{|z-y|^{n-a}}\\
&\lesssim\sum_{k\in\mathbb{N}}
\frac{\mathcal{H}^d(\Gamma(y,2^k\delta(y)))}{[2^k\delta(y)]^{n-a}}
\lesssim\sum_{k\in\mathbb{N}}\left[2^k\delta(y)\right]^{d+a-n}
\sim[\delta(y)]^{d-n+a},
\end{align*}
and hence \eqref{20260511.1829}, i.e. the above claim, holds.

From Tonelli's theorem and the above claim, we infer that,
for any $x\in\Gamma$ and $r\in(0,\infty)$,
\begin{align*}
\int_{\Gamma(x,r)}\left[\int_{B(x,r)}\frac{|\nabla f(\xi)|}{|z-\xi|^{n-1}}
\,d\xi\right]\,d\mathcal{H}^{d}(z)
&=\int_{B(x,r)}\left[\int_{\Gamma(x,r)}
\frac{|\nabla f(\xi)|}{|z-\xi|^{n-1}}\,d\mathcal{H}^{d}(z)\right]d\xi\\
&\lesssim\int_{B(x,r)}|\nabla f(\xi)|w(\xi)\,d\xi<\infty.
\end{align*}
Consequently, for $\mathcal{H}^d$-almost every $x\in\Gamma$
and any $r\in(0,\infty)$,
\begin{align*}
\int_{B(x,r)}\frac{|\nabla f(\xi)|}{|x-\xi|^{n-1}}\,d\xi<\infty,
\end{align*}
which, together with (i), implies that (ii) holds.

Finally, we prove (iii). Let $x\in\Gamma$ satisfy \eqref{eq:250928-1}.
Applying Lemma \ref{DFM21}(i), we find that, for any $r\in(0,\infty)$,
\begin{align*}
\fint_{B(x,r)}|Tf(x)-f(y)|\,dy&\leq\left|Tf(x)-f_{B(x,r)}\right|+
\fint_{B(x,r)}\left|f_{B(x,r)}-f(y)\right|\,dy\\
&\leq\left|Tf(x)-f_{B(x,r)}\right|+
\fint_{B(x,r)}\fint_{B(x,r)}|f(y)-f(z)|\,dy\,dz\\
&\lesssim\left|Tf(x)-f_{B(x,r)}\right|+r\fint_{B(x,r)}|\nabla f(y)|\,dy\\
&\lesssim\left|Tf(x)-f_{B(x,r)}\right|+
\int_{B(x,r)}\frac{|\nabla f(y)|}{|x-y|^{n-1}}\,dy.
\end{align*}
Letting $r\to0^+$ and using \eqref{eq:250928-1}, (i), and the absolute
continuity of the integral, we conclude that (iii) holds. This finishes the
proof of Proposition \ref{prop:260126-1}.
\end{proof}

Let $p\in[1,\infty)$ and $w$ be as in \eqref{dw}.
The \emph{local weighted Sobolev space
$\dot{W}^{1,p}_\mathrm{loc}(\Omega,w)$} is defined by setting
\begin{align*}
\dot{W}^{1,p}_\mathrm{loc}(\Omega,w):=
\left\{f\in L^1_{\mathrm{loc}}(\Omega):\nabla f
\in L^p_\mathrm{loc}(\Omega,w)\right\}.
\end{align*}
By Proposition \ref{prop:260126-1} and an argument similar to that
used in the proof of \cite[Lemma 4.1]{DFM21}, we obtain
the following Poincar\'e inequality adapted to the boundary for functions in
$\dot{W}^{1,1}_\mathrm{loc}(\Omega,w)$; we omit the details.

\begin{proposition}\label{prop:Poincare}
Let $w$ be as in \eqref{dw} and
$f\in\dot{W}^{1,1}_\mathrm{loc}(\Omega,w)$. Then, for any $x\in\Gamma$ and
$r\in(0,\infty)$ such that $Tf=0$ on $\Gamma(x,r)$,
\begin{align*}
\fint_{B(x,r)}|f(y)|\,dy\lesssim\frac{1}{r^d}
\int_{B(x,r)}|\nabla f(y)|w(y)\,dy,
\end{align*}
where the implicit positive constant is independent of $f$, $x$, and $r$.
\end{proposition}

Finally, as a consequence of Proposition \ref{prop:Poincare} and an argument
similar to that used in the proof of \cite[Lemma 4.2]{DFM21}, we conclude the
following weighted Poincar\'e inequality; we omit the details.

\begin{proposition}\label{prop:Poincare2}
Let $p\in(1,n)$, $p^*:=\frac{np}{n-p}$, $q\in[1,p^*]$, and
$w$ be as in \eqref{dw}. Then, for any
$f\in\dot{W}^{1,\max\{p,\,q\}}_\mathrm{loc}(\Omega,w)$, $x\in\mathbb{R}^n$, and
$r\in(0,\infty)$,
\begin{align*}
&\left[\frac{1}{w(B(x,r))}\int_{B(x,r)}\left|f(y)-
f_{B(x,r)}\right|^qw(y)\,dy\right]^\frac1q\\
&\quad\sim\left[\frac{1}{w(B(x,r))}\int_{B(x,r)}\left|f(y)-\frac{1}{w(B(x,r))}
\int_{B(x,r)}f(z)w(z)\,dz\right|^qw(y)\,dy\right]^\frac1q\\
&\quad\lesssim r\left[\frac{1}{w(B(x,r))}\int_{B(x,r)}
|\nabla f(y)|^pw(y)\,dy\right]^\frac1p,
\end{align*}
where the implicit positive constants are independent of $f$, $x$, and $r$.
\end{proposition}

\section{Weighted Morrey Spaces $\mathcal{M}^p_q(\Omega,w)$}
\label{section3}

In this section, we study weighted Morrey spaces
$\mathcal{M}^p_q(\Omega,w)$ adapted to $\Gamma$ by four subsections. Subsection
\ref{subsection3.1} is devoted to their definitions, together with
their basic properties. Subsection \ref{subsection3.2} handles with the
boundedness properties of relevant operators on $\mathcal{M}^p_q(\Omega,w)$.
In Subsection \ref{subsection3.3}, we establish the Littlewood--Paley
characterization of $\mathcal{M}^p_q(\Omega,w)$, which is used to provide
the potential characterization of the weighted Sobolev--Morrey spaces adapted
to $\Gamma$ in Section \ref{section4}. Finally, in Subsection
\ref{subsection3.4}, we present the complex interpolation formulae for these
Morrey spaces.

\subsection{Weighted Morrey Spaces Adapted to $\Gamma$}
\label{subsection3.1}

In this subsection, we introduce the concept and some basic properties of
weighted Morrey spaces adapted to $\Gamma$. Let $1\leq q\leq p<\infty$
and $w$ be as in \eqref{dw}. The
\emph{weighted Morrey space $\mathcal{M}^p_q(\Omega,w)$} adapted to $\Gamma$ is
defined to be the set of all measurable functions
$f$ on $\mathbb{R}^n$ such that
\begin{align}\label{20260604.2231}
\|f\|_{\mathcal{M}^p_q(\Omega,w)}:= \sup_{x\in\mathbb{R}^n,\,r\in(0,\infty)}
|B(x,r)|^{\frac1p-\frac1q}\left[\int_{\Omega(x,r)}|f(y)|^qw(y)\,dy\right]
^\frac{1}{q}<\infty,
\end{align}
where, for any $x\in\mathbb{R}^n$ and $r\in(0,\infty)$,
$\Omega(x,r):=\Omega\cap B(x,r)$. This space can be seen
as a natural extension of the weighted Lebesgue space $L^p(\Omega,w)$.
In particular, for any $p\in[1,\infty)$,
\begin{align*}
\mathcal{M}^p_p(\Omega,w)=L^p(\Omega,w)=L^p(\mathbb{R}^n,w).
\end{align*}

By the definition of $\|\cdot\|_{\mathcal{M}^p_q(\Omega,w)}$ and Lemma
\ref{lem:250104-11}(iii), we find that the following continuous embedding
holds; we omit the details.

\begin{proposition}\label{lem:250919-2}
Let $1\leq q\leq p<\infty$ and $w$ be as in \eqref{dw}. Then
\begin{align*}
\mathcal{M}^p_q(\Omega,w)\hookrightarrow L^1_{\mathrm{loc}}(\Omega,w)
\hookrightarrow L^1_{\mathrm{loc}}(\Omega).
\end{align*}
\end{proposition}

Next, we discuss conditions under which the \emph{space
$L^\infty_{\mathrm{c}}(\mathbb{R}^n)$} of bounded functions with compact
support is embedded into $\mathcal{M}^p_q(\Omega,w)$,
where a new restriction appears.
\begin{proposition}\label{prop:250926-1}
Let $1\leq q\leq p<\infty$ and $w$ be as in \eqref{dw}.
Then the following assertions hold.
\begin{enumerate}
\item[\rm(i)] For any $x\in\mathbb{R}^n$ and $r\in(0,\infty)$,
$\boldsymbol{1}_{B(x,r)}\in\mathcal{M}^p_q(\Omega,w)$ if and only if
\eqref{eq:cond-pq} holds. In this case,
$\|\boldsymbol{1}_{B(x,r)}\|_{\mathcal{M}^p_q(\Omega,w)}\sim
r^{\frac{n}{p}-\frac{n-d-1}{q}}$, where the positive equivalence constants are
independent of both $x$ and $r$.

\item[\rm(ii)] The constant function $1\in\mathcal{M}^p_q(\Omega,w)$ if and
only if $\frac{n}{p}=\frac{n-d-1}{q}$.
\end{enumerate}
\end{proposition}

\begin{proof}
We only give the proof of (i) because the proof of (ii) is similar.
We first prove the necessity. Let $x\in\Gamma$ and $r\in(0,\infty)$.
Using Lemma \ref{lem:250104-11}(ii), we find that, for any $s\in(0,r]$,
\begin{align*}
s^{\frac{n}{p}-\frac{n-d-1}{q}}\sim
|B(x,s)|^{\frac{1}{p}-\frac{1}{q}}[w(B(x,s))]^{\frac{1}{q}}
\leq\|\boldsymbol{1}_{B(x,r)}\|_{\mathcal{M}^p_q(\Omega,w)}<\infty,
\end{align*}
which further implies that \eqref{eq:cond-pq} holds.

Now, we show the sufficiency. Let $x,y\in\mathbb{R}^n$ and
$r,s\in(0,\infty)$ such that $B(x,r)\cap B(y,s)\neq\emptyset$.
We prove $\boldsymbol{1}_{B(x,r)}\in\mathcal{M}^p_q(\Omega,w)$ by considering
the following four cases for $B(y,s)$.

\emph{Case (1)} $\delta(x)\leq2r$, $B(y,2s)\cap\Gamma\neq\emptyset$,
and $r\leq s$. In this case, by Lemma \ref{lem:250104-11}(ii), we conclude that
\begin{align*}
|B(y,s)|^{\frac{1}{p}-\frac{1}{q}}[w(B(x,r)\cap B(y,s))]^\frac{1}{q}
&\leq|B(y,s)|^{\frac1p-\frac1q}[w(B(x,r))]^{\frac1q}\\
&\leq|B(y,r)|^{\frac1p-\frac1q}[w(B(x,r))]^{\frac1q}\\
&\sim r^{\frac np-\frac{n-d-1}q}.
\end{align*}

\emph{Case (2)} $\delta(x)>2r$, $B(y,2s)\cap\Gamma\neq\emptyset$,
and $r\leq s$. In this case, from Lemma \ref{lem:250104-11}(i), it follows that
\begin{align*}
|B(y,s)|^{\frac{1}{p}-\frac{1}{q}}[w(B(x,r)\cap B(y,s))]^\frac{1}{q}
&\leq|B(y,s)|^{\frac1p-\frac1q}[w(B(x,r))]^{\frac1q}\\
&\leq|B(y,r)|^{\frac1p-\frac1q}[w(B(x,r))]^{\frac1q}\\
&\sim r^\frac{n}{p}[w(x)]^\frac1q\lesssim r^{\frac np-\frac{n-d-1}q}.
\end{align*}

\emph{Case (3)} $B(y,2s)\cap\Gamma\neq\emptyset$ and $r\geq s$. In this case,
using Lemma \ref{lem:250104-11}(ii) and \eqref{eq:cond-pq}, we find that
\begin{align*}
|B(y,s)|^{\frac1p-\frac1q}[w(B(x,r)\cap B(y,s))]^{\frac1q}
&\le|B(y,s)|^{\frac1p-\frac1q}[w(B(y,s))]^{\frac1q}\\
&\sim s^{\frac np-\frac{n-d-1}q}\le r^{\frac np-\frac{n-d-1}q}.
\end{align*}

\emph{Case (4)} $B(y,2s)\cap\Gamma=\emptyset$. In this case, from the
assumptions that $B(y,2s)\cap\Gamma=\emptyset$ and $B(x,r)\cap B(y,s)\neq\emptyset$,
we deduce that $r\geq s$. Indeed, if $r<s$, then $|x-y|\leq|x-z|+|z-y|<r+s<2s$,
where $z\in B(x,r)\cap B(y,s)$. This further implies that $x\in
B(y,2s)\cap\Gamma$, in contradiction to $B(y,2s)\cap\Gamma=\emptyset$.
Thus, $r\geq s$. Moreover, both the assumption that $\delta(y)\geq 2s$
and Lemma \ref{lem:250104-11}(i) yield
\begin{align*}
w(B(y,s))\sim s^nw(y)\lesssim s^{d+1}\sim w(B(x,s)),
\end{align*}
which, together with Lemma \ref{lem:250104-11}(ii), \eqref{eq:cond-pq}, and
$r\geq s$, further implies that
\begin{align*}
|B(y,s)|^{\frac1p-\frac1q}[w\left(B(x,r)\cap B(y,s)\right)]^{\frac1q}&
\lesssim|B(y,s)|^{\frac1p-\frac1q}[w\left(B(x,s)\right)]^{\frac1q}\\
&\lesssim s^{\frac np-\frac{n-d-1}q}\lesssim
r^{\frac np-\frac{n-d-1}q}.
\end{align*}

In conclusion, we obtain
\begin{align*}
\|\boldsymbol{1}_{B(x,r)}\|_{\mathcal{M}^p_q(\Omega,w)}
\lesssim r^{\frac np-\frac{n-d-1}q}.
\end{align*}
This finishes the proof of the sufficiency and hence
Proposition \ref{prop:250926-1}.
\end{proof}

Let $\omega$ be a weight on $\mathbb{R}^n$; i.e., $\omega(x)\in(0,\infty)$ for
almost every $x\in\mathbb{R}^n$. Recall that the \emph{weighted
Hardy--Littlewood maximal operator $M_\omega$} is defined by setting,
for any $f\in L^1_\mathrm{loc}(\mathbb{R}^n,\omega)$ and $x\in\mathbb{R}^n$,
\begin{equation*}
M_\omega(f)(x):=\sup_{r\in(0,\infty)}\frac{1}{\omega(B(x,r))}
\int_{B(x,r)}|f(y)|\omega(y)\,dy.
\end{equation*}
In particular, when $\omega\equiv1$, $M_\omega$ reduces to the classical
Hardy--Littlewood maximal operator (see, for example, \cite{Stein93}) and,
in this case, we simply denote it by $M$. Applying Lemma
\ref{lem:250104-11}, we obtain the following equivalent norm of Morrey spaces
${\mathcal M}^p_q(\Omega,w)$, which is of independent interest.
In contrast to the classical case, the higher
co-dimension assumption $d\in(0,n-1)$ leads to the additional restriction
\eqref{eq:260121-1}.

\begin{theorem}\label{prop:250104-11}
Let $1\leq q\leq p<\infty$, $\theta\in(0,1)$, and $w$ be as in \eqref{dw}.
Then, for any $f\in\mathcal{M}^p_q(\Omega,w)$,
\begin{align}\label{20260629.2159}
\|f\|_{\mathcal{M}^p_q(\Omega,w)}\sim
\|f\|_{\widetilde{\mathcal{M}}^p_q(\Omega,w)}:=
\sup_{\genfrac{}{}{0pt}{}{x\in\Omega}{r\in(0,\infty)}}
r^{\frac{n}{p}-\frac{n}{q}}\left\{\int_\Omega|f(y)|^q
\left[M_w(\boldsymbol{1}_{\Omega(x,r)})(y)\right]
^\theta w(y)\,dy\right\}^\frac1q
\end{align}
holds with the positive equivalence constants independent of $f$ if and only if
\begin{align}\label{eq:260121-1}
\frac{n}{p}-\frac{n}{q}+\frac{(d+1)\theta}{q}>0.
\end{align}
\end{theorem}

\begin{proof}
We first prove the sufficiency. Assume that \eqref{eq:260121-1} holds and let
$f\in\mathcal{M}^p_q(\Omega,w)$. Since, for any $x\in\mathbb{R}^n$ and
$r\in(0,\infty)$, $M_w(\boldsymbol{1}_{\Omega(x,r)})\geq
\boldsymbol{1}_{\Omega(x,r)}$, it follows that
$\|f\|_{\mathcal{M}^p_q(\Omega,w)}\lesssim
\|f\|_{\widetilde{\mathcal{M}}^p_q(\Omega,w)}$.

To prove the reverse inequality of \eqref{20260629.2159}, we let
$x\in\Omega$ and $r\in(0,\infty)$. Note that
$M_w(\boldsymbol{1}_{\Omega(x,r)})\leq1$. Thus,
\begin{align*}
&r^{\frac{n}{p}-\frac{n}{q}}\left\{\int_{\Omega(x,2r)}|f(y)|^q
\left[M_w(\boldsymbol{1}_{\Omega(x,r)})(y)\right]^\theta
w(y)\,dy\right\}^\frac1q \\
&\quad\leq r^{\frac{n}{p}-\frac{n}{q}}\left[\int_{\Omega(x,2r)}
|f(y)|^q w(y)\,dy\right]^\frac1q\lesssim\|f\|_{\mathcal{M}^p_q(\Omega,w)}.
\end{align*}
It remains to show that
\begin{align}\label{eq:260104-1}
r^{\frac{n}{p}-\frac{n}{q}}\left\{\int_{\Omega\setminus\Omega(x,2r)}|f(y)|^q
\left[M_w(\boldsymbol{1}_{\Omega(x,r)})(y)\right]^\theta
w(y)\,dy\right\}^\frac1q\lesssim\|f\|_{{\mathcal M}^p_q(\Omega,w)}.
\end{align}
We claim that, for any $k\in\mathbb{N}$ and
$y\in\Omega(x,2^{k+1}r)\setminus\Omega(x,2^kr)$,
$M_w\left(\boldsymbol{1}_{\Omega(x,r)}\right)(y)\lesssim2^{-k(d+1)}$,
where the implicit positive constant is independent of $x$, $y$, and $r$.
Indeed, observe that, for any $s\in(0,\infty)$,
\begin{align*}
\begin{cases}
B(y,s)\cap B(x,r)=\emptyset &\mbox{if}\ s\in(0,|x-y|-r],\\
B(y,s)\subset B(x,2^{k+3}r) &\mbox{if}\ s\in(|x-y|-r,|x-y|+r),\\
B(x,r)\subset B(y,s) &\mbox{if}\ s\in[|x-y|+r,\infty).
\end{cases}
\end{align*}
From this, the doubling condition of $w$, and Proposition \ref{prop:260117-2},
we infer that
\begin{align*}
M_w\left(\boldsymbol{1}_{\Omega(x,r)}\right)(y)&=\sup_{s\in(0,\infty)}
\frac{w(B(y,s)\cap B(x,r)\cap\Omega)}{w(B(y,s))}\\
&=\sup_{s\in(|x-y|-r,|x-y|+r)}
\frac{w(B(y,s)\cap B(x,r)\cap\Omega)}{w(B(y,s))}\\
&\leq\sup_{s\in(|x-y|-r,|x-y|+r)}\frac{w(B(x,r))}{w(B(x,2^kr))}
\frac{w(B(x,2^{k+3}r))}{w(B(y,s))}\\
&\lesssim\frac{w(B(x,r))}{w(B(x,2^kr))}\sup_{s\in(|x-y|-r,|x-y|+r)}
\frac{|B(x,2^{k+3}r)|}{|B(y,s)|}
\sim\frac{w(B(x,r))}{w(B(x,2^kr))}\lesssim2^{-k(d+1)}.
\end{align*}
This, together with \eqref{eq:260121-1}, further implies that
\begin{align*}
&r^{\frac{n}{p}-\frac{n}{q}}\left\{\int_{\Omega\setminus\Omega(x,2r)}|f(y)|^q
\left[M_w(\boldsymbol{1}_{\Omega(x,r)})(y)\right]^\theta
w(y)\,dy\right\}^\frac1q\\
&\quad\leq\sum_{k\in\mathbb{N}}r^{\frac{n}{p}-\frac{n}{q}}
\left\{\int_{\Omega(x,2^{k+1}r)\setminus\Omega(x,2^kr)}|f(y)|^q
\left[M_w(\boldsymbol{1}_{\Omega(x,r)})(y)\right]^\theta
w(y)\,dy\right\}^\frac1q\\
&\quad\lesssim\sum_{k\in\mathbb{N}}r^{\frac{n}{p}-\frac{n}{q}}
2^{-\frac{k(d+1)\theta}{q}}\left[\int_{\Omega(x,2^{k+1}r)}|f(y)|^q
w(y)\,dy\right]^\frac1q\\
&\quad\lesssim\sum_{k\in\mathbb{N}}r^{\frac{n}{p}-\frac{n}{q}}
2^{-\frac{k(d+1)\theta}{q}}(2^{k+1}r)^{\frac{n}{q}-\frac{n}{p}}
\|f\|_{\mathcal{M}^p_q(\Omega,w)}\\
&\quad=\sum_{k\in\mathbb{N}}
2^{k[\frac{n}{q}-\frac{n}{p}-\frac{(d+1)\theta}{q}]}
\|f\|_{\mathcal{M}^p_q(\Omega,w)}\lesssim\|f\|_{\mathcal{M}^p_q(\Omega,w)}.
\end{align*}
Therefore, \eqref{eq:260104-1} and hence \eqref{20260629.2159} hold.
This finishes the proof of the sufficiency.

We then show the necessity by contradiction. Assume that $\frac{n}{p}-
\frac{n}{q}+\frac{(d+1)\theta}{q}\leq0$. Since $\Gamma$ is a $d$-set with
$d<n-1$, it follows that $\Gamma$ is \emph{porous}; i.e., there exists a
positive constant $c\in(0,1)$ such that, for any $\xi\in\Gamma$ and
$r\in(0,\infty)$, there exists $z\in B(\xi,r)$ satisfying $B(z,cr)\subset
B(\xi,r)\setminus\Gamma\subset\Omega$ (see, for instance, \cite[Proposition
1.3]{t(mn-2023)}). Fix $\xi\in\Gamma$ and $R\in(\frac{c+2}{c},\infty)$.
It is easy to prove that, for any $j\in\mathbb{N}$, there exists
$z_j\in B(\xi,R^j)\cap\Omega$ such that
\begin{align}\label{20260629.1801}
cR^j\leq\delta\left(z_j\right)\leq\left|z_j-\xi\right|<R^j.
\end{align}
For any $N\in\mathbb{N}$, define
\begin{align*}
f_N:=\sum_{j=1}^NR^{-j(\frac{n}{p}-\frac{n-d-1}{q})}
\boldsymbol{1}_{B(z_j,\frac{c}{2}R^j)}.
\end{align*}
By \eqref{20260629.1801} and the assumption that $R>\frac{c+2}{c}$, we find
that, for any $j,k\in\mathbb{N}$ with $j<k$,
\begin{align*}
\left|z_j-z_k\right|\geq|z_k-\xi|-\left|\xi-z_j\right|\geq cR^k-R^j>
\frac{c}{2}R^k+\frac{c}{2}R^j,
\end{align*}
and hence $B(z_j,\frac{c}{2}R^j)\cap B(z_k,\frac{c}{2}R^k)=\emptyset$.
Moreover, from \eqref{20260629.1801}, we infer that, for any $j\in\mathbb{N}$
and $z\in B(z_j,\frac{c}{2}R^j)$, $\frac{c}{2}R^j<\delta(z_j)-|z-z_j|
\leq\delta(z)\leq|z-z_j|+\delta(z_j)<(\frac{c}{2}+1)R^j$.
These further implies that, for any $x\in\mathbb{R}^n$, $r\in(0,\infty)$,
and $N\in\mathbb{N}$,
\begin{align*}
\int_{B(x,r)}|f_N(z)|^qw(z)\,dz&\leq\left(\sum_{\genfrac{}{}{0pt}{}
{j\in\mathbb{N}}{R^j\leq2r}}+\sum_{\genfrac{}{}{0pt}{}
{j\in\mathbb{N}}{R^j>2r}}\right)R^{-jq(\frac{n}{p}-\frac{n-d-1}{q})}
w\left(B(x,r)\cap B\left(z_j,\frac{c}{2}R^j\right)\right)\\
&\lesssim\sum_{\genfrac{}{}{0pt}{}{j\in\mathbb{N}}{R^j\leq2r}}
R^{-jq(\frac{n}{p}-\frac{n-d-1}{q})}R^{j(d+1-n)}
\left|B\left(z_j,\frac{c}{2}R^j\right)\right|\\
&\quad+\sum_{\genfrac{}{}{0pt}{}{j\in\mathbb{N}}{R^j>2r}}
R^{-jq(\frac{n}{p}-\frac{n-d-1}{q})}R^{j(d+1-n)}|B(x,r)|\\
&\sim\sum_{\genfrac{}{}{0pt}{}{j\in\mathbb{N}}{R^j\leq2r}}R^{jn(1-\frac{q}{p})}
+\sum_{\genfrac{}{}{0pt}{}{j\in\mathbb{N}}{R^j>2r}}R^{-j\frac{nq}{p}}r^n
\lesssim r^{n(1-\frac{q}{p})}.
\end{align*}
Thus, for any $N\in\mathbb{N}$,
\begin{align*}
\|f_N\|_{\mathcal{M}_q^p(\Omega,w)}=
\sup_{x\in\mathbb{R}^n,\,r\in(0,\infty)}
|B(x,r)|^{\frac1p-\frac1q}\left[\int_{B(x,r)}|f_N(z)|^qw(z)\,dz\right]^\frac1q
\lesssim1,
\end{align*}
where the implicit positive constant is independent of $N$.

On the other hand, we show that
$\|f_N\|_{\widetilde{\mathcal{M}}^p_q(\Omega,w)}\to\infty$
as $N\to\infty$, which contradicts \eqref{20260629.2159}. Indeed,
by \eqref{20260629.1801}, we find that, for any $j\in\mathbb{N}$ and $z\in
B(z_j,\frac{c}{2}R^j)$,
\begin{align*}
\frac{c}{2}R^j<\left|z_j-\xi\right|-|\xi-z|\leq
|z-\xi|\leq\left|z-z_j\right|+\left|z_j-\xi\right|
<\left(\frac{c}{2}+1\right)R^j,
\end{align*}
which further implies that $\{z\}\cup B(\xi,1)\subset
B(\xi,(\frac{c}{2}+1)R^j)$. Combining this and Lemma \ref{lem:250104-11}(ii),
we obtain
\begin{align*}
R^{-j(d+1)}\lesssim\frac{w(\Omega(\xi,1))}{w(B(\xi,(\frac{c}{2}+1)R^j))}\leq
M_w\left(\boldsymbol{1}_{\Omega(\xi,1)}\right)(z).
\end{align*}
This, together with the proven conclusions that
$\{B(z_j,\frac{c}{2}R^j)\}_{j\in\mathbb{N}}$ are pairwise disjoint and
$\delta(z)\sim R^j$ for any $j\in\mathbb{N}$ and $z\in B(z_j,\frac{c}{2}R^j)$,
and with the assumption that $\frac{n}{p}-\frac{n}{q}+\frac{(d+1)\theta}{q}\leq0$,
further implies that
\begin{align*}
\|f_N\|_{\widetilde{\mathcal{M}}^p_q(\Omega,w)}&\geq
\left\{\int_\Omega|f_N(z)|^q\left[M_w(\boldsymbol{1}
_{\Omega(\xi,1)})(z)\right]^\theta w(z)\,dz\right\}^\frac1q\\
&\gtrsim\left[\sum_{j=1}^N\int_{B(z_j,\frac{c}{2}R^j)}
R^{-jq(\frac{n}{p}-\frac{n-d-1}{q})}R^{-j\theta(d+1)}w(z)\,dz\right]^\frac1q\\
&\sim\left\{\sum_{j=1}^NR^{-jq[\frac{n}{p}-
\frac{n}{q}+\frac{(d+1)\theta}{q}]}\right\}^\frac1q\\
&\sim\begin{cases}
R^{-N[\frac{n}{p}-\frac{n}{q}+\frac{(d+1)\theta}{q}]}\to\infty &\mbox{if\ }
\frac{n}{p}-\frac{n}{q}+\frac{(d+1)\theta}{q}<0,\\
N\to\infty &\mbox{if\ }\frac{n}{p}-\frac{n}{q}+\frac{(d+1)\theta}{q}=0
\end{cases}
\end{align*}
as $N\to\infty$. This contradicts \eqref{20260629.2159}. Therefore,
\eqref{eq:260121-1} holds, which completes the proof of the necessity and
hence Theorem \ref{prop:250104-11}.
\end{proof}

\subsection{Boundedness of Some Classical Operators}
\label{subsection3.2}

In this subsection, we state the boundedness results of some classical
operators on $\mathcal{M}^p_q(\Omega,w)$. First, as a direct application of
Propositions \ref{prop:260117-1} and \ref{prop:260117-2} and \cite[Theorem
1.4]{Nakamura16}, we obtain the following Fefferman--Stein vector-valued
inequality on weighted Morrey spaces ${\mathcal M}^p_q(\Omega,w)$.
We omit the details here.

\begin{proposition}\label{thm:260126-1}
Let $1<q\leq p<\infty$ satisfy \eqref{20260623.2235},
$r\in(1,\infty)$, and $w$ be as in \eqref{dw}.
Then there exists a positive constant $C$ such that,
for any sequence $\{f_k\}_{k\in\mathbb{N}}$ in $\mathcal{M}^p_q(\Omega,w)$,
\begin{align*}
\left\|\left(\sum_{k\in\mathbb{N}}|Mf_k|^r\right)
^{\frac1r}\right\|_{\mathcal{M}^p_q(\Omega,w)}
\leq C\left\|\left(\sum_{k\in\mathbb{N}}|f_k|^r\right)
^{\frac1r}\right\|_{\mathcal{M}^p_q(\Omega,w)}.
\end{align*}
\end{proposition}

For any given $j\in\mathbb{N}\cap[1,n]$, the \emph{$j$-th Riesz
transform} $R_j$ is defined by setting, for any suitable function $f$ on
$\mathbb{R}^n$ and for any $x\in\mathbb{R}^n$,
\begin{align*}
R_jf(x):=\mathrm{p.v.}\int_{\mathbb{R}^n}
\frac{x_j-y_j}{|x-y|^{n+1}}f(y)\,dy.
\end{align*}
Second, from Propositions \ref{prop:260117-1} and
\ref{prop:260117-2} and \cite[Theorem 156]{Sawano-Book II}, we deduce that the
following vector-valued inequalities of Riesz transforms on
$\mathcal{M}^p_q(\Omega,w)$ hold; we omit the details here.

\begin{proposition}\label{prop:260117-3}
Let $1<q\leq p<\infty$ satisfy \eqref{20260623.2235},
$r\in(1,\infty)$, and $w$ be as in \eqref{dw}.
Then there exists a positive constant $C$ such that, for
any $j\in\mathbb{N}\cap[1,n]$ and any sequence
$\{f_k\}_{k\in\mathbb{N}}$ in $\mathcal{M}^p_q(\Omega,w)$,
\begin{align*}
\left\|\left(\sum_{k\in\mathbb{N}}\left|R_jf_k\right|^r\right)
^{\frac1r}\right\|_{\mathcal{M}^p_q(\Omega,w)}
\leq C\left\|\left(\sum_{k\in\mathbb{N}}|f_k|^r\right)
^{\frac1r}\right\|_{\mathcal{M}^p_q(\Omega,w)}.
\end{align*}
\end{proposition}

The boundedness of Riesz transforms on $\mathcal{M}^p_q(\Omega,w)$,
established in Proposition~\ref{prop:260117-3},
allows us to identify the homogeneous weighted Sobolev--Morrey norm
with a norm defined via the fractional Laplacian (see Theorem
\ref{cor:260117-1} below).

Let $\alpha\in(0,n)$. The \emph{fractional integral operator
$I_\alpha$} is defined by setting, for any suitable function $f$ on
$\mathbb{R}^n$ and for any $x\in\mathbb{R}^n$,
\begin{align*}
I_\alpha f(x):=\int_{\mathbb{R}^n}\frac{f(y)}{|x-y|^{n-\alpha}}\,dy.
\end{align*}
Finally, we establish the boundedness of fractional integral operators
on $\mathcal{M}^p_q(\Omega,w)$.

\begin{proposition}\label{prop:20260420.1836}
Let $1<u\leq s<\infty$, $1<q\leq p<\infty$ with $\frac{s}{u}=\frac{p}{q}$
and $u<q$, $\alpha\in(0,n)$ satisfy $\alpha=(n-d-1)(\frac1u-\frac1q)
-n(\frac1s-\frac1p)$, and $w$ be as in \eqref{dw}.
Then $I_\alpha$ is bounded from
$\mathcal{M}_q^p(\Omega,w)$ to $\mathcal{M}_u^s(\Omega,w)$.
\end{proposition}

\begin{proof}
By \cite[Corollary 1.9]{Nakamura16}, it suffices to prove that there exists
a positive constant $C$ such that, for any ball $B\subset\mathbb{R}^n$,
\begin{align}\label{20260420.1843}
|B|^{\frac1s-\frac1p+\frac\alpha n}\leq
C\left[\frac{w(B)}{|B|}\right]^{\frac1q-\frac1u}.
\end{align}
Indeed, let $x\in\mathbb{R}^n$ and $r\in(0,\infty)$. If $\delta(x)\leq2r$, then,
from the definition of $\alpha$ and Lemma \ref{lem:250104-11}(ii), we infer that
\begin{align*}
|B(x,r)|^{\frac1s-\frac1p+\frac\alpha n}&\sim
r^{n(\frac1s-\frac1p+\frac\alpha n)}
=r^{(n-d-1)(\frac1u-\frac1q)}\sim
\left[\frac{w(B(x,r))}{|B(x,r)|}\right]^{\frac1q-\frac1u}.
\end{align*}
If $\delta(x)>2r$, then, using the definition of $\alpha$, $u<q$, and Lemma
\ref{lem:250104-11}(i), we find that
\begin{align*}
|B(x,r)|^{\frac1s-\frac1p+\frac\alpha n}&\sim
r^{n(\frac1s-\frac1p+\frac\alpha n)}=r^{(n-d-1)(\frac1u-\frac1q)}
\lesssim[\delta(x)]^{(n-d-1)(\frac1u-\frac1q)}
\sim\left[\frac{w(B(x,r))}{|B(x,r)|}\right]^{\frac1q-\frac1u}.
\end{align*}
In conclusion, \eqref{20260420.1843} holds. This finishes the proof of
Proposition \ref{prop:20260420.1836}.
\end{proof}

\subsection{Littlewood--Paley Characterizations}
\label{subsection3.3}

Denote by $\mathcal{S}(\mathbb{R}^n)$ the \emph{space of all Schwartz
functions} equipped with the well-known topology determined by a countable
family of norms and by $\mathcal{S}'(\mathbb{R}^n)$ its dual space, the
\emph{space of all tempered distributions} equipped with the weak-$\ast$
topology. In this subsection, we present the Littlewood--Paley
characterization of $\mathcal{M}^p_q(\Omega,w)$. To begin with, we point out
that $\mathcal{M}^p_q(\Omega,w)$ can be embedded into
$\mathcal{S}'(\mathbb{R}^n)$.

\begin{proposition}\label{lem:250919-31}
Let $1\leq q\leq p<\infty$ and $w$ be as in \eqref{dw}. Then
$\mathcal{M}^p_q(\Omega,w)\subset\mathcal{S}'(\mathbb{R}^n)$.
\end{proposition}

\begin{proof}
Let $f\in\mathcal{M}^p_q(\Omega,w)$. After a translation of the coordinate
system, we may assume without loss of generality that $\mathbf{0}\in\Gamma$.
Here, and thereafter, ${\bf0}$ denotes the \emph{origin} of $\mathbb{R}^n$.
Combining (iii) and (ii) of Lemma \ref{lem:250104-11} and H\"older's
inequality, we obtain, for any $j\in\mathbb{Z}$,
\begin{align}\label{eq:20260209}
\int_{B(\mathbf{0},2^j)}|f(x)|\,dx
&=|B(\mathbf{0},2^j)|\fint_{B(\mathbf{0},2^j)}|f(x)|\,dx\notag\\
&\lesssim\frac{|B(\mathbf{0},2^j)|}{w(B(\mathbf{0},2^j))}
\int_{B(\mathbf{0},2^j)}|f(x)|w(x)\,dx\notag\\
&\leq\frac{|B(\mathbf{0},2^j)|^{1+\frac1q-\frac1p}}
{[w(B(\mathbf{0},2^j))]^\frac1q}\|f\|_{\mathcal{M}^p_q(\Omega,w)}
\sim2^{j(n-\frac{n}{p}+\frac{n-d-1}{q})}
\|f\|_{\mathcal{M}^p_q(\Omega,w)}.
\end{align}
Choose $N\in\mathbb{N}$ such that $n-\frac{n}{p}+\frac{n-d-1}{q}-N<0$.
Let $\varphi\in\mathcal{S}(\mathbb{R}^n)$. Then, for any $x\in\mathbb{R}^n$,
$|\varphi(x)|\lesssim\frac{1}{(1+|x|)^N}$, where the implicit positive constant
depends only on $N$ and the Schwartz norm of $\varphi$. This, together with
\eqref{eq:20260209}, further implies that
\begin{align*}
\|f\varphi\|_{L^1(\mathbb{R}^n)}&=\int_{B(\mathbf{0},1)}|f(x)\varphi(x)|\,dx
+\sum_{j\in\mathbb{Z}_+}\int_{B(\mathbf{0},2^{j+1})\setminus B(\mathbf{0},2^j)}
|f(x)\varphi(x)|\,dx\\
&\lesssim\|f\|_{\mathcal{M}^p_q(\Omega,w)}+\sum_{j\in\mathbb{Z}_+}
\int_{B(\mathbf{0},2^{j+1})}\frac{|f(x)|}{2^{jN}}\,dx\\
&\lesssim\left[1+\sum_{j\in\mathbb{Z}_+}
2^{j(n-\frac{n}{p}+\frac{n-d-1}{q}-N)}\right]\|f\|_{\mathcal{M}^p_q(\Omega,w)}
\lesssim\|f\|_{\mathcal{M}^p_q(\Omega,w)}.
\end{align*}
Therefore, $f$ determines a continuous linear functional on
$\mathcal{S}(\mathbb{R}^n)$ and hence $f\in\mathcal{S}'(\mathbb{R}^n)$.
This finishes the proof of Proposition \ref{lem:250919-31}.
\end{proof}

Let $\varphi\in\mathcal{S}(\mathbb{R}^n)$. In what follows, for any
$j\in\mathbb{Z}$, define $\varphi_j(\cdot):=\varphi(2^{-j}\cdot)$.
In addition, the \emph{Fourier transform} $\mathcal{F}\varphi$ and the
\emph{inverse Fourier transform} $\mathcal{F}^{-1}\varphi$ of $\varphi$ are
defined, respectively, by setting, for any $\xi\in\mathbb{R}^n$,
\begin{align*}
\mathcal{F}\varphi(\xi):=(2\pi)^{-\frac{n}{2}}\int_{\mathbb{R}^n}
\varphi(x)e^{-ix\cdot\xi}\,dx\mbox{\ \ and\ \ }
\mathcal{F}^{-1}\varphi(\xi):=\mathcal{F}\varphi(-\xi).
\end{align*}
Let $f\in\mathcal{S}'(\mathbb{R}^n)$. The \emph{Fourier transform}
$\mathcal{F}f$ and the \emph{inverse Fourier transform} $\mathcal{F}^{-1}f$ of
$f$ are defined, respectively, by setting, for any
$\varphi\in\mathcal{S}(\mathbb{R}^n)$, $\langle\mathcal{F}f,\varphi\rangle:=
\langle f,\mathcal{F}\varphi\rangle$ and $\langle\mathcal{F}^{-1}f,\varphi\rangle
:=\langle f,\mathcal{F}^{-1}\varphi\rangle$. It is well known that, for any
$g\in\mathcal{S}'(\mathbb{R}^n)$ and $h\in\mathcal{S}(\mathbb{R}^n)$,
$gh\in\mathcal{S}'(\mathbb{R}^n)$. For any given
$\psi\in\mathcal{S}(\mathbb{R}^n)$, define
$\psi(D)f:=\mathcal{F}^{-1}(\psi\mathcal{F}f).$
The operator $\psi(D)$ is called the \emph{Fourier multiplier}
associated with the symbol $\psi$.

Let $1<q\leq p<\infty$. Recall that the Morrey space
$\mathcal{M}^p_q(\mathbb{R}^n)$ can be realized as the dual
space of a Banach space, i.e., the block space
(see, for instance, \cite[Theorem 347]{SDH20-1}). In addition, the mapping
\begin{align*}
f\in\mathcal{M}^p_q(\mathbb{R}^n)\longmapsto w^{-\frac1q}f
\in\mathcal{M}^p_q(\Omega,w)
\end{align*}
is an isomorphism. Thus, $\mathcal{M}^p_q(\Omega,w)$ can also be regarded as
the dual space of a Banach space. As a consequence, we obtain the following
Littlewood--Paley characterization of $\mathcal{M}^p_q(\Omega,w)$, whose proof
follows closely that of \cite[Theorem 1.1]{IST15}; we omit the details.

\begin{proposition}\label{prop:LP-Morrey}
Let $1<q\leq p<\infty$ satisfy \eqref{20260623.2235},
$w$ be as in \eqref{dw},
and $\varphi\in C^\infty_{\mathrm{c}}(\mathbb{R}^n)$ be a nonnegative
function such that $\sum_{j\in\mathbb{Z}}\varphi_j=
\boldsymbol{1}_{\mathbb{R}^n\setminus\{{\mathbf{0}}\}}$.
Then the following statements hold.
\begin{enumerate}
\item[\rm(i)] If $f\in\mathcal{M}^p_q(\Omega,w)$, then
$f=\sum_{j\in\mathbb{Z}}\varphi_j(D)f$ in the weak-$\ast$ topology
of $\mathcal{M}^p_q(\Omega,w)$ and
\begin{align*}
\|f\|_{\mathcal{M}^p_q(\Omega,w)}\sim\left\|\left[\sum_{j\in\mathbb{Z}}
\left|\varphi_j(D)f\right|^2\right]^\frac12\right\|_{\mathcal{M}^p_q(\Omega,w)},
\end{align*}
where the positive equivalence constants are independent of $f$.

\item[\rm(ii)] If $f\in\mathcal{S}'(\mathbb{R}^n)$ satisfies
\begin{align*}
\left\|\left[\sum_{j\in\mathbb{Z}}\left|\varphi_j(D)f\right|^2\right]^\frac12
\right\|_{\mathcal{M}^p_q(\Omega,w)}<\infty,
\end{align*}
then there exists a decomposition $f=F+P$ in $\mathcal{S}'(\mathbb{R}^n)$,
where $F:=\sum_{j\in\mathbb{Z}}\varphi_j(D)f$ pointwise and $P$ is a
polynomial on $\mathbb{R}^n$. Moreover,
\begin{align*}
\|F\|_{\mathcal{M}^p_q(\Omega,w)}\sim\left\|\left[\sum_{j\in\mathbb{Z}}
\left|\varphi_j(D)f\right|^2\right]^\frac12\right\|_{\mathcal{M}^p_q(\Omega,w)},
\end{align*}
where the positive equivalence constants are independent of $f$.
\end{enumerate}
\end{proposition}

\subsection{Complex Interpolation}
\label{subsection3.4}

In this subsection, we list an interpolation formula of
$\mathcal{M}^p_q(\Omega,w)$. Suppose that $X_0$ and $X_1$ are two complex
Banach spaces. The couple $(X_0, X_1)$ is said to be \emph{compatible} if
$X_0$ and $X_1$ are continuously embedded into a common Hausdorff topological
vector space $X$. In this case, we can naturally define two Banach spaces
$X_0+X_1$ and $X_0 \cap X_1$. More precisely, let
\begin{align*}
X_0+X_1:=\left\{x\in X:x=x_0+x_1,\ x_0\in X_0,\mbox{\ and\ }x_1\in X_1\right\}
\end{align*}
equipped with the norm
\begin{align*}
\|x\|_{X_0+X_1}:=\inf\{\|x_0\|_{X_0}+\|x_1\|_{X_1}:
x_0\in X_0,\ x_1\in X_1,\mbox{\ and\ }x=x_0+x_1\}.
\end{align*}
For any $x\in X_0\cap X_1$, we define the norm $\|x\|_{X_0\cap X_1}:=
\max\{\|x\|_{X_0},\,\|x\|_{X_1}\}$. For a comprehensive introduction of these
interpolation couples, we refer to \cite[Section 2.3]{BeLo76}.

Next, we recall the concept of Calder\'on's first
complex interpolation spaces (see \cite[p.\,114, 3]{Calderon64}
or \cite[p.\,88]{BeLo76}). To this end, let
$S:= \{ z\in\mathbb{C}:0<\Re(z)<1\}$ and $\overline{S}$ be its closure in
$\mathbb{C}$, where $\Re(z)$ denotes the \emph{real part} of $z$.

\begin{definition}
Let $(X_0, X_1)$ be a compatible couple of complex Banach spaces.
\begin{enumerate}
\item[\rm(i)] The \emph{space $\mathcal{F}(X_0, X_1)$} is defined to be the
set of all functions $F:\overline{S}\to X_0+X_1$ such that
\begin{enumerate}
\item[{\rm(a)}] $F$ is bounded and continuous on $\overline{S}$,
\item[{\rm(b)}] $F$ is analytic in $S$,
\item[{\rm(c)}] for any $j\in\{0,\,1\}$, the function $t\in\mathbb{R}
\longmapsto F(j+it)\in X_j$ is bounded and continuous.
\end{enumerate}
Moreover, the space $\mathcal{F}(X_0,X_1)$ is equipped with the norm,
for any $F\in\mathcal{F}(X_0,X_1)$,
\begin{align*}
\|F\|_{\mathcal{F}(X_0,X_1)}:=\max\left\{\sup_{z\in
i\mathbb{R}}\|F(z)\|_{X_0},\,\sup_{z\in 1+i\mathbb{R}}\|F(z)\|_{X_1}\right\}.
\end{align*}
\item[\rm(ii)] Let $\theta\in(0,1)$. The \emph{first complex interpolation
space $[X_0,X_1]_\theta$} with respect to $(X_0, X_1)$ is defined to be the
set of all functions $f\in X_0+X_1$ such that $f=F(\theta)$ for some
$F\in\mathcal{F}(X_0,X_1)$, equipped with the norm
\begin{align*}
\|f\|_{[X_0,X_1]_\theta}:=\inf\left\{\|F\|_{\mathcal{F}(X_0,X_1)}:
f=F(\theta)\mbox{\ for\ some\ }F\in\mathcal{F}(X_0,X_1)\right\}.
\end{align*}
\end{enumerate}
\end{definition}

We then recall the concept of Calder\'on's second complex interpolation spaces
(see \cite[p.\,115, 5]{Calderon64} or \cite[p.\,89]{BeLo76}).

\begin{definition}
Let $(X_0, X_1)$ be a compatible couple of complex Banach spaces.
\begin{enumerate}
\item[\rm(i)] The \emph{space $\mathcal{G}(X_0, X_1)$} is defined to be the
set of all functions $G:\overline{S}\to X_0+X_1$ such that
\begin{enumerate}
\item[{\rm (a)}] $G$ is continuous on $\overline{S}$ and
$\sup_{z\in\overline{S}}\|\frac{G(z)}{1+|z|}\|_{X_0+X_1}<\infty$,
\item[{\rm (b)}] $G$ is analytic in $S$,
\item[{\rm (c)}] for any $j\in\{0,\,1\}$, the function $t\in\mathbb{R}
\longmapsto G(j+it)-G(j)\in X_j$ is Lipschitz continuous; i.e., there exists
a positive constant $C$ such that, for any $t,s\in\mathbb{R}$,
\begin{align*}
\left|\left[G(j+it)-G(j)\right]-\left[G(j+is)-G(j)\right]\right|\leq C|t-s|.
\end{align*}
\end{enumerate}
In addition, the space $\mathcal{G}(X_0,X_1)$ is equipped with the norm
\begin{align*}
\|G\|_{\mathcal{G}(X_0,X_1)}:=\max\left\{
\sup_{\genfrac{}{}{0pt}{}{t,s\in\mathbb{R}}{t\neq s}}
\left\|\frac{G(it)-G(is)}{t-s}\right\|_{X_0},\,
\sup_{\genfrac{}{}{0pt}{}{t,s\in\mathbb{R}}{t\neq s}}
\left\|\frac{G(1+it)-G(1+is)}{t-s}\right\|_{X_1}\right\}.
\end{align*}
\item[\rm(ii)] Let $\theta\in(0,1)$. The \emph{second complex interpolation
space $[X_0, X_1]^\theta$} with respect to $(X_0, X_1)$ is defined to be the
set of all functions $f\in X_0+X_1$ such that $f=G'(\theta)$ for some
$G\in\mathcal{G}(X_0,X_1)$, equipped with the norm
\begin{align*}
\|f\|_{[X_0,X_1]^\theta}:=\inf\left\{\|G\|_{\mathcal{G}(X_0,X_1)}:
f=G'(\theta)\mbox{\ for\ some\ }G\in\mathcal{G}(X_0,X_1) \right\}.
\end{align*}
\end{enumerate}
\end{definition}

As a direct corollary of \cite[Theorem 2.3]{HakimNakamuraSawano17}, we state
the following complex interpolation formula for $\mathcal{M}^p_q(\Omega,w)$.
The details are omitted.

\begin{proposition}\label{thm:interp-morrey}
Let $1\leq q_0\leq p_0<\infty$, $1\leq q_1\leq p_1<\infty$ with
\begin{align}\label{eq:ratio-condition}
\frac{q_0}{p_0}=\frac{q_1}{p_1},
\end{align}
$\theta\in(0,1)$, $1\leq q\leq p<\infty$ satisfy
\begin{align}\label{eq:interp-indices}
\frac{1}{p}=\frac{1-\theta}{p_0}+\frac{\theta}{p_1}
\mbox{\ \ and\ \ }
\frac{1}{q}=\frac{1-\theta}{q_0}+\frac{\theta}{q_1},
\end{align}
and $w$ be as in \eqref{dw}.
Then
\begin{align}\label{eq:interp-morrey}
\left[\mathcal{M}^{p_0}_{q_0}(\Omega,w),
\mathcal{M}^{p_1}_{q_1}(\Omega,w)\right]^\theta=\mathcal{M}^p_q(\Omega,w).
\end{align}
\end{proposition}

\begin{remark}\label{rem:diag-Lp}
We use the same notation as in Proposition \ref{thm:interp-morrey}.
If $p_0=q_0$ and $p_1=q_1$, then, combining \cite[Theorem 4.3.1]{BeLo76}, the
reflexivity of $\mathcal M^{p_i}_{q_i}(\Omega,w)=L^{p_i}(\Omega,w)$ with
$i\in\{0,1\}$, and \eqref{eq:interp-morrey}, we recover the standard weighted
$L^p$ interpolation formula
\begin{align*}
\left[L^{p_0}(\Omega,w),L^{p_1}(\Omega,w) \right]_\theta=
\left[L^{p_0}(\Omega,w),L^{p_1}(\Omega,w)\right]^\theta=L^p(\Omega,w).
\end{align*}
\end{remark}

\section{Weighted Sobolev--Morrey Spaces $\dot{W}^1\mathcal{M}^p_q(\Omega,w)$}
\label{section4}

In this section, we extend the weighted Sobolev space theory developed in
\cite{DFM21} to the framework of weighted Sobolev--Morrey spaces by seven
subsections. In Subsection \ref{subsection4.1}, we introduce the
weighted Sobolev--Morrey spaces $\dot{W}^1\mathcal{M}^p_q(\Omega,w)$ adapted
to $\Gamma$ and establish their completeness. Subsection \ref{subsection4.2}
is devoted to presenting a Riesz potential characterization of these spaces,
which subsequently leads to the corresponding Sobolev--Morrey embedding
theorem. The focus then shifts to the qualitative properties of functions
within these spaces: Subsection \ref{subsection4.3} examines the behavior of
ball averages as the radius tends to infinity, while Subsection
\ref{subsection4.4} investigates their continuity properties. In Subsections
\ref{subsection4.5} and \ref{subsection4.6}, we turn our attention to boundary
behavior, defining the trace space $Q^p_q(\Gamma)$ and proving the boundedness
of the trace operator $T$ and the extension operator $E$. We also show that
$E$ serves as a right inverse of $T$. Finally, in Subsection
\ref{subsection4.7}, combining the trace and the extension theorems, we
establish the complex interpolation identities for both
$\dot{W}^1 \mathcal{M}^p_q(\Omega, w)$ and $Q^p_q(\Gamma)$.

\subsection{Weighted Sobolev--Morrey Spaces Adapted to $\Gamma$}
\label{subsection4.1}

In this subsection, we introduce the weighted Sobolev--Morrey space
$\dot{W}^1\mathcal{M}^p_q(\Omega,w)$ adapted to $\Gamma$ and we prove
the completeness of these spaces.

\begin{definition}\label{def5.1}
Let $1\leq q\leq p<\infty$ and $w$ be as in \eqref{dw}.
The \emph{weighted Sobolev--Morrey space
$\dot{W}^1\mathcal{M}^p_q(\Omega,w)$} adapted to $\Gamma$
is defined to be the set of all $f\in L^1_\mathrm{loc}(\Omega)$, modulo
constant functions, such that $\nabla f \in \mathcal{M}^p_q(\Omega,w)$. This
space is equipped with the norm $\|f\|_{\dot{W}^1\mathcal{M}^p_q(\Omega,w)}:=
\|\nabla f\|_{\mathcal{M}^p_q(\Omega,w)}$.
\end{definition}

In particular, for any given $p\in[1,\infty)$, we simply denote
$\dot{W}^1\mathcal{M}^p_p(\Omega,w)$ by $\dot{W}^{1,p}(\Omega,w)$,
the \emph{weighted Sobolev space} adapted to $\Gamma$.
As an extension of \cite[Lemma 5.1]{DFM21}, we show the completeness of
$\dot{W}^1\mathcal{M}^p_q(\Omega,w)$.

\begin{lemma}\label{lem:250919-3}
Let $1\leq q\leq p<\infty$ and $w$ be as in \eqref{dw}.
Then $\dot{W}^1\mathcal{M}^p_q(\Omega,w)$ is complete.
\end{lemma}

\begin{proof}
Without loss of generality, we may assume that $\mathbf{0}\in\Gamma$. Let
$\{u_j\}_{j\in\mathbb{N}}$ be a Cauchy sequence in
$\dot{W}^1\mathcal{M}^p_q(\Omega,w)$. From Tonelli's
theorem, Lemma \ref{DFM21}(i), and (iii) and (ii) of Lemma \ref{lem:250104-11},
we deduce that, for any given $R\in[1,\infty)$ and for any $j,k\in\mathbb{N}$,
\begin{align}\label{20260426.1837}
&\int_{B(\mathbf{0},R)}\left|u_j(x)-\left(u_j\right)_{B(\mathbf{0},1)}-u_k(x)
+(u_k)_{B(\mathbf{0},1)}\right|\,dx\notag\\
&\quad\lesssim R^n\fint_{B(\mathbf{0},R)}\fint_{B(\mathbf{0},1)}
\left|u_j(x)-u_j(y)-u_k(x)+u_k(y)\right|\,dy\,dx\notag\\
&\quad\lesssim R^{2n}\fint_{B(\mathbf{0},R)}\fint_{B(\mathbf{0},R)}
\left|u_j(x)-u_k(x)-u_j(y)+u_k(y)\right|\,dx\,dy\notag\\
&\quad\lesssim R^{2n+1}\fint_{B(\mathbf{0},R)}
\left|\nabla u_j(x)-\nabla u_k(x)\right|\,dx
\lesssim\frac{R^{2n+1}}{w(B(\mathbf{0},R))}\int_{B(\mathbf{0},R)}
\left|\nabla u_j(x)-\nabla u_k(x)\right|w(x)\,dx\notag\\
&\quad\lesssim\frac{R^{2n+1}|B(\mathbf{0},R)|^{\frac1q-\frac1p}}
{[w(B(\mathbf{0},R))]^\frac1q}\left\|\nabla u_j-\nabla u_k\right\|
_{\mathcal{M}^p_q(\Omega,w)}\notag\\
&\quad\sim R^{2n+1-\frac{n}{p}+\frac{n-d-1}{q}}
\left\|\nabla u_j-\nabla u_k\right\|_{\mathcal{M}^p_q(\Omega,w)}.
\end{align}
Therefore, $\{u_j-(u_j)_{B(\mathbf{0},1)}\}_{j\in\mathbb{N}}$ is a Cauchy
sequence in $L^1_\mathrm{loc}(\mathbb{R}^n)$ and hence converges to a function
$u\in L^1_\mathrm{loc}(\mathbb{R}^n)$.

Moreover, using Proposition \ref{lem:250919-2}, we find that, for any
$k\in\mathbb{N}\cap[1,n]$, the sequence $\{\partial_k u_j\}_{j\in\mathbb{N}}$
is a Cauchy sequence in $L^1_\mathrm{loc}(\mathbb{R}^n)$ and hence converges to
a function $u^k\in L^1_\mathrm{loc}(\mathbb{R}^n)$. Then it is easy to prove
that, for any $k\in\mathbb{N}\cap[1,n]$, $u^k=\partial_k u$ in
$\mathcal{S}'(\mathbb{R}^n)$. This, together with Riesz's lemma and a standard
Cantor's diagonal argument, further implies that we can extract a subsequence
$\{u_{j_k}\}_{k\in\mathbb{N}}$ of $\{u_j\}_{j\in\mathbb{N}}$ such that $\nabla
u_{j_k}\to\nabla u$ almost everywhere as $k\to\infty$. From this and Fatou's
property of $\|\cdot\|_{\mathcal {M}^p_q(\Omega,w)}$, it follows that
\begin{align*}
\left\|u-u_j\right\|_{\dot{W}^1\mathcal{M}^p_q(\Omega,w)}=
\left\|\nabla u-\nabla u_j\right\|_{\mathcal{M}^p_q(\Omega,w)}
\leq\liminf_{k\to\infty}\left\|\nabla u_{j_k}-\nabla u_j\right\|
_{\mathcal{M}^p_q(\Omega,w)}\to0
\end{align*}
as $j\to\infty$. Thus, $u_j\to u$ in $\dot{W}^1\mathcal{M}^p_q(\Omega,w)$ as
$j\to\infty$. This finishes the proof of Lemma \ref{lem:250919-3}.
\end{proof}

\begin{remark}
We use the same notation as in Lemma \ref{lem:250919-3}.
The proof of Lemma \ref{lem:250919-3} suggests that, when considering
functions $u$ in $\dot{W}^1\mathcal{M}^p_q(\Omega,w)$,
it is more natural to work with $u-u_{B(\mathbf{0},1)}$ rather than $u$ itself.
\end{remark}

\subsection{Riesz Potential Characterizations and Sobolev--Morrey Embeddings}
\label{subsection4.2}

In this subsection, we characterize $\dot{W}^1\mathcal{M}^p_q(\Omega,w)$ in
terms of Riesz potentials and derive the corresponding Sobolev--Morrey
embedding theorem. We first observe that, if $\frac{n}{p}-\frac{n-d-1}{q}>0$,
then any element $u\in\dot{W}^1\mathcal{M}^p_q(\Omega,w)$ can be
realized as a tempered distribution in $\mathcal{S}'(\mathbb{R}^n)$ via the
mapping $u\mapsto u-u_{B(\mathbf{0},1)}$.

\begin{lemma}\label{lem:250919-31b}
Let $1\leq q\leq p<\infty$ satisfy \eqref{20260623.2235},
$w$ be as in \eqref{dw},
and $u\in\dot{W}^1\mathcal{M}^p_q(\Omega,w)$.
Then $u-u_{B(\mathbf{0},1)}\in\mathcal{S}'(\mathbb{R}^n)$.
\end{lemma}

\begin{proof}
Without loss of generality, we may assume that $\mathbf{0}\in\Gamma$ and
$u_{B(\mathbf{0},1)}=0$. By Lemma \ref{DFM21}(i), Proposition
\ref{lem:250919-2}, Lemma \ref{lem:250104-11}(iii), H\"older's inequality, and
Lemma \ref{lem:250104-11}(ii), we conclude that, for any $j\in\mathbb{Z}_+$,
\begin{align}\label{20260422.2213}
&\fint_{B(\mathbf{0},2^{j+1})}|u(x)|\,dx-
\fint_{B(\mathbf{0},2^j)} |u(x)|\,dx\notag\\
&\quad\lesssim\fint_{B(\mathbf{0},2^{j+1})}
\fint_{B(\mathbf{0},2^{j+1})}|u(x)-u(y)|\,dx\,dy
\lesssim 2^j\fint_{B(\mathbf{0},2^{j+1})}|\nabla u(x)|\,dx\notag\\
&\quad\lesssim\frac{2^j}{w(B(\mathbf{0},2^{j+1}))}
\int_{B(\mathbf{0},2^{j+1})}|\nabla u(x)|w(x)\,dx\notag\\
&\quad\lesssim2^j\left[\frac{1}{w(B(\mathbf{0},2^{j+1}))}
\int_{B(\mathbf{0},2^{j+1})}|\nabla u(x)|^qw(x)\,dx\right]^\frac1q\notag\\
&\quad\sim2^{j(1-\frac{n}{p}+\frac{n-d-1}{q})}
\left|B(\mathbf{0},2^{j+1})\right|
^{\frac1p-\frac1q}\left[\int_{B(\mathbf{0},2^{j+1})}
|\nabla u(x)|^q w(x)\,dx\right]^\frac1q\notag\\
&\quad\leq2^{j(n-\frac{n}{p}+\frac{n-d-1}{q})}
\|u\|_{\dot{W}^1\mathcal{M}^p_q(\Omega,w)}.
\end{align}
Moreover, since $u_{B(\mathbf{0},1)}=0$, it follows from an argument
similar to the one used in \eqref{20260422.2213} that
\begin{align*}
\fint_{B(\mathbf{0},1)}|u(x)|\,dx&=\fint_{B(\mathbf{0},1)}
\left|u(x)-u_{B(\mathbf{0},1)}\right|\,dx\\
&\leq\fint_{B(\mathbf{0},1)}\fint_{B(\mathbf{0},1)}|u(x)-u(y)|\,dx\,dy
\lesssim\|u\|_{\dot{W}^1\mathcal{M}^p_q(\Omega,w)}.
\end{align*}
Combining this and \eqref{20260422.2213}, we obtain, for any $j\in\mathbb{N}$,
\begin{align*}
\int_{B(\mathbf{0},2^j)}|u(x)|\,dx\lesssim
2^{j(n-\frac{n}{p}+\frac{n-d-1}{q})}
\|u\|_{\dot{W}^1\mathcal{M}^p_q(\Omega,w)},
\end{align*}
which, together with an argument used in the proof of Proposition
\ref{lem:250919-31}, further implies that, for any
$\varphi\in\mathcal{S}(\mathbb{R}^n)$,
\begin{align*}
\|u\varphi\|_{L^1(\mathbb{R}^n)}\lesssim
\|u\|_{\dot{W}^1\mathcal{M}^p_q(\Omega,w)},
\end{align*}
where the implicit positive constant depends only on the Schwartz norm of
$\varphi$. Therefore, $u-u_{B(\mathbf{0},1)}$ determines a continuous linear
functional on $\mathcal{S}(\mathbb{R}^n)$ and hence
$u-u_{B(\mathbf{0},1)}\in\mathcal{S}'(\mathbb{R}^n)$.
This finishes the proof of Lemma \ref{lem:250919-31b}.
\end{proof}

Let $\alpha\in\mathbb{R}$ and $\varphi\in
C^\infty_{\mathrm{c}}(\mathbb{R}^n)$ satisfy
$\sum_{j\in\mathbb{Z}}\varphi_j=\boldsymbol{1}_{\mathbb{R}^n\setminus
\{{\mathbf{0}}\}}$. The \emph{Riesz potential $(-\Delta)^\alpha$} is
defined by setting, for any suitable $f\in\mathcal{S}'(\mathbb{R}^n)$,
\begin{align*}
(-\Delta)^\alpha f:=\mathcal{F}^{-1}\left(|\cdot|^{2\alpha}
\mathcal{F}f\right)=\lim_{J\to\infty}\sum_{j=-J}^J\mathcal{F}^{-1}
\left(|\cdot|^{2\alpha}\varphi_j\mathcal{F}f\right).
\end{align*}
The main result of this subsection is the following theorem, which gives
a characterization of $\dot{W}^1\mathcal{M}^p_q(\Omega,w)$
in terms of the Riesz potential $(-\Delta)^\frac12$.

\begin{theorem}\label{cor:260117-1}
Let $1<q\leq p<\infty$ satisfy \eqref{20260623.2235},
$w$ be as in \eqref{dw}, and
$\varphi\in C^\infty_{\mathrm{c}}(\mathbb{R}^n)$ be a nonnegative
function such that $\sum_{j\in\mathbb{Z}}\varphi_j=
\boldsymbol{1}_{\mathbb{R}^n\setminus\{{\mathbf{0}}\}}$.
Then the following assertions hold.
\begin{enumerate}
\item[{\rm(i)}] If $f\in\dot{W}^1\mathcal{M}^p_q(\Omega,w)$, then
$(-\Delta)^\frac12f\in\mathcal{M}^p_q(\Omega,w)$ and
$\|f\|_{\dot{W}^1\mathcal{M}^p_q(\Omega,w)}\sim
\|(-\Delta)^\frac12f\|_{\mathcal{M}^p_q(\Omega,w)}$,
where the positive equivalence constants are independent of $f$.

\item[{\rm(ii)}] Conversely, if $f\in\mathcal{S}'(\mathbb{R}^n)$ satisfies
$(-\Delta)^\frac12f\in\mathcal{M}^p_q(\Omega,w)$,
then there exist $F\in\dot{W}^1\mathcal{M}^p_q(\Omega,w)$ and a polynomial
$P$ on $\mathbb{R}^n$ such that $f=F+P$ in $\mathcal{S}'(\mathbb{R}^n)$.
Moreover, $\|F\|_{\dot{W}^1\mathcal{M}^p_q(\Omega,w)}\sim
\|(-\Delta)^\frac12f\|_{\mathcal{M}^p_q(\Omega,w)}$,
where the positive equivalence constants are independent of $f$.
\end{enumerate}
In particular, $(-\Delta)^\frac12$ is an isomorphism
from $\dot{W}^1\mathcal{M}^p_q(\Omega,w)$ to $\mathcal{M}^p_q(\Omega,w)$.
\end{theorem}

\begin{proof} We first prove (i). Let $f\in\dot{W}^1\mathcal{M}^p_q(\Omega,w)$.
We claim that $(-\Delta)^\frac12f$ is well defined and belongs to
$\mathcal{M}^p_q(\Omega,w)$. Indeed, by Lemma \ref{lem:250919-31b}, we find
that $f\in\mathcal{S}'(\mathbb{R}^n)$. In addition, note that, for any
$\xi\in\mathbb{R}^n\setminus\{\mathbf{0}\}$,
\begin{align*}
|\xi|=\sum_{k=1}^n\frac{-i\xi_k}{|\xi|}(i\xi_k).
\end{align*}
Thus, for any $j\in\mathbb{Z}$,
\begin{align}\label{20260423.1438}
\mathcal{F}^{-1}\left(|\cdot|\varphi_j\mathcal{F}f\right)=
\sum_{k=1}^n\mathcal{F}^{-1}\left(\frac{-i\xi_k}{|\xi|}
\varphi_j\mathcal{F}f\right)=\sum_{k=1}^nR_k\left(\varphi_j(D)
(\partial_k f)\right).
\end{align}
From Proposition \ref{prop:260117-3}, we deduce that, for any
$k\in\mathbb{N}\cap[1,n]$, the Riesz transform $R_k$ is continuous on
$\mathcal{M}^p_q(\Omega,w)$ with respect to the weak-$\ast$ topology.
This, together with \eqref{20260423.1438} and Proposition
\ref{prop:LP-Morrey}(i), further implies that
\begin{align}\label{20260423.1503}
(-\Delta)^\frac12f&=\lim_{J\to\infty}\sum_{j=-J}^J
\sum_{k=1}^nR_k\left(\varphi_j(D)(\partial_k f)\right)=
\lim_{J\to\infty}\sum_{k=1}^n\sum_{j=-J}^J
R_k\left(\varphi_j(D)(\partial_k f)\right)\notag\\
&=\sum_{k=1}^nR_k\left(\lim_{J\to\infty}\sum_{j=-J}^J
\varphi_j(D)(\partial_k f)\right)=\sum_{k=1}^nR_k(\partial_k f)
\end{align}
in the weak-$\ast$ topology of $\mathcal{M}^p_q(\Omega,w)$. This finishes
the proof the above claim.

From this claim, it follows that, for any $k\in\mathbb{N}\cap[1,n]$ and any
$h\in\mathcal{S}(\mathbb{R}^n)$,
\begin{align*}
\langle-\partial_kf,h\rangle=\langle f,\partial_kh\rangle=
\left\langle f,\mathcal{F}^{-1}(i\xi_k\mathcal{F}h)\right\rangle=
\left\langle f,-(-\Delta)^\frac12R_kh\right\rangle=
\left\langle R_k\left((-\Delta)^\frac12f\right),h\right\rangle
\end{align*}
and hence $R_k((-\Delta)^\frac12f)=-\partial_kf$
in $\mathcal{S}'(\mathbb{R}^n)$. Combining this, \eqref{20260423.1503},
and Proposition \ref{prop:260117-3}, we conclude that
\begin{align*}
\|\nabla f\|_{\mathcal{M}^p_q(\Omega,w)}\lesssim
\left\|(-\Delta)^\frac12 f\right\|_{\mathcal{M}^p_q(\Omega,w)}
\lesssim\|\nabla f\|_{\mathcal{M}^p_q(\Omega,w)}
\end{align*}
and hence complete the proof of (i).

Next, we show (ii). Assume that $f\in\mathcal{S}'(\mathbb{R}^n)$ satisfies
$(-\Delta)^\frac12f\in\mathcal{M}^p_q(\Omega,w)$.
Let $k\in\mathbb{N}\cap[1,n]$. From the fact that
$R_k((-\Delta)^\frac12f)=-\partial_kf$ and Propositions \ref{prop:260117-3}
and \ref{prop:LP-Morrey}(i), we infer that
\begin{align*}
\left\|\left[\sum_{j\in\mathbb{Z}}\left|\varphi_j(D)(\partial_k f)
\right|^2\right]^\frac12\right\|_{\mathcal{M}^p_q(\Omega,w)}&=
\left\|\left[\sum_{j\in\mathbb{Z}}\left|\varphi_j(D)
\left(R_k\left((-\Delta)^\frac12f\right)\right)
\right|^2\right]^\frac12\right\|_{\mathcal{M}^p_q(\Omega,w)}\\
&=\left\|\left[\sum_{j\in\mathbb{Z}}\left|R_k\left(\varphi_j(D)
\left((-\Delta)^\frac12f\right)\right)
\right|^2\right]^\frac12\right\|_{\mathcal{M}^p_q(\Omega,w)}\\
&\lesssim\left\|\left[\sum_{j\in\mathbb{Z}}\left|\varphi_j(D)
\left((-\Delta)^\frac12f\right)
\right|^2\right]^\frac12\right\|_{\mathcal{M}^p_q(\Omega,w)}\\
&\sim\left\|(-\Delta)^\frac12f\right\|_{\mathcal{M}^p_q(\Omega,w)}<\infty.
\end{align*}
This, together with Proposition \ref{prop:LP-Morrey}(ii), further implies
that, for any $k\in\mathbb{N}\cap[1,n]$, there exists a decomposition
$\partial_k f=F_k+P_k$ in $\mathcal{S}'(\mathbb{R}^n)$, where
$F_k:=\sum_{j\in\mathbb{Z}}\varphi_j(D)
(\partial_k f)\in\mathcal{M}^p_q(\Omega,w)$ and
$P_k$ is a polynomial on $\mathbb{R}^n$.

Note that, for any $k_1,k_2\in\mathbb{N}\cap[1,n]$,
$\partial_{k_1}F_{k_2}=\partial_{k_2}F_{k_1}$. Therefore, for any
$k_1,k_2\in\mathbb{N}\cap[1,n]$, $\partial_{k_1}P_{k_2}=\partial_{k_2}P_{k_1}$.
This further indicates that there exists a polynomial $P$ on $\mathbb{R}^n$
such that $\partial_k P=P_k$ for any $k\in\mathbb{N}\cap[1,n]$.
Define $F:=f-P$. Then $F\in\mathcal{S}'(\mathbb{R}^n)$ and, for any
$k\in\mathbb{N}\cap[1,n]$, $\partial_kF=F_k\in\mathcal{M}^p_q(\Omega,w)
\hookrightarrow L_\mathrm{loc}^1(\mathbb{R}^n)$.
Let $\rho\in C_\mathrm{c}^\infty(\mathbb{R}^n)$ be a radial decreasing
nonnegative function with $\operatorname{supp}(\rho)\subset B(\mathbf{0},1)$
and $\|\rho\|_{L^1(\mathbb{R}^n)}=1$. For any $\varepsilon\in(0,\infty),$
let $\rho_\varepsilon(\cdot):=\frac{1}{\varepsilon^n}
\rho(\frac{\cdot}{\varepsilon})$.
By an argument similar to that used in \eqref{20260426.1837}, we conclude
that, for any given $R\in[1,\infty)$,
\begin{align*}
&\fint_{B(\mathbf{0},R)}\left|\rho_{2^{-j}}*F-
\left(\rho_{2^{-j}}*F\right)_{B(\mathbf{0},1)}-\rho_{2^{-k}}*F+
\left(\rho_{2^{-k}}*F\right)_{B(\mathbf{0},1)}\right|\,dx\\
&\quad\lesssim R^{n+1}\fint_{B(\mathbf{0},R)}\left|\rho_{2^{-j}}*\nabla F-
\rho_{2^{-k}}*\nabla F\right|\,dx\to0
\end{align*}
as $j,k\to\infty$, and hence the sequence
$\{\rho_{2^{-j}}*F-(\rho_{2^{-j}}*F)_{B(\mathbf{0},1)}\}_{j\in\mathbb{Z}}$
is a Cauchy sequence in $L^1_{\mathrm{loc}}({\mathbb R}^n)$.
Denote its limit by $h\in L^1_\mathrm{loc}(\mathbb{R}^n)$.
Then, for any $k\in\mathbb{N}\cap[1,n]$,
\begin{align*}
\partial_kh=\lim_{j\to\infty}\rho_{2^{-j}}*\partial_kF
=\lim_{j\to\infty}\rho_{2^{-j}}*F_k=F_k
\end{align*}
in $\mathcal{S}'(\mathbb{R}^n)$. Thus, $h$ and $F$ differ by a constant and
hence $F\in\dot{W}^1\mathcal{M}^p_q(\Omega,w)$. This, combined with
(i), further implies that
\begin{align*}
\|F\|_{\dot{W}^1\mathcal{M}^p_q(\Omega,w)}
\sim\left\|(-\Delta)^\frac12F\right\|_{\mathcal{M}^p_q(\Omega,w)}
=\left\|(-\Delta)^\frac12f\right\|_{\mathcal{M}^p_q(\Omega,w)},
\end{align*}
which completes the proof of Theorem \ref{cor:260117-1}.
\end{proof}

As a direct consequence of Theorem \ref{cor:260117-1} and Proposition
\ref{prop:20260420.1836}, we obtain the following
Sobolev--Morrey embedding theorem; we omit the details here.

\begin{corollary}\label{cor:260605}
Let $1<u\leq s<\infty$, $1<q\leq p<\infty$ with $\frac{s}{u}=\frac{p}{q}$,
$u<q$, and $(n-d-1)(\frac1u-\frac1q)-n(\frac1s-\frac1p)=1$. Then
$\dot{W}^1\mathcal{M}^p_q(\Omega,w)\hookrightarrow\mathcal{M}^u_s(\Omega,w)$.
\end{corollary}

\subsection{Convergence of Integral Averages at Infinity in Lower Critical Case}
\label{subsection4.3}

Functions in $\dot{W}^1\mathcal{M}^p_q(\Omega,w)$ are defined only up to
additive constants, and their pointwise values at infinity are not a priori
meaningful. Nevertheless, when the gradient satisfies suitable Morrey-type
decay, one can recover a canonical representative by considering averages
over expanding balls. In this subsection, we prove that, under a natural
balance condition between the dimension of $\Gamma$ and the integrability
exponents, these averages converge as the radius tends to infinity.

\begin{lemma}\label{lem:250916-1}
Let $1\leq q\leq p<\infty$ satisfy
\begin{align}\label{eq:250915-1}
\frac{n}{p}-\frac{n-d-1}{q}>1
\end{align}
and $w$ be as in \eqref{dw}.
Fix $x_0\in\Gamma$. Then, for any $u\in\dot{W}^1\mathcal{M}^p_q(\Omega,w)$,
\begin{align*}
u^0(x_0):=\lim_{j\to\infty} \fint_{B(x_0,2^j)}u(z)\,dz
\end{align*}
exists and is finite.
\end{lemma}

\begin{proof}
Let $u\in\dot{W}^1\mathcal{M}^p_q(\Omega,w)$. For any
$j\in\mathbb{N}$, define $B_j:=B(x_0,2^j)$. By an argument similar to
that used in \eqref{20260422.2213}, we obtain, for any $j\in\mathbb{N}$,
\begin{align*}
\left|u_{B_j}-u_{B_{j+1}}\right|\lesssim
\fint_{B_{j+1}}\fint_{B_{j+1}}|u(y)-u(z)|\,dy\,dz
\lesssim2^{j(1-\frac{n}{p}+\frac{n-d-1}{q})}
\|u\|_{\dot{W}^1\mathcal{M}^p_q(\Omega,w)},
\end{align*}
which, combined with \eqref{eq:250915-1}, further implies that
$\{u_{B_j}\}_{j\in\mathbb{N}}$ is a Cauchy sequence and hence the limit
$u^0(x_0)$ exists and is finite.
This finishes the proof of Lemma \ref{lem:250916-1}.
\end{proof}

\begin{remark}
We use the same notation as in Lemma \ref{lem:250916-1}. It is worth
noting that Lemma \ref{lem:250916-1} is a generalization of
\cite[Lemma 5.7]{DFM21}. Indeed, under the choice $p=q=2$, our condition
\eqref{eq:250915-1} is consistent with their condition $d>1$ and our
result reduces to theirs. Moreover, according to the proof of
Lemma \ref{lem:250916-1}, we find that $u^0(x_0)$ does not depend
on the specific choice of $x_0\in\Gamma$.
\end{remark}

As was done in \cite[Lemma 5.5]{DFM21}, we give the following definition
motivated by Lemma \ref{lem:250916-1}.
\begin{definition}
Let $1\leq q\leq p<\infty$ satisfy \eqref{eq:250915-1} and
$w$ be as in \eqref{dw}. Define
\begin{align*}
\dot{W}^1\mathcal{M}^p_q(\Omega,w)_0:=\left\{u\in\dot{W}^1
\mathcal{M}^p_q(\Omega,w): u^0\equiv0\mbox{\ on\ }\Gamma\right\}.
\end{align*}
\end{definition}

\subsection{Continuity of Functions in Upper Critical Case}
\label{subsection4.4}

Functions in $\dot{W}^1\mathcal{M}^p_q(\Omega,w)$ are controlled only
through their gradients, and neither continuity nor pointwise behavior is
automatic. In particular, even local regularity inside $\Omega$ must be
recovered from integral information on $\nabla f$.
In this subsection, we first show that, in contrast to the lower critical
condition \eqref{eq:250915-1}, functions in
$\dot{W}^1\mathcal{M}^p_q(\Omega, w)$ possess pointwise traces on $\Gamma$
that satisfy a specific H\"older continuity condition when the condition is
reversed as in \eqref{eq:250915-1b}. We then show that, under a stronger
condition $p\in(n,\infty)$, these functions are continuous away from $\Gamma$
and we further provides a quantitative pointwise estimate that bridges interior
values and boundary traces.

\begin{lemma}\label{lem:260126-11}
Let $1\leq q\leq p<\infty$ satisfy
\begin{align}\label{eq:250915-1b}
\frac{n}{p}-\frac{n-d-1}{q}<1
\end{align}
and $w$ be as in \eqref{dw}.
Then, for any $u\in\dot{W}^1\mathcal{M}^p_q(\Omega,w)$, $x\in\Gamma$,
and $r\in(0,\infty)$,
\begin{align}\label{eq:260125-101}
\int_{B(x,r)}\frac{|\nabla u(\xi)|}{|x-\xi|^{n-1}}\,d\xi
\lesssim r^{1-\frac{n}{p}+\frac{n-d-1}{q}}
\|u\|_{\dot{W}^1\mathcal{M}^p_q(\Omega,w)},
\end{align}
where the implicit positive constant is independent of $u$, $x$, and $r$.
In particular, for any $u\in\dot{W}^1\mathcal{M}^p_q(\Omega,w)$ and
$x\in\Gamma$, $Tu(x)$ exists and, for any $x,y\in\Gamma$,
\begin{align*}
|Tu(x)-Tu(y)|\lesssim|x-y|^{1-\frac{n}{p}+\frac{n-d-1}{q}}
\|u\|_{\dot{W}^1\mathcal{M}^p_q(\Omega,w)},
\end{align*}
where the implicit positive constant is independent of $u$, $x$, and $y$.
\end{lemma}

\begin{proof}
Let $u\in\dot{W}^1\mathcal{M}^p_q(\Omega,w)$. From \eqref{eq:260125-101} and
an argument similar to that used in \eqref{20260422.2213}, it follows
that, for any $x\in\Gamma$, and $r\in(0,\infty)$,
\begin{align*}
\int_{B(x,r)}\frac{|\nabla u(\xi)|}{|x-\xi|^{n-1}}\,d\xi
&=\sum_{j\in\mathbb{N}}\int_{B(x,2^{-j+1}r)\setminus B(x,2^{-j}r)}
\frac{|\nabla u(\xi)|}{|x-\xi|^{n-1}}\,d\xi\\
&\lesssim\sum_{j\in\mathbb{N}}2^{-j}r
\fint_{B(x,2^{-j+1}r)}|\nabla u(\xi)|\,d\xi\\
&\lesssim\sum_{j\in\mathbb{N}}\left(2^{-j}r\right)
^{1-\frac{n}{p}+\frac{n-d-1}{q}}\|u\|_{\dot{W}^1\mathcal{M}^p_q(\Omega,w)}
\lesssim r^{1-\frac{n}{p}+\frac{n-d-1}{q}}
\|u\|_{\dot{W}^1\mathcal{M}^p_q(\Omega,w)}.
\end{align*}
Thus, \eqref{eq:260125-101} holds. Combining \eqref{eq:260125-101} and
Proposition \ref{prop:260126-1}(i), we find that, for any $x\in\Gamma$,
$Tu(x)$ exists. Moreover, using Lemma \ref{DFM21}(i), we find that,
for any $x\in\Gamma$ and $r\in(0,\infty)$,
\begin{align}\label{20260427.2156}
\left|Tu(x)-u_{B(x,r)}\right|&\leq\sum_{j\in\mathbb{N}}
\left|u_{B(x,2^{-j}r)}-u_{B(x,2^{1-j}r)}\right|\notag\\
&\lesssim\sum_{j\in\mathbb{N}}\fint_{B(x,2^{1-j}r)}\fint_{B(x,2^{1-j}r)}
|u(y)-u(z)|\,dy\,dz\notag\\
&\lesssim\sum_{j\in\mathbb{N}}2^{-j}r\fint_{B(x,2^{1-j}r)}
|\nabla u(\xi)|\,d\xi\notag\\
&\sim\int_{B(x,r)}|\nabla u(\xi)|\left[\sum_{j\in\mathbb{N}}
\left(2^{-j}r\right)^{1-n}\boldsymbol{1}_{B(x,2^{1-j}r)}(\xi)\right]\,d\xi.
\end{align}
To proceed, note that, for any given $\xi\in\mathbb{R}^n$,
$\boldsymbol{1}_{B(x,2^{1-j}r)}(\xi)=1$ if and only if
$j<\log_2\frac{2r}{|\xi-x|}$. Define
$J_\xi:=\lceil\log_2\frac{2r}{|\xi-x|}\rceil-1$. Then
\begin{align*}
\sum_{j\in\mathbb{N}}\left(2^{-j}r\right)^{1-n}
\boldsymbol{1}_{B(x,2^{1-j}r)}(\xi)&=r^{1-n}\sum_{j=1}^{J_\xi}2^{(n-1)j}
\sim r^{1-n}2^{(n-1)J_\xi}\\
&\sim r^{1-n}\left(\frac{r}{|\xi-x|}\right)^{n-1}
=\frac{1}{|\xi-x|^{n-1}}.
\end{align*}
From this, \eqref{20260427.2156}, and \eqref{eq:260125-101},
we deduce that, for any $x\in\Gamma$ and $r\in(0,\infty)$,
\begin{align}\label{20250427.2213}
\left|Tu(x)-u_{B(x,r)}\right|
\lesssim\int_{B(x,r)}\frac{|\nabla u(\xi)|}{|\xi-x|^{n-1}}\,d\xi
\lesssim r^{1-\frac{n}{p}+\frac{n-d-1}{q}}
\|u\|_{\dot{W}^1\mathcal{M}^p_q(\Omega,w)}.
\end{align}

Let $x,y\in\Gamma$ be such that $x\neq y$ and define $R:=|x-y|$.
Observe that $B(y,R)\subset B(x,2R)$. By this, Lemma \ref{DFM21}(i),
and \eqref{eq:260125-101}, we obtain
\begin{align*}
\left|u_{B(x,R)}-u_{B(y,R)}\right|&\lesssim\fint_{B(x,2R)}\fint_{B(x,2R)}
|u(z)-u(\xi)|\,dz\,d\xi\\
&\lesssim R\fint_{B(x,2R)}|\nabla u(\xi)|\,d\xi
\sim R^{1-n}\int_{B(x,2R)}|\nabla u(\xi)|\,d\xi\\
&\lesssim\int_{B(x,2R)}\frac{|\nabla u(\xi)|}{|\xi-x|^{n-1}}\,d\xi
\lesssim R^{1-\frac{n}{p}+\frac{n-d-1}{q}}
\|u\|_{\dot{W}^1\mathcal{M}^p_q(\Omega,w)}.
\end{align*}
Combining this, Proposition \ref{prop:260126-1}(iii), and
\eqref{20250427.2213}, we conclude that
\begin{align*}
|Tu(x)-Tu(y)|&\leq\left|Tu(x)-u_{B(x,R)}\right|+
\left|u_{B(x,R)}-u_{B(y,R)}\right|+\left|Tu(y)-u_{B(y,R)}\right|\\
&\lesssim R^{1-\frac{n}{p}+\frac{n-d-1}{q}}
\|u\|_{\dot{W}^1\mathcal{M}^p_q(\Omega,w)}
=|x-y|^{1-\frac{n}{p}+\frac{n-d-1}{q}}
\|u\|_{\dot{W}^1\mathcal{M}^p_q(\Omega,w)}.
\end{align*}
This finishes the proof of Lemma \ref{lem:260126-11}.
\end{proof}

Under the stronger condition $p\in(n,\infty)$, functions in
$\dot{W}^1\mathcal{M}^p_q(\Omega,w)$ are continuous far from $\Gamma$.

\begin{lemma}\label{lem:local-continuity}
Let $1\leq q\leq p<\infty$ satisfy $p\in(n,\infty)$ and
$w$ be as in \eqref{dw}. Then, for any
$u\in\dot{W}^1\mathcal{M}^p_q(\Omega,w)$ and any Lebesgue points
$x,y\in\Omega$ such that $B(x,4|x-y|)\subset\Omega$,
\begin{align*}
|u(x)-u(y)|\lesssim|x-y|^{1-\frac{n}{p}}[w(x)]^{-\frac1q}
\|u\|_{\dot{W}^1\mathcal{M}^p_q(\Omega,w)},
\end{align*}
where the implicit positive constant is independent of $u$, $x$, and $y$.
\end{lemma}

\begin{proof}
Let $u\in\dot{W}^1\mathcal{M}^p_q(\Omega,w)$ and $x,y\in\Omega$ be Lebesgue
points of $u$ such that $x\neq y$ and $B(x,4|x-y|)\subset\Omega$.
Let $R:=|x-y|$. From Lemma \ref{DFM21}(i), (iii) and (i) of Lemma
\ref{lem:250104-11}, H\"older's inequality, and $p\in(n,\infty)$, we infer that,
\begin{align}\label{20260428.1415}
\left|u(x)-u_{B(x,R)}\right|&\leq\sum_{j\in\mathbb{N}}
\left|u_{B(x,2^{-j}R)}-u_{B(x,2^{1-j}R)}\right|\notag\\
&\lesssim\sum_{j\in\mathbb{N}}\fint_{B(x,2^{1-j}R)}\fint_{B(x,2^{1-j}R)}
|u(y)-u(z)|\,dy\,dz\notag\\
&\lesssim\sum_{j\in\mathbb{N}}2^{-j}R\fint_{B(x,2^{1-j}R)}
|\nabla u(z)|\,dz\notag\\
&\lesssim\sum_{j\in\mathbb{N}}\frac{2^{-j}R}{w(B(x,2^{1-j}R))}
\int_{B(x,2^{1-j}R)}|\nabla u(z)|w(z)\,dz\notag\\
&\lesssim\sum_{j\in\mathbb{N}}2^{-j}R\left[\frac{1}{w(B(x,2^{1-j}R))}
\int_{B(x,2^{1-j}R)}|\nabla u(z)|^qw(z)\,dz\right]^\frac1q\notag\\
&\leq\sum_{j\in\mathbb{N}}\frac{2^{-j}R|B(x,2^{1-j}R)|^{\frac1q-\frac1p}}
{[w(B(x,2^{1-j}R))]^\frac1q}
\|u\|_{\dot{W}^1\mathcal{M}^p_q(\Omega,w)}\notag\\
&\sim\sum_{j\in\mathbb{N}}\frac{(2^{-j}R)^{1+\frac{n}{q}-\frac{n}{p}}}
{(2^{-j}R)^\frac{n}{q}}[w(x)]^{-\frac1q}
\|u\|_{\dot{W}^1\mathcal{M}^p_q(\Omega,w)}\notag\\
&=\sum_{j\in\mathbb{N}}\left(2^{-j}R\right)^{1-\frac{n}{p}}[w(x)]^{-\frac1q}
\|u\|_{\dot{W}^1\mathcal{M}^p_q(\Omega,w)}\notag\\
&\sim |x-y|^{1-\frac{n}{p}}[w(x)]^{-\frac1q}
\|u\|_{\dot{W}^1\mathcal{M}^p_q(\Omega,w)}.
\end{align}
In addition, since $y\in B(x,2R)$ and $\delta(x)\geq4R$, it follows that
$\delta(y)\leq\delta(x)+|x-y|\leq\frac32\delta(x)$ and hence
$[w(y)]^{-\frac1q}\lesssim[w(x)]^{-\frac1q}$. This, together with the
observations that $B(y,2R)\subset B(x,4R)\subset\Omega$ and $B(y,R)\subset
B(x,2R)\subset B(x,4R)\subset\Omega$ and an argument similar to that used
in \eqref{20260428.1415}, further implies that
\begin{align*}
\left|u(y)-u_{B(y,R)}\right|&\leq\sum_{j\in\mathbb{N}}
\left|u_{B(y,2^{-j}R)}-u_{B(y,2^{1-j}R)}\right|\\
&\lesssim |x-y|^{1-\frac{n}{p}}[w(y)]^{-\frac1q}
\|u\|_{\dot{W}^1\mathcal{M}^p_q(\Omega,w)}
\lesssim |x-y|^{1-\frac{n}{p}}[w(x)]^{-\frac1q}
\|u\|_{\dot{W}^1\mathcal{M}^p_q(\Omega,w)}
\end{align*}
and
\begin{align*}
\left|u_{B(x,R)}-u_{B(y,R)}\right|\lesssim
\fint_{B(x,2R)}\fint_{B(x,2R)}|u(z)-u(\xi)|\,dz\,d\xi
\lesssim |x-y|^{1-\frac{n}{p}}[w(x)]^{-\frac1q}
\|u\|_{\dot{W}^1\mathcal{M}^p_q(\Omega,w)}.
\end{align*}
Combining these and \eqref{20260428.1415}, we conclude that
\begin{align*}
|u(x)-u(y)|&\leq\left|u(x)-u_{B(x,R)}\right|+\left|u_{B(x,R)}
-u_{B(y,R)}\right|+\left|u_{B(y,R)}-u(y)\right|\notag\\
&\leq\sum_{j\in\mathbb{N}}
\left|u_{B(x,2^{-j}R)}-u_{B(x,2^{1-j}R)}\right|+
\left|u_{B(x,R)}-u_{B(y,R)}\right|\\
&\quad+\sum_{j\in\mathbb{N}}
\left|u_{B(y,2^{-j}R)}-u_{B(y,2^{1-j}R)}\right|\\
&\lesssim|x-y|^{1-\frac{n}{p}}[w(x)]^{-\frac1q}
\|u\|_{\dot{W}^1\mathcal{M}^p_q(\Omega,w)}.
\end{align*}

This finishes the proof of Lemma \ref{lem:local-continuity}.
\end{proof}

As a consequence of Lemmas \ref{lem:260126-11} and \ref{lem:local-continuity},
we establish a quantitative pointwise estimate between interior values
and boundary traces of functions in $\dot{W}^1\mathcal{M}^p_q(\Omega,w)$
under the sharp assumption $p\in(n,\infty)$.

\begin{theorem}\label{thm:5.1}
Let $1\leq q\leq p<\infty$ satisfy $p\in(n,\infty)$, and let
$w$ be as in \eqref{dw}. Then, for any
$u\in\dot{W}^1\mathcal{M}^p_q(\Omega,w)$, any Lebesgue point $x\in\Omega$,
and any $y\in\Gamma$,
\begin{align}\label{20260630.1334}
|u(x)-Tu(y)|\lesssim|x-y|^{1-\frac{n}{p}+\frac{n-d-1}{q}}
\|u\|_{\dot{W}^1\mathcal{M}^p_q(\Omega,w)}+
|x-y|^{1-\frac{n}{p}}[w(x)]^{-\frac{1}{q}}
\|u\|_{\dot{W}^1\mathcal{M}^p_q(\Omega,w)},
\end{align}
where the implicit positive constant is independent of $u$, $x$, and $y$.
\end{theorem}

\begin{proof}
Let $u\in\dot{W}^1\mathcal{M}^p_q(\Omega,w)$, $x\in\Omega$, and
$y\in\Gamma$. Then there exists $\xi_x\in\Gamma$ such that $\delta(x)=|x-\xi_x|$.
By $p\in(n,\infty)$ and an argument similar to that used in
\eqref{20260428.1415}, we obtain
\begin{align}\label{20260428.1825}
\left|u(x)-u_{B(x,\frac{\delta(x)}{4})}\right|\lesssim[\delta(x)]
^{1-\frac{n}{p}}[w(x)]^{-\frac1q}\|u\|_{\dot{W}^1\mathcal{M}^p_q(\Omega,w)}
\leq|x-y|^{1-\frac{n}{p}}[w(x)]^{-\frac1q}
\|u\|_{\dot{W}^1\mathcal{M}^p_q(\Omega,w)}.
\end{align}
In addition, from $p\in(n,\infty)$, the observation that
$B(x,\frac{\delta(x)}{4})\subset B(\xi_x,\frac{5\delta(x)}{4})$, and an
argument similar to that used in \eqref{20260422.2213}, we deduce that
\begin{align}\label{20260428.1826}
\left|u_{B(x,\frac{\delta(x)}{4})}-u_{B(\xi_x,\frac{\delta(x)}{2})}\right|
&\lesssim\fint_{B(\xi_x,\frac{5\delta(x)}{4})}
\fint_{B(\xi_x,\frac{5\delta(x)}{4})}|u(z)-u(\eta)|\,dz\,d\eta\notag\\
&\lesssim[\delta(x)]^{1-\frac{n}{p}+\frac{n-d-1}{q}}
\|u\|_{\dot{W}^1\mathcal{M}^p_q(\Omega, w)}
\lesssim|x-y|^{1-\frac{n}{p}+\frac{n-d-1}{q}}
\|u\|_{\dot{W}^1\mathcal{M}^p_q(\Omega, w)}.
\end{align}
Moreover, note that $p\in(n,\infty)$ implies that \eqref{eq:250915-1b} holds. This,
together with \eqref{20250427.2213}, further implies that
\begin{align}\label{20260428.1827}
\left|u_{B(\xi_x,\frac{\delta(x)}{2})}-Tu(\xi_x)\right|\lesssim
[\delta(x)]^{1-\frac{n}{p}+\frac{n-d-1}{q}}
\|u\|_{\dot{W}^1\mathcal{M}^p_q(\Omega, w)}
\lesssim|x-y|^{1-\frac{n}{p}+\frac{n-d-1}{q}}
\|u\|_{\dot{W}^1\mathcal{M}^p_q(\Omega, w)}.
\end{align}
Furthermore, using \eqref{eq:250915-1b}, the fact that
$|\xi_x-y|\leq|\xi_x-x|+|x-y|\leq2|x-y|$, and Lemma \ref{lem:260126-11},
we obtain
\begin{align*}
|Tu(\xi_x)-Tu(y)|\lesssim|\xi_x-y|^{1-\frac{n}{p}+\frac{n-d-1}{q}}
\|u\|_{\dot{W}^1\mathcal{M}^p_q(\Omega, w)}
\lesssim|x-y|^{1-\frac{n}{p}+\frac{n-d-1}{q}}
\|u\|_{\dot{W}^1\mathcal{M}^p_q(\Omega, w)}.
\end{align*}
Finally, combining this, \eqref{20260428.1825}, \eqref{20260428.1826},
and \eqref{20260428.1827}, we conclude that
\begin{align*}
|u(x)-Tu(y)|&\leq\left|u(x)-u_{B(x,\frac{\delta(x)}{4})}\right|+
\left|u_{B(x,\frac{\delta(x)}{4})}-u_{B(\xi_x,\frac{\delta(x)}{2})}\right|\\
&\quad+\left|u_{B(\xi_x,\frac{\delta(x)}{2})}-Tu(\xi_x)\right|+
|Tu(\xi_x)-Tu(y)|\\
&\lesssim|x-y|^{1-\frac{n}{p}+\frac{n-d-1}{q}}
\|u\|_{\dot{W}^1\mathcal{M}^p_q(\Omega,w)}+
|x-y|^{1-\frac{n}{p}}[w(x)]^{-\frac{1}{q}}
\|u\|_{\dot{W}^1\mathcal{M}^p_q(\Omega,w)}.
\end{align*}
This finishes the proof of Theorem \ref{thm:5.1}.
\end{proof}

\begin{remark}
We use the same notation as in Theorem \ref{thm:5.1}. We point out that
the assumption $p>n$ in Theorem \ref{thm:5.1} is sharp. Indeed, assume that
$1\leq q\leq p\leq n$. Fix $x_0\in\Omega$ and $r_0\in(0,\infty)$ such that
$B(x_0,4r_0)\subset\Omega$. Let $\eta\in C_\mathrm{c}^\infty
(\mathbb{R}^n)$ satisfy
\begin{align*}
\begin{cases}
\eta(x)=1 &\mbox{if\ }x\in B(x_0,r_0),\\
\eta(x)\in[0,1] &\mbox{if\ }x\in B(x_0,2r_0)\setminus B(x_0,r_0),\\
\eta(x)=0 &\mbox{if\ }x\in\mathbb{R}^n\setminus B(x_0,2r_0).
\end{cases}
\end{align*}
Since $B(x_0,4r_0)\subset\Omega$, it follows that, for any $z\in B(x_0,2r_0)$,
$w(z)\sim w(x_0)$. To construct a function
$u\in\dot{W}^1\mathcal{M}^p_q(\Omega,w)$ for which \eqref{20260630.1334}
fails when $p\le n$, we consider the following two cases for $p$.

\emph{Case (1)} $p<n$. In this case, fix $a\in(1,\frac{n}{p})$ and, for
any $x\in\mathbb{R}^n$, define $u(x):=\frac{\eta(x)}{|x-x_0|^{a-1}}$.
By a simple calculation, we find that $u\in L_\mathrm{loc}^1(\mathbb{R}^n)$
and, for almost every $x\in\mathbb{R}^n$,
\begin{align*}
\nabla u(x)=\frac{\nabla\eta(x)}{|x-x_0|^{a-1}}
+(1-a)\frac{\eta(x)(x-x_0)}{|x-x_0|^{a+1}}.
\end{align*}
Therefore, for almost every $x\in\mathbb{R}^n$,
\begin{align*}
|\nabla u(x)|\lesssim\boldsymbol{1}_{B(x_0,2r_0)\setminus B(x_0,r_0)}(x)+
\frac{\boldsymbol{1}_{B(x_0,2r_0)}(x)}{|x-x_0|^a}.
\end{align*}
This, together with the assumption that $a\in(1,\frac{n}{p})$,
further implies $\|\nabla u\|_{\mathcal{M}^p_q(\Omega,w)}<\infty$; i.e.,
$u\in\dot{W}^1\mathcal{M}^p_q(\Omega,w)$.

\emph{Case (2)} $p=n$. In this case, for any $x\in\mathbb{R}^n$, define
$u(x):=\eta(x)\ln\ln\frac{e^er_0}{|x-x_0|}$. From a simple calculation, we
deduce that $u\in L_\mathrm{loc}^1(\mathbb{R}^n)$ and, for almost every
$x\in\mathbb{R}^n$,
\begin{align*}
\nabla u(x)=\nabla\eta(x)\ln\ln\frac{e^er_0}{|x-x_0|}-
\frac{\eta(x)(x-x_0)}{|x-x_0|^2\ln(\frac{e^er_0}{|x-x_0})}.
\end{align*}
Thus, for almost every $x\in\mathbb{R}^n$,
\begin{align*}
|\nabla u(x)|\lesssim\boldsymbol{1}_{B(x_0,2r_0)\setminus B(x_0,r_0)}(x)+
\frac{\boldsymbol{1}_{B(x_0,2r_0)}(x)}{|x-x_0|\ln(\frac{e^er_0}{|x-x_0})},
\end{align*}
which further implies that $u\in\dot{W}^1\mathcal{M}^p_q(\Omega,w)$.

Next, we verify that the above constructed $u$ does not
satisfy \eqref{20260630.1334}. To this end, fix $\sigma\in\mathbb{S}^{n-1}$
and $y\in\Gamma$. For any $j\in\mathbb{N}$, let
$x_j:=x_0+\frac{r_0}{2j}\sigma$. Then, in both cases, the points
$\{x_j\}_{j\in\mathbb{N}}$ are Lebesgue points of $u$ in $\Omega$ and
$Tu(y)=0$. Note that $u(x_j)\to\infty$ as $j\to\infty$.
However, for any $j\in\mathbb{N}$,
\begin{align*}
&\left|x_j-y\right|^{1-\frac{n}{p}+\frac{n-d-1}{q}}
\|u\|_{\dot{W}^1\mathcal{M}^p_q(\Omega,w)}+
\left|x_j-y\right|^{1-\frac{n}{p}}\left[w\left(x_j\right)\right]^{-\frac{1}{q}}
\|u\|_{\dot{W}^1\mathcal{M}^p_q(\Omega,w)}\\
&\sim|x_0-y|^{1-\frac{n}{p}+\frac{n-d-1}{q}}
\|u\|_{\dot{W}^1\mathcal{M}^p_q(\Omega,w)}+
|x_0-y|^{1-\frac{n}{p}}[w(x_0)]^{-\frac{1}{q}}
\|u\|_{\dot{W}^1\mathcal{M}^p_q(\Omega,w)}
\end{align*}
is uniformly bounded, which means that \eqref{20260630.1334} fails for the
above constructed $u$. This indicates that the condition $p>n$
in Theorem \ref{thm:5.1} is sharp.
\end{remark}

\subsection{Trace Operators}
\label{subsection4.5}

In this subsection, we characterize the trace space of
$\dot{W}^1\mathcal{M}^p_q(\Omega,w)$, denoted by $Q_q^p(\Gamma)$, and establish
the bounded mapping property of the trace operator $T$ from
$\dot{W}^1\mathcal{M}^p_q(\Omega,w)$ into this newly defined space. We begin by
ensuring the existence of traces for functions in
$\dot{W}^1\mathcal{M}^p_q(\Omega,w)$, which follows directly from Proposition
\ref{prop:260126-1}(ii) and Proposition \ref{lem:250919-2};
we omit the details here.

\begin{lemma}\label{lem:5.1}
Let $1\leq q\leq p<\infty$ and $w$ be as in \eqref{dw}. Then, for any
$u\in\dot{W}^1\mathcal{M}^p_q(\Omega,w)$, $Tu(x)$ exists for
$\mathcal{H}^d$-almost every $x\in\Gamma$.
\end{lemma}

To precisely describe the regularity of these boundary
traces, we introduce the trace space $Q_q^p(\Gamma)$, inspired by the framework
of fractional Morrey-type spaces developed in \cite{Xiao08}.

\begin{definition}
Let $1\leq q\leq p<\infty$. The space $Q^p_q(\Gamma)$ is defined
to be the set of all $\mathcal{H}^d$-measurable functions $f$ on
$\Gamma$ such that
\begin{align*}
\|f\|_{Q^p_q(\Gamma)}:=\sup_{v\in\Gamma,R\in(0,\infty)}
|B(v,R)|^{\frac1p-\frac1q}\left[\iint_{\Gamma(v,R)\times\Gamma(v,R)}
\frac{|f(x)-f(y)|^q}{|x-y|^{q+d-1}}\,d\mathcal{H}^d(x)\,d\mathcal{H}^d(y)
\right]^\frac1q<\infty.
\end{align*}
\end{definition}

The following theorem is the main result of this subsection. By generalizing
the trace estimates in \cite[Theorem 3.4]{DFM21} from the weighted Sobolev
space $\dot{W}^{1,2}(\Omega,w)$ to the weighted Sobolev--Morrey
space $\dot{W}^1\mathcal{M}^p_q(\Omega,w)$, we demonstrate
that the trace operator $T$ is indeed a bounded linear
mapping from $\dot{W}^1\mathcal{M}^p_q(\Omega,w)$ to $Q^p_q(\Gamma)$.

\begin{theorem}\label{thm:local-trace}
Let $1<q\leq p<\infty$ and $w$ be as in \eqref{dw}.
Then the trace operator $T$ can be extended to a linear bounded mapping
$T:\dot{W}^1\mathcal{M}^p_q(\Omega,w)\longrightarrow Q^p_q(\Gamma)$.
\end{theorem}

\begin{proof}
Let $u\in\dot{W}^1\mathcal{M}^p_q(\Omega,w)$.
By Lemma \ref{lem:5.1}, we find that $Tu(x)$ exists for
$\mathcal{H}^d$-almost every $x\in\Gamma$. Thus, to prove Theorem
\ref{thm:local-trace}, it suffices to show that, for any given $v\in\Gamma$ and
$R\in(0,\infty)$,
\begin{align}\label{eq:local-trace}
|B(v,R)|^{\frac1p-\frac1q}\left[\iint_{\Gamma(v,R)\times\Gamma(v,R)}
\frac{|Tu(x)-Tu(y)|^q}{|x-y|^{q+d-1}}\,d\mathcal{H}^d(x)\,d\mathcal{H}^d(y)
\right]^\frac1q\lesssim\|u\|_{\dot{W}^1\mathcal{M}^p_q(\Omega,w)},
\end{align}
where the implicit positive constant is independent of $u$, $v$, and $R$.
Furthermore, to prove \eqref{eq:local-trace}, it remains to show that, for any
given $r\in(0,\infty)$,
\begin{align}\label{20260501.1900}
\int_{\Gamma(v,R)}\int_{\Gamma(v,R)\setminus B(y,r)}
\frac{|u_{B(x,r)}-u_{B(y,r)}|^q}{|x-y|^{q+d-1}}
\,d\mathcal{H}^d(x)\,d\mathcal{H}^d(y)\lesssim
\int_{B(v,11R)}|\nabla u(\xi)|^qw(\xi)\,d\xi
\end{align}
with the implicit positive constant independent of $r$. Indeed, if
\eqref{20260501.1900} holds, then, from Proposition \ref{prop:260126-1}(i),
Fatou's lemma, and \eqref{20260501.1900}, we infer that
\begin{align*}
&|B(v,R)|^{\frac1p-\frac1q}\left[\iint_{\Gamma(v,R)\times\Gamma(v,R)}
\frac{|Tu(x)-Tu(y)|^q}{|x-y|^{q+d-1}}\,d\mathcal{H}^d(x)\,d\mathcal{H}^d(y)
\right]^\frac1q\\
&\quad=|B(v,R)|^{\frac1p-\frac1q}\left[\int_{\Gamma(v,R)}\int_{\Gamma(v,R)}
\lim_{r\to0^+}\frac{|u_{B(x,r)}-u_{B(y,r)}|^q}{|x-y|^{q+d-1}}
\boldsymbol{1}_{\Gamma(v,R)\setminus B(y,r)}(x)
\,d\mathcal{H}^d(x)\,d\mathcal{H}^d(y)\right]^\frac1q\\
&\quad\leq\lim_{r\to0^+}|B(v,R)|^{\frac1p-\frac1q}
\left[\int_{\Gamma(v,R)}\int_{\Gamma(v,R)\setminus B(y,r)}
\frac{|u_{B(x,r)}-u_{B(y,r)}|^q}{|x-y|^{q+d-1}}
\,d\mathcal{H}^d(x)\,d\mathcal{H}^d(y)\right]^\frac1q\\
&\quad\lesssim|B(v,11R)|^{\frac1p-\frac1q}
\left[\int_{B(v,11R)}|\nabla u(\xi)|^qw(\xi)\,d\xi\right]^\frac1q
\leq\|u\|_{\dot{W}^1\mathcal{M}^p_q(\Omega,w)},
\end{align*}
and hence \eqref{eq:local-trace} holds.

Now, we prove \eqref{20260501.1900}. Fix $r\in(0,\infty)$ and
$a\in(0,\min\{\frac{1}{q'},\,\frac{n-d-1}{q}\})$. Let
$k\in\mathbb{Z}_+$ and $x,y\in\Gamma(v,R)$ be such that
$2^kr\leq|x-y|<2^{k+1}r$. Note that, for any $z\in B(x,r)$ and $\xi\in B(y,2^kr)$,
\begin{align*}
|\xi-z|\leq|\xi-y|+|y-x|+|x-z|<2^kr+2^{k+1}r+r\leq2^{k+2}r.
\end{align*}
Therefore, for any $z\in B(x,r)$, $B(y,2^kr)\subset B(z,2^{k+2}r)$.
Using this and Lemma \ref{DFM21}(ii), we obtain, for any $z\in B(x,r)$,
\begin{align*}
\left|u_{B(x,r)}-u_{B(y,2^kr)}\right|&\leq\fint_{B(x,r)}\fint_{B(y,2^kr)}
|u(z)-u(\xi)|\,d\xi\,dz\\
&\lesssim\fint_{B(x,r)}\fint_{B(z,2^{k+2}r)}|u(z)-u(\xi)|\,d\xi\,dz\\
&\lesssim\left(2^kr\right)^n\fint_{B(x,r)}\fint_{B(z,2^{k+2}r)}
\frac{|\nabla u(\xi)|}{|\xi-z|^{n-1}}\,d\xi\,dz.
\end{align*}
This, combined with H\"older's inequality and the assumption that
$a<\frac{1}{q'}$, further implies that
\begin{align}\label{20260501.2223}
\left|u_{B(x,r)}-u_{B(y,2^kr)}\right|^q&\lesssim\left(2^kr\right)^{nq}
\fint_{B(x,r)}\fint_{B(z,2^{k+2}r)}
\frac{|\nabla u(\xi)|^q}{|\xi-z|^{n-1-aq}}\,d\xi\,dz\notag\\
&\quad\times\left[\fint_{B(x,r)}\fint_{B(z,2^{k+2}r)}
\frac{1}{|\xi-z|^{n-1+aq'}}\,d\xi\,dz\right]^{q-1}\notag\\
&\sim\left(2^kr\right)^{n+q-1-aq}\fint_{B(x,r)}\fint_{B(z,2^{k+2}r)}
\frac{|\nabla u(\xi)|^q}{|\xi-z|^{n-1-aq}}\,d\xi\,dz.
\end{align}
By an argument similar to that used in \eqref{20260501.2223}, we find that
\begin{align*}
\left|u_{B(y,r)}-u_{B(y,2^kr)}\right|^q\lesssim
\left(2^kr\right)^{n+q-1-aq}\fint_{B(y,r)}\fint_{B(z,2^{k+2}r)}
\frac{|\nabla u(\xi)|^q}{|\xi-z|^{n-1-aq}}\,d\xi\,dz.
\end{align*}
Thus, for any $k\in\mathbb{Z}_+$ and $x,y\in\Gamma(v,R)$ satisfying
$2^kr\leq|x-y|<2^{k+1}r$,
\begin{align}\label{20260501.2234}
\left|u_{B(x,r)}-u_{B(y,r)}\right|^q&\lesssim
\left|u_{B(x,r)}-u_{B(y,2^kr)}\right|^q+
\left|u_{B(y,2^kr)}-u_{B(y,r)}\right|^q\notag\\
&\lesssim\left(2^kr\right)^{n+q-1-aq}\fint_{B(x,r)}\fint_{B(z,2^{k+2}r)}
\frac{|\nabla u(\xi)|^q}{|\xi-z|^{n-1-aq}}\,d\xi\,dz\notag\\
&\quad+\left(2^kr\right)^{n+q-1-aq}\fint_{B(y,r)}\fint_{B(z,2^{k+2}r)}
\frac{|\nabla u(\xi)|^q}{|\xi-z|^{n-1-aq}}\,d\xi\,dz.
\end{align}
In addition, since $\Gamma$ is a $d$-set, it follows that, for any
$y\in\Gamma(v,R)$ and $k\in\mathbb{Z}_+$,
\begin{align*}
\mathcal{H}^d\left(\Gamma(v,R)\cap B\left(y,2^{k+1}r\right)
\setminus B\left(y,2^kr\right)\right)\leq\mathcal{H}^d\left(\Gamma
\cap B\left(y,2^{k+1}r\right)\right)\sim\left(2^kr\right)^d.
\end{align*}
This, together with \eqref{20260501.2234},
further implies that, for any $k\in\mathbb{Z}_+$,
\begin{align}\label{20260502.2013}
&\iint_{\genfrac{}{}{0pt}{}{\Gamma(v,R)\times\Gamma(v,R)}
{2^kr\leq|x-y|<2^{k+1}r}}
\frac{|u_{B(x,r)}-u_{B(y,r)}|^q}{|x-y|^{q+d-1}}
\,d\mathcal{H}^d(x)\,d\mathcal{H}^d(y)\notag\\
&\quad\lesssim\left(2^kr\right)^{-d-aq}r^{-n}
\iint_{\genfrac{}{}{0pt}{}{\Gamma(v,R)\times\Gamma(v,R)}
{2^kr\leq|x-y|<2^{k+1}r}}
\int_{B(x,r)}\int_{B(z,2^{k+2}r)}\notag\\
&\quad\quad\times\frac{|\nabla u(\xi)|^q}{|\xi-z|^{n-1-aq}}\,d\xi\,dz
\,d\mathcal{H}^d(x)\,d\mathcal{H}^d(y)\notag\\
&\quad\sim\left(2^kr\right)^{-aq}r^{-n}
\int_{\Gamma(v,R)}\int_{B(x,r)}\int_{B(z,2^{k+2}r)}
\frac{|\nabla u(\xi)|^q}{|\xi-z|^{n-1-aq}}\,d\xi\,dz\,d\mathcal{H}^d(x).
\end{align}

To proceed, observe that, for any $x,y\in\Gamma(v,R)$, it holds that $|x-y|<2R$.
Therefore, for any $k\in\mathbb{Z}_+$ satisfying
$\{(x,y)\in\Gamma(v,R)\times\Gamma(v,R):2^kr\leq|x-y|<2^{k+1}r\}\neq\emptyset$,
it holds that $2^kr<2R$ and hence $k<\log_2\frac{2R}{r}$. Define
$K:=\lceil\log_2\frac{2R}{r}\rceil-1$. Then, for any $x\in\Gamma(v,R)$, $z\in
B(x,r)$, $k\in\mathbb{Z}_+\cap[0,K]$, and $\xi\in B(z,2^{k+2}r)$, it holds that
\begin{align*}
|\xi-v|\leq|\xi-z|+|z-x|+|x-v|<2^{k+2}r+r+R\leq11R
\end{align*}
and hence $\xi\in B(v,11R)$ and $z\in B(x,r)\cap B(\xi,2^{k+2}r)$. Combining
this, \eqref{20260502.2013}, and Tonelli's theorem, we obtain, for any
$k\in\mathbb{Z}_+\cap[0,K]$,
\begin{align}\label{20260502.2029}
&\iint_{\genfrac{}{}{0pt}{}{\Gamma(v,R)\times\Gamma(v,R)}
{2^kr\leq|x-y|<2^{k+1}r}}
\frac{|u_{B(x,r)}-u_{B(y,r)}|^q}{|x-y|^{q+d-1}}
\,d\mathcal{H}^d(x)\,d\mathcal{H}^d(y)\notag\\
&\quad\lesssim\left(2^kr\right)^{-aq}r^{-n}\int_{B(v,11R)}
|\nabla u(\xi)|^q\int_{\Gamma(v,R)}\int_{B(x,r)\cap B(\xi,2^{k+2}r)}
\frac{1}{|z-\xi|^{n-1-aq}}\,dz\,d\mathcal{H}^d(x)\,d\xi\notag\\
&\quad=\int_{B(v,11R)}|\nabla u(\xi)|^qh_k(\xi)\,d\xi,
\end{align}
where, for any $\xi\in B(v,11R)$,
\begin{align*}
h_k(\xi):=\left(2^kr\right)^{-aq}r^{-n}
\int_{\Gamma(v,R)}\int_{B(x,r)\cap B(\xi,2^{k+2}r)}
\frac{1}{|z-\xi|^{n-1-aq}}\,dz\,d\mathcal{H}^d(x).
\end{align*}
For any $k\in\mathbb{Z}_+\cap[0,K]$ and $\xi\in B(v,11R)$, we estimate
$h_k(\xi)$ by considering the following two cases for $x$.

\emph{Case (1)} $x\in\Gamma(v,R)\setminus B(\xi,2r)$. In this case, for any
$\xi\in B(v,11R)$ and $z\in B(x,r)\cap B(\xi,2^{k+2}r)$,
$|z-\xi|\geq|x-\xi|-|x-z|\geq\frac12|x-\xi|$ and
$|x-\xi|\leq|x-z|+|z-\xi|<r+2^{k+2}r\leq5\cdot2^kr$. This, together with the
assumption that $a<\frac{n-d-1}{q}$, further implies that, for any
$\xi\in\Gamma(v,11R)$,
\begin{align}\label{20260502.2112}
h_k^0(\xi)&:=\left(2^kr\right)^{-aq}r^{-n}\int_{\Gamma(v,R)\setminus B(\xi,2r)}
\int_{B(x,r)\cap B(\xi,2^{k+2}r)}\frac{1}{|z-\xi|^{n-1-aq}}
\,dz\,d\mathcal{H}^d(x)\notag\\
&\lesssim\left(2^kr\right)^{-aq}r^{-n}
\int_{\Gamma(v,R)\setminus B(\xi,2r)}\int_{B(x,r)\cap B(\xi,2^{k+2}r)}
\frac{1}{|x-\xi|^{n-1-aq}}\,dz\,d\mathcal{H}^d(x)\notag\\
&\lesssim\left(2^kr\right)^{-aq}\int_{\Gamma(v,R)\cap B(\xi,5\cdot2^kr)}
\frac{1}{|x-\xi|^{n-1-aq}}\,d\mathcal{H}^d(x).
\end{align}
If $\delta(\xi)\geq5\cdot2^kr$, then $\Gamma(v,R)\cap
B(\xi,5\cdot2^kr)=\emptyset$ and hence $h_k^0(\xi)=0$. Thus, we may assume that
$\delta(\xi)<5\cdot2^kr$. For any $m\in\mathbb{N}$, define
$A_m:=\{x\in\Gamma(v,R):2^m\delta(\xi)\leq|x-\xi|<2^{m+1}\delta(\xi)\}$
and let $A_0:=\{x\in\Gamma(v,R):\delta(\xi)\leq|x-\xi|<2\delta(\xi)\}$.
Since $\Gamma$ is a $d$-set, it follows that, for any $m\in\mathbb{Z}_+$,
\begin{align*}
\mathcal{H}^d(A_m)\leq\mathcal{H}^d\left(\Gamma\cap
B\left(\xi,2^{m+1}\delta(\xi)\right)\right)\sim\left[2^m\delta(\xi)\right]^d.
\end{align*}
Applying this, \eqref{20260502.2112}, and the assumptions that
$\delta(\xi)<5\cdot2^kr$ and $a<\frac{n-d-1}{q}$, we find that
\begin{align}\label{20260502.2118}
h_k^0(\xi)&\lesssim\left(2^kr\right)^{-aq}\sum_{m\in\mathbb{Z}_+}\int_{A_m}
\frac{1}{|x-\xi|^{n-1-aq}}\,d\mathcal{H}^d(x)\notag\\
&\lesssim\left(2^kr\right)^{-aq}\sum_{m\in\mathbb{Z}_+}
\left[2^m\delta(\xi)\right]^{d+1+aq-n}\lesssim
\left(2^kr\right)^{-aq}[\delta(\xi)]^{d+1+aq-n}.
\end{align}

\emph{Case (2)} $x\in\Gamma(v,R)\cap B(\xi,2r)$. In this case, for any $x\in
B(v,11R)$ and $z\in B(x,r)$, $|z-\xi|\leq|z-x|+|x-\xi|<3r$ and hence
$z\in B(\xi,3r)$. In addition, note that, if $\delta(\xi)\geq2r$, then
$\Gamma\cap B(\xi,2r)=\emptyset$. Thus, we may assume that $\delta(\xi)<2r$.
From $\Gamma$ is a $d$-set, we infer that
\begin{align*}
\mathcal{H}^d(\Gamma(v,R)\cap B(\xi,2r))
\leq\mathcal{H}^d(\Gamma\cap B(\xi,2r))\sim r^d.
\end{align*}
This, together with \eqref{20260502.2112}, the proven conclusion that
$z\in B(\xi,3r)$ for any $z\in B(x,r)$ and the assumptions that
$\delta(\xi)<2r$ and $a<\frac{n-d-1}{q}$, further implies that, for any
$\xi\in B(v,11R)$,
\begin{align*}
h_k^1(\xi)&:=\left(2^kr\right)^{-aq}r^{-n}\int_{\Gamma(v,R)\cap B(\xi,2r)}
\int_{B(x,r)\cap B(\xi,2^{k+2}r)}\frac{1}{|z-\xi|^{n-1-aq}}
\,dz\,d\mathcal{H}^d(x)\\
&\lesssim\left(2^kr\right)^{-aq}r^{d-n}\int_{B(\xi,3r)}
\frac{1}{|z-\xi|^{n-1-aq}}\,dz\sim\left(2^kr\right)^{-aq}r^{d+1+aq-n}
\lesssim2^{-kaq}w(\xi).
\end{align*}
Finally, combining this, \eqref{20260502.2029}, \eqref{20260502.2118},
the assumption that $\delta(\xi)<5\cdot2^kr$ in Case (1), we conclude that
\begin{align*}
&\int_{\Gamma(v,R)}\int_{\Gamma(v,R)\setminus B(y,r)}
\frac{|u_{B(x,r)}-u_{B(y,r)}|^q}{|x-y|^{q+d-1}}
\,d\mathcal{H}^d(x)\,d\mathcal{H}^d(y)\\
&\quad=\sum_{k\in\mathbb{Z}_+\cap[0,K]}
\iint_{\genfrac{}{}{0pt}{}{\Gamma(v,R)\times\Gamma(v,R)}
{2^kr\leq|x-y|<2^{k+1}r}}
\frac{|u_{B(x,r)}-u_{B(y,r)}|^q}{|x-y|^{q+d-1}}
\,d\mathcal{H}^d(x)\,d\mathcal{H}^d(y)\\
&\quad\lesssim\int_{B(v,11R)}|\nabla u(\xi)|^q
\sum_{k\in\mathbb{Z}_+\cap[0,K]}h_k(\xi)\,d\xi\\
&\quad=\int_{B(v,11R)}|\nabla u(\xi)|^q
\sum_{k\in\mathbb{Z}_+\cap[0,K]}\left[h_k^0(\xi)+h_k^1(\xi)\right]\,d\xi\\
&\quad\lesssim\int_{B(v,11R)}|\nabla u(\xi)|^q
\left\{\sum_{k\in\mathbb{Z}_+\cap[0,K]}\left(2^kr\right)^{-aq}
[\delta(\xi)]^{d+1+aq-n}\boldsymbol{1}_{\{\xi\in\mathbb{R}^n:
\delta(\xi)<5\cdot2^kr\}}\right.\\
&\quad\quad\left.+\sum_{k\in\mathbb{Z}_+\cap[0,K]}2^{-kaq}w(\xi)\right\}\,d\xi\\
&\quad\lesssim\int_{B(v,11R)}|\nabla u(\xi)|^qw(\xi)\,d\xi
\end{align*}
and hence \eqref{20260501.1900} holds.
This finishes the proof Theorem \ref{thm:local-trace}.
\end{proof}

\begin{remark}
When $p=q=2$, Theorem \ref{thm:local-trace} in this case coincides with
\cite[Theorem 3.4]{DFM21}. Moreover, even in the special case of $p=q\neq2$,
the conclusion of Theorem \ref{thm:local-trace} is new.
\end{remark}

\subsection{Extension Operators}
\label{subsection4.6}

In this subsection, we show the boundedness of the extension operator
$E$ from the trace space $Q^p_q(\Gamma)$ to
$\dot{W}^1\mathcal{M}^p_q(\Omega,w)$. Furthermore, we prove that $E$ is the
right inverse of the trace operator $T$. To introduce the concept of $E$, we
first recall the Whitney decomposition of open sets and the associated
partition of unity constructed in \cite[Theorem 7.5.2 and Lemma
7.5.6]{Grafakos24}. In what follows, for any cube $Q\subset\mathbb{R}^n$ and
$a\in(1,\infty)$, denote by $c_Q$ the center of $Q$, $l(Q)$ its edge length,
and $aQ$ the cube concentric with $Q$ having edge length $al(Q)$. Recall that a
\emph{dyadic cube} $Q_{j,k}$ in $\mathbb{R}^n$ is defined as a set of the form
$Q_{j,k}=[k_12^{-j},(k_1+1)2^{-j})\times\cdots\times [k_n2^{-j},(k_n+1)2^{-j})$,
where $j\in\mathbb{Z}$ and $k:=(k_1,\dots,k_n)\in\mathbb{Z}^n$. Moreover,
for any $j\in\mathbb{Z}$, let $\mathcal{D}_j:=\{Q_{j,k}:k\in\mathbb{Z}^n\}$
be the set of all dyadic cubes with edge length $2^{-j}$ and let
$\mathcal{D}:=\bigcup_{j\in\mathbb{Z}}\mathcal{D}_j$ be the set of all
dyadic cubes in $\mathbb{R}^n$.

\begin{lemma}\label{lem:20260502.2249}
Let $G\subset\mathbb{R}^n$ be a nonempty proper open subset. Then there exist a
countable family $\mathcal{W}:=\{Q\}_{Q\in\mathcal{W}}$ of disjoint dyadic
cubes and a sequence $\{\varphi_Q\}_{Q\in\mathcal{W}}$ of functions
in $C_\mathrm{c}^\infty(\mathbb{R}^n)$ such that the following statements hold.
\begin{enumerate}
\item[\rm(i)] $\bigcup_{Q\in\mathcal{W}}Q=G$.

\item[\rm(ii)] For any $Q\in\mathcal{W}$, $\sqrt{n}l(Q)\leq
\operatorname{dist}(\overline{Q},G^\complement)\leq4\sqrt{n}l(Q)$.

\item[\rm(iii)] For any $Q,R\in\mathcal{W}$, if $\overline{Q}\cap\overline{R}
\neq\emptyset$, then $\frac{l(Q)}{l(R)}\in\{\frac12,\,1,\,2\}$.

\item[\rm(iv)] There exists a positive constant $C$ such that, for any
$Q\in\mathcal{W}$, there exist at most $C$ cubes $R\in\{Q\}_{Q\in\mathcal{W}}$
such that $\overline{R}\cap\overline{Q}\neq\emptyset$.

\item[\rm(v)] For any $Q,R\in\mathcal{W}$, if $\overline{Q}\cap\overline{R}
=\emptyset$, then $\frac98Q\cap\frac98R=\emptyset$. Moreover, for any
$Q\in\mathcal{W}$, $\frac98Q\subset G$ and $\boldsymbol{1}_G\leq
\sum_{Q\in\mathcal{W}}\boldsymbol{1}_{\frac98Q}\leq2^n\boldsymbol{1}_G$.

\item[\rm(vi)] For any $Q\in\mathcal{W}$, $0\leq\varphi_Q\leq1$ and
$\operatorname{supp}(\varphi_Q)\subset\frac98Q$.

\item[\rm(vii)] There exists a positive constant $K$ such that,
for any $Q\in\mathcal{W}$,
\begin{align*}
\left|\nabla\varphi_Q\right|\leq\frac{K}{l(Q)}.
\end{align*}

\item[\rm(viii)] The family $\{\varphi_Q\}_{Q\in\mathcal{W}}$ forms a partition
of $G$; i.e., $\sum_{Q\in\mathcal{W}}\varphi_Q=\boldsymbol{1}_G$.
\end{enumerate}
\end{lemma}

Applying the Whitney decomposition in Lemma
\ref{lem:20260502.2249} to $\Omega$, we obtain a countable
family $\mathcal{W}_\Omega:=\{Q\}_{Q\in\mathcal{W}_\Omega}$ of disjoint dyadic
cubes and the corresponding sequence $\{\varphi_Q\}_{Q\in\mathcal{W}_\Omega}$
of functions in $C_\mathrm{c}^\infty(\mathbb{R}^n)$.
For any cube $Q\subset\mathbb{R}^n$, let
$\delta(Q):=\operatorname{dist}(Q,\Gamma)$. Since
$\Gamma$ is closed, it follows that, for any $Q\in\mathcal{W}_\Omega$, there
exists a point $\xi_Q\in\Gamma$ such that
$\operatorname{dist}(\xi_Q,Q)=\delta(Q)$.
The \emph{extension operator} $E$ is then defined by setting, for any function
$g\in L^1_\mathrm{loc}(\Gamma,\mathcal{H}^d)$ and $x\in\mathbb{R}^n$,
\begin{align*}
Eg(x):=\sum_{Q\in\mathcal{W}_\Omega}y_Q\varphi_Q(x),
\end{align*}
where, for any $Q\in\mathcal{W}_\Omega$,
\begin{align*}
y_Q:=\frac{1}{\mathcal{H}^d(\Gamma(\xi_Q,\delta(Q)))}
\int_{\Gamma(\xi_Q,\delta(Q))}g(z)\,d\mathcal{H}^d(z).
\end{align*}
Note that, for any $x\in\Omega$, the summation $Eg(x)$ is indeed a finite
sum. Therefore, $Eg$ is infinitely differentiable on $\Omega$.

Now, we present the main result of this subsection, which is an
extension of \cite[Theorem 7.3]{DFM21} from the weighted Sobolev
space $\dot{W}^{1,2}(\Omega,w)$ to the weighted Sobolev--Morrey space $\dot{W}^1\mathcal{M}^p_q(\Omega,w)$.

\begin{theorem}\label{thm:Eg-membership}
Let $1<q\leq p<\infty$ and $w$ be as in \eqref{dw}.
Then the extension operator $E$ is bounded
from $Q^p_q(\Gamma)$ to $\dot{W}^1\mathcal{M}^p_q(\Omega,w)$.
Moreover, for any $g\in Q^p_q(\Gamma)$ and $\mathcal{H}^d$-almost every
$x\in\Gamma$, $(T\circ E)g(x)=g(x)$.
\end{theorem}

\begin{proof}
Let $g\in Q^p_q(\Gamma)$. To prove the boundedness of $E$, it suffices to
show that, for any $v\in\Omega$ and $r\in(0,\infty)$,
\begin{align}\label{20260503.2213}
|B(v,r)|^{\frac1p-\frac1q}\left[\int_{B(v,r)}|\nabla Eg(x)|^qw(x)\,dx
\right]^\frac1q\lesssim\|g\|_{Q^p_q(\Gamma)},
\end{align}
where the implicit positive constant is independent of $g$, $v$, and $r$.
To this end, from Lemmas \ref{lem:20260502.2249} and \ref{lem:250104-11}(i),
we deduce that, for any $R\in\mathcal{W}_\Omega$,
\begin{align*}
\nabla Eg=\sum_{Q\in\mathcal{W}_\Omega}(y_Q-y_R)\nabla\varphi_Q
\end{align*}
almost everywhere in $\mathbb{R}^n$, and hence
\begin{align}\label{20260506.1938}
\int_{B(v,r)}|\nabla Eg(x)|^qw(x)\,dx&\lesssim
\sum_{R\in\mathcal{W}_\Omega}\sum_{Q\in\mathcal{W}_\Omega}
\int_{B(v,r)\cap R}\left|(y_Q-y_R)\nabla\varphi_Q(x)\right|^qw(x)\,dx\notag\\
&\lesssim\sum_{R\in\mathcal{W}_\Omega}\sum_{\genfrac{}{}{0pt}{}
{Q\in\mathcal{W}_\Omega}{\frac98Q\cap R\cap B(v,r)\neq\emptyset}}
\frac{1}{[l(Q)]^q}\left|y_Q-y_R\right|^qw(B(v,r)\cap R)\notag\\
&\lesssim\sum_{R\in\mathcal{W}_\Omega}\sum_{\genfrac{}{}{0pt}{}
{Q\in\mathcal{W}_\Omega}{\frac98Q\cap R\cap B(v,r)\neq\emptyset}}
\frac{1}{[\delta(R)]^q}\left|y_Q-y_R\right|^qw(B(v,r)\cap R).
\end{align}
To proceed, we estimate $|y_Q-y_R|^q$ for any $Q,R\in\mathcal{W}_\Omega$ with
$\frac98Q\cap R\neq\emptyset$. To this end, we first claim that, for any
$Q,R\in\mathcal{W}_\Omega$ such that $\frac98Q\cap R\neq\emptyset$,
\begin{align}\label{20260503.2253}
\Gamma(\xi_Q,\delta(Q))\subset\Gamma(\xi_R,100\delta(R)).
\end{align}
Indeed, for any $y\in\Gamma(\xi_Q,\delta(Q))$, $z_Q\in Q$, $z_R\in R$,
and $\xi\in\frac98Q\cap R$,
\begin{align*}
|y-\xi_R|&\leq|y-\xi_Q|+|\xi_Q-\xi_R|<\delta(Q)+|\xi_Q-z_Q|+|z_Q-\xi|
+|\xi-z_R|+|z_R-\xi_R|\\
&\leq\delta(Q)+|\xi_Q-z_Q|+\frac98\sqrt{n}l(Q)+\sqrt{n}l(R)+|z_R-\xi_R|.
\end{align*}
Taking the infimum over $z_Q\in Q$ and $z_R\in R$ and using
Lemma \ref{lem:20260502.2249}, we obtain
\begin{align*}
|y-\xi_R|&\leq\delta(Q)+\delta(Q)+\frac98\sqrt{n}l(Q)+\sqrt{n}l(R)+\delta(R)\\
&\leq\left(9+\frac98\right)\sqrt{n}l(Q)+\sqrt{n}l(R)+\delta(R)\\
&\leq\left(19+\frac94\right)\sqrt{n}l(R)+\delta(R)
\leq\left(20+\frac94\right)\delta(R)<100\delta(R).
\end{align*}
Thus, \eqref{20260503.2253} and hence the above claim holds. In addition,
since $\Gamma$ is a $d$-set, it follows that, for any
$Q,R\in\mathcal{W}_\Omega$ with $\frac98Q\cap R\neq\emptyset$,
\begin{align*}
\mathcal{H}^d\left(\Gamma(\xi_Q,\delta(Q))\right)
\mathcal{H}^d\left(\Gamma(\xi_R,\delta(R))\right)\sim[\delta(Q)\delta(R)]^d
\sim[\delta(R)]^{2d}.
\end{align*}
Combining this, \eqref{20260503.2253}, and H\"older's inequality,
we conclude that, for any $Q,R\in\mathcal{W}_\Omega$
such that $\frac98Q\cap R\neq\emptyset$,
\begin{align*}
\left|y_Q-y_R\right|^q&\leq\left[\frac{1}{\mathcal{H}^d(\Gamma(\xi_Q,\delta(Q)))
\mathcal{H}^d(\Gamma(\xi_R,\delta(R)))}\right.\\
&\quad\left.\times
\iint_{\Gamma(\xi_Q,\delta(Q))\times\Gamma(\xi_R,\delta(R))}
|g(x)-g(y)|\,d\mathcal{H}^d(x)\,d\mathcal{H}^d(y)\right]^q\\
&\lesssim\frac{1}{[\delta(R)]^{2d}}
\iint_{\Gamma(\xi_R,\delta(R))\times\Gamma(\xi_R,100\delta(R))}
|g(x)-g(y)|^q\,d\mathcal{H}^d(x)\,d\mathcal{H}^d(y).
\end{align*}
This, combined with \eqref{20260506.1938} and Lemma \ref{lem:20260502.2249},
further implies that
\begin{align}\label{20260506.1949}
\int_{B(v,r)}|\nabla Eg(x)|^qw(x)\,dx&\lesssim\sum_{\genfrac{}{}{0pt}{}
{R\in\mathcal{W}_\Omega}{R\cap B(v,r)\neq\emptyset}}
\frac{1}{[\delta(R)]^{q+2d}}w(B(v,r)\cap R)\notag\\
&\quad\times\iint_{\Gamma(\xi_R,\delta(R))\times\Gamma(\xi_R,100\delta(R))}
|g(x)-g(y)|^q\,d\mathcal{H}^d(x)\,d\mathcal{H}^d(y).
\end{align}

Next, we show that, for $\mathcal{H}^d$-almost every $x,y\in\Gamma$,
\begin{align}\label{20260506.2019}
\sum_{\genfrac{}{}{0pt}{}{R\in\mathcal{W}_\Omega}{R\cap B(v,r)\neq\emptyset}}
\frac{\boldsymbol{1}_{\Gamma(\xi_R,\delta(R))}(x)
\boldsymbol{1}_{\Gamma(\xi_R,100\delta(R))}(y)}
{[\delta(R)]^{q+d-1}}\lesssim\frac{1}{|x-y|^{q+d-1}},
\end{align}
where the implicit positive constant is independent of $x$ and $y$.
Indeed, let $x,y\in\Gamma$ with $x\neq y$. For any $k\in\mathbb{Z}$, define
\begin{align*}
\mathcal{W}_\Omega^{(k)}:=\left\{R\in\mathcal{W}_\Omega:l(R)=2^k,\
R\cap B(v,r)\neq\emptyset,\ \delta(R)>\frac{|x-y|}{101},
\mbox{\ and\ }\Gamma(\xi_R,\delta(R))\ni x\right\}.
\end{align*}
Let $K:=\lceil\log_2\frac{|x-y|}{404\sqrt{n}}\rceil$. Then it is easy to see
that, for any $k\in\mathbb{Z}\cap(-\infty,K-1]$,
$\mathcal{W}_\Omega^{(k)}=\emptyset$. Moreover, note that, for any
$R\in\mathcal{W}_\Omega$, there exists $\eta_R\in R$ such
that $\delta(R)=\operatorname{dist}(\xi_R,R)=|\xi_R-\eta_R|$. These, together with Lemma
\ref{lem:20260502.2249}, further implies that,
for any $k\in\mathbb{Z}\cap[K,\infty)$, $R\in \mathcal{W}_\Omega^{(k)}$,
and $z\in R$,
\begin{align*}
|z-x|&\leq|z-\eta_R|+|\eta_R-\xi_R|+|\xi_R-x|\leq\sqrt{n}l(R)
+\delta(R)+\delta(R)\leq9\sqrt{n}l(R)=9\sqrt{n}2^k.
\end{align*}
Thus, for any $k\in\mathbb{Z}\cap[K,\infty)$,
\begin{align*}
\bigcup_{R\in\mathcal{W}_\Omega^{(k)}}R\subset B(x,9\sqrt{n}2^k),
\end{align*}
which, together with the fact that cubes in $\mathcal{W}_\Omega^{(k)}$ are
pairwise disjoint, further implies that the number of cubes in
$\mathcal{W}_\Omega^{(k)}$ is uniformly bounded with respect to $k$.
Combining this, Lemma \ref{lem:20260502.2249}, and the definition of $K$, we
conclude that, for any $x\in\Gamma(\xi_R,\delta(R))$ and
$y\in\Gamma(\xi_R,100\delta(R))$ with $x\neq y$,
\begin{align*}
\sum_{\genfrac{}{}{0pt}{}{R\in\mathcal{W}_\Omega}{R\cap B(v,r)\neq\emptyset}}
\frac{1}{[\delta(R)]^{q+d-1}}&\leq\sum_{k=K}^\infty
\sum_{R\in\mathcal{W}_\Omega^{(k)}}\frac{1}{[\delta(R)]^{q+d-1}}
\lesssim\sum_{k=K}^\infty2^{k(-d-q+1)}\\
&\sim2^{K(-d-q+1)}\sim\frac{1}{|x-y|^{q+d-1}},
\end{align*}
and hence \eqref{20260506.2019} holds.
Subsequently, we consider the following two cases for $v$ and $r$.

\emph{Case (1)} $\delta(v)<100r$. In this case,
$\Gamma(\xi_R,100\delta(R))\subset\Gamma(v,20000r)$. Indeed,
from Lemma \ref{lem:20260502.2249}, we deduce that, for any
$z\in\Gamma(\xi_R,100\delta(R))$, $z_R\in R$, and $x\in R\cap B(v,r)$,
\begin{align*}
|z-v|&\leq|z-\xi_R|+|\xi_R-z_R|+|z_R-x|+|x-v|\\
&<100\delta(R)+|\xi_R-z_R|+\sqrt{n}l(R)+r
\leq101\delta(R)+|\xi_R-z_R|+r.
\end{align*}
Taking the infimum over $z_R\in R$, we obtain $|z-v|\leq102\delta(R)+r$.
Moreover, note that there exists $\xi_v\in\Gamma$ such that
$\delta(v)=|v-\xi_v|$. This, together with the assumption that
$\delta(v)<100r$, further implies that, for any $x\in R\cap B(v,r)$,
\begin{align*}
\delta(R)\leq\delta(x)\leq|x-\xi_v|\leq|x-v|+|v-\xi_v|<r+100r=101r.
\end{align*}
Therefore, for any $z\in\Gamma(\xi_R,100\delta(R))$,
$|z-v|<102\cdot101r+r<20000r$ and hence $z\in\Gamma(v,20000r)$.
Thus, $\Gamma(\xi_R,100\delta(R))\subset\Gamma(v,20000r)$. Furthermore,
observe that, for any $R\in\mathcal{W}_\Omega$ with $R\cap B(v,r)\neq\emptyset$
and $x\in B(v,r)\cap R$, it holds that $\delta(x)\geq\delta(R)$. This, combined
with Lemma \ref{lem:20260502.2249}, further implies that
\begin{align*}
w(B(v,r)\cap R)\lesssim[\delta(R)]^{d+1-n}|R|\sim[\delta(R)]^{d+1}.
\end{align*}
From this, the proven conclusion that
$\Gamma(\xi_R,100\delta(R))\subset\Gamma(v,20000r)$,
\eqref{20260506.1949}, and \eqref{20260506.2019}, we deduce that
\begin{align*}
\int_{B(v,r)}|\nabla Eg(x)|^qw(x)\,dx&\lesssim
\iint_{\Gamma(v,20000r)\times\Gamma(v,20000r)}
|g(x)-g(y)|^q\\
&\quad\times\sum_{\genfrac{}{}{0pt}{}{R\in\mathcal{W}_\Omega}
{R\cap B(v,r)\neq\emptyset}}\frac{\boldsymbol{1}_{\Gamma(\xi_R,\delta(R))}(x)
\boldsymbol{1}_{\Gamma(\xi_R,100\delta(R))}(y)}{[\delta(R)]^{q+d-1}}
\,d\mathcal{H}^d(x)\,d\mathcal{H}^d(y)\\
&\lesssim\iint_{\Gamma(v,20000r)\times\Gamma(v,20000r)}
\frac{|g(x)-g(y)|^q}{|x-y|^{q+d-1}}\,d\mathcal{H}^d(x)\,d\mathcal{H}^d(y),
\end{align*}
and hence \eqref{20260503.2213} holds.

\emph{Case (2)} $\delta(v)\geq100r$. In this case, $\Gamma(\xi_R,100\delta(R))
\subset\Gamma(\xi_v,1000\delta(v))$. Indeed, by Lemma \ref{lem:20260502.2249},
we find that, for any $z\in\Gamma(\xi_R,100\delta(R))$,
$z_R\in R$, and $x\in R\cap B(v,r)$,
\begin{align*}
|z-\xi_v|&\leq|z-\xi_R|+|\xi_R-z_R|+|z_R-x|+|x-v|+|v-\xi_v|\\
&<100\delta(R)+|\xi_R-z_R|+\sqrt{n}l(R)+r+\delta(v)\\
&\leq101\delta(R)+|\xi_R-z_R|+r+\delta(v).
\end{align*}
Taking the infimum over $z_R$, we obtain
$|z-\xi_v|\leq102\delta(R)+r+\delta(v)$. Moreover, for any $x\in R\cap B(v,r)$,
\begin{align*}
\delta(R)\leq\delta(x)\leq|x-\xi_v|\leq|x-v|+|v-\xi_v|<r+\delta(v).
\end{align*}
Combining these and the assumption that $\delta(v)\geq100r$, we obtain,
for any $z\in\Gamma(\xi_R,100\delta(R))$, $|z-\xi_v|\leq103(r+\delta(v))
\leq1000\delta(v)$ and hence $z\in\Gamma(\xi_v,1000\delta(x))$. Therefore,
$\Gamma(\xi_R,100\delta(R))\subset\Gamma(\xi_v,1000\delta(v))$. In addition,
from Lemma \ref{lem:250104-11}(i) and the assumption that $\delta(v)\geq100r$,
it follows that, for any $R\in\mathcal{W}_\Omega$ with $R\cap
B(x,r)\neq\emptyset$,
\begin{align*}
w(B(v,r)\cap R)\lesssim r^nw(v)\sim r^n[\delta(R)]^{d+1-n}.
\end{align*}
According to this, the proven conclusion that
$\Gamma(\xi_R,100\delta(R))\subset\Gamma(\xi_v,1000\delta(v))$,
\eqref{20260506.1949}, and \eqref{20260506.2019}, we conclude that
\begin{align*}
\int_{B(v,r)}|\nabla Eg(x)|^qw(x)\,dx&\lesssim
\iint_{\Gamma(\xi_v,1000\delta(v))\times\Gamma(\xi_v,1000\delta(v))}
|g(x)-g(y)|^qr^n\\
&\quad\times\sum_{\genfrac{}{}{0pt}{}{R\in\mathcal{W}_\Omega}
{R\cap B(v,r)\neq\emptyset}}\frac{\boldsymbol{1}_{\Gamma(\xi_R,\delta(R))}(x)
\boldsymbol{1}_{\Gamma(\xi_R,100\delta(R))}(y)}{[\delta(R)]^{q+d+n-1}}
\,d\mathcal{H}^d(x)\,d\mathcal{H}^d(y)\\
&\lesssim\frac{r^n}{[\delta(v)]^n}\iint_{\Gamma(\xi_v,1000\delta(v))
\times\Gamma(\xi_v,1000\delta(v))}
\frac{|g(x)-g(y)|^q}{|x-y|^{q+d-1}}\,d\mathcal{H}^d(x)\,d\mathcal{H}^d(y).
\end{align*}
This, together with the assumption that $\delta(v)\geq100r$,
further implies that
\begin{align*}
&|B(v,r)|^{\frac1p-\frac1q}\left[\int_{B(v,r)}|\nabla Eg(x)|^qw(x)\,dx
\right]^\frac1q\\
&\quad\lesssim\frac{r^\frac{n}{p}}{[\delta(v)]^\frac{n}{q}}
\left[\iint_{\Gamma(\xi_v,1000\delta(v))
\times\Gamma(\xi_v,1000\delta(v))}\frac{|g(x)-g(y)|^q}{|x-y|^{q+d-1}}
\,d\mathcal{H}^d(x)\,d\mathcal{H}^d(y)\right]^\frac1q\\
&\quad\lesssim|B(\xi_v,1000\delta(v))|^{\frac1p-\frac1q}
\left[\iint_{\Gamma(\xi_v,1000\delta(v))\times\Gamma(\xi_v,1000\delta(v))}
\frac{|g(x)-g(y)|^q}{|x-y|^{q+d-1}}
\,d\mathcal{H}^d(x)\,d\mathcal{H}^d(y)\right]^\frac1q\\
&\quad\leq\|g\|_{Q_q^p(\Gamma)}.
\end{align*}

Combining the argument in both Case (1) and Case (2), we find that
\eqref{20260503.2213} holds. Finally, we prove that, for
$\mathcal{H}^d$-almost every $x\in\Gamma$, $(T\circ E)g(x)=g(x)$.
According to \cite[Theorem 7.3]{DFM21}, it suffices to show that
$g\in L_\mathrm{loc}^1(\Gamma,\mathcal{H}^d)$.
Indeed, let $v\in\Gamma$ and $R\in(0,\infty)$. Since $\Gamma$ is a $d$-set,
it follows that $\mathcal{H}^d(\Gamma(v,R))\sim R^d$. Using this and H\"older's
inequality, we conclude that
\begin{align*}
&\iint_{\Gamma(v,R)\times\Gamma(v,R)}|g(x)-g(y)|\,d\mathcal{H}^d(x)
\,d\mathcal{H}^d(y)\\
&\quad\leq R^\frac{d+q-1}{q}\left[\mathcal{H}^d(\Gamma(v,R))\right]^\frac{2}{q'}
\left[\iint_{\Gamma(v,R)\times\Gamma(v,R)}
\frac{|g(x)-g(y)|^q}{R^{d+q-1}}\,d\mathcal{H}^d(x)
\,d\mathcal{H}^d(y)\right]^\frac1q\\
&\quad\lesssim R^\frac{q+d+2d(q-1)-1}{q}
\left[\iint_{\Gamma(v,R)\times\Gamma(v,R)}
\frac{|g(x)-g(y)|^q}{|x-y|^{d+q-1}}\,d\mathcal{H}^d(x)
\,d\mathcal{H}^d(y)\right]^\frac1q\\
&\quad\lesssim R^{\frac{q+d+2d(q-1)+n-1}{q}-\frac{n}{p}}
\|g\|_{Q_q^p(\Gamma)}<\infty,
\end{align*}
which further implies that, for $\mathcal{H}^d$-almost every $y\in\Gamma(v,R)$,
\begin{align*}
\int_{\Gamma(v,R)}|g(x)-g(y)|\,d\mathcal{H}^d(x)<\infty.
\end{align*}
From this and the arbitrariness of $v$ and $R$, we deduce that
$g\in L_\mathrm{loc}^1(\Gamma,\mathcal{H}^d)$ and hence complete the
proof of Theorem \ref{thm:Eg-membership}.
\end{proof}

\begin{remark}
When $p=q=2$, Theorem \ref{thm:Eg-membership} in this case coincides with
\cite[Theorem 7.3]{DFM21}. Moreover, even in the special case of $p=q\neq2$,
the conclusion of Theorem \ref{thm:Eg-membership} is new.
\end{remark}

\subsection{Complex Interpolation}
\label{subsection4.7}

In this subsection, using Theorems \ref{thm:local-trace} and
\ref{thm:Eg-membership} and Proposition \ref{thm:interp-morrey}, we
establish the corresponding interpolation identities for both
$\dot{W}^1\mathcal{M}^p_q(\Omega,w)$ and $Q_q^p(\Gamma)$.
Notably, the complex interpolation method preserves the intrinsic structure of
these spaces and yields the expected interpolation properties of the indices.

\begin{theorem}\label{thm:interp-sobolev-morrey}
Let $1<q_0\leq p_0<\infty$, $1<q_1\leq p_1<\infty$ satisfy
$\frac{n}{p_i}-\frac{n-d-1}{q_i}>0$ for any $i\in\{0,\,1\}$,
$1\leq q\leq p<\infty$, $\theta\in(0,1)$ be such that
\eqref{eq:ratio-condition} and \eqref{eq:interp-indices} hold,
and $w$ be as in \eqref{dw}. Then
\begin{align}\label{eq:interp-sobolev-morrey}
\left[\dot{W}^1\mathcal{M}^{p_0}_{q_0}(\Omega,w),
\dot{W}^1\mathcal{M}^{p_1}_{q_1}(\Omega,w)\right]^\theta
=\dot{W}^1\mathcal{M}^p_q(\Omega,w).
\end{align}
\end{theorem}

\begin{proof}
By Theorem \ref{cor:260117-1}, we obtain, for any $i\in\{0,\,1\}$,
\begin{align*}
(-\Delta)^\frac12:\dot{W}^1\mathcal{M}^{p_i}_{q_i}(\Omega,w)\longrightarrow
\mathcal{M}^{p_i}_{q_i}(\Omega,w)
\end{align*}
is an isomorphism. From this, Proposition \ref{thm:interp-morrey},
and the fact that $(-\Delta)^\frac12$ commutes with interpolation (see, for
instance, \cite[Theorem 4.1.4]{BeLo76}), we infer that
\begin{align*}
(-\Delta)^\frac12:\left[\dot{W}^1\mathcal{M}^{p_0}_{q_0}(\Omega,w),
\dot{W}^1\mathcal{M}^{p_1}_{q_1}(\Omega,w)\right]^\theta\longrightarrow
\left[\mathcal{M}^{p_0}_{q_0}(\Omega,w),
\mathcal{M}^{p_1}_{q_1}(\Omega,w)\right]^\theta=\mathcal{M}^p_q(\Omega,w)
\end{align*}
is an isomorphism, which, together with Theorem \ref{cor:260117-1} again,
further implies that \eqref{eq:interp-sobolev-morrey} holds. This finishes
the proof of Theorem \ref{thm:interp-sobolev-morrey}.
\end{proof}

\begin{remark}\label{rem:diag-sobolev}
We use the same notation as in Theorem \ref{thm:interp-sobolev-morrey}.
If $p_0=q_0$ and $p_1=q_1$, then, applying \cite[Theorem 4.3.1]{BeLo76}, the
reflexivity of $\dot{W}^{1,p_i}(\Omega,w)=L^{p_i}(\Omega,w)$ with
$i\in\{0,1\}$, and \eqref{eq:interp-sobolev-morrey}, we recover the standard
weighted Sobolev interpolation formula
\begin{align*}
\left[\dot{W}^{1,p_0}(\Omega,w),\dot{W}^{1,p_1}(\Omega,w)\right]_\theta=
\left[\dot{W}^{1,p_0}(\Omega,w),\dot{W}^{1,p_1}(\Omega,w)\right]^\theta=
\dot{W}^{1,p}(\Omega,w)
\end{align*}
(see, for instance, \cite[p.\,2440]{ce(jfa-2019)}).
\end{remark}

As a consequence of Theorems \ref{thm:local-trace}, \ref{thm:Eg-membership},
and \ref{thm:interp-sobolev-morrey}, we obtain the interpolation of
corresponding trace spaces.

\begin{corollary}\label{cor:interp-Qpq}
Let $1<q_0\leq p_0<\infty$, $1<q_1\leq p_1<\infty$ satisfy
$\frac{n}{p_i}-\frac{n-d-1}{q_i}>0$ for any $i\in\{0,\,1\}$,
$1\leq q\leq p<\infty$, $\theta\in(0,1)$ be such that
\eqref{eq:ratio-condition} and \eqref{eq:interp-indices} hold,
and $w$ be as in \eqref{dw}. Then
\begin{align*}
\left[Q^{p_0}_{q_0}(\Gamma),Q^{p_1}_{q_1}(\Gamma)\right]^\theta=Q^p_q(\Gamma).
\end{align*}
\end{corollary}

\begin{proof}
From Theorem \ref{thm:local-trace} and \ref{thm:Eg-membership}, we infer that,
for any $i\in\{0,\,1\}$, both
\begin{align*}
T:\dot{W}^1\mathcal{M}^{p_i}_{q_i}(\Omega,w)\longrightarrow
Q^{p_i}_{q_i}(\Gamma)\mbox{\ \ and\ \ }
E:Q^{p_i}_{q_i}(\Gamma)\longrightarrow
\dot{W}^1\mathcal{M}^{p_i}_{q_i}(\Omega,w)
\end{align*}
are bounded. Combining this, \cite[Theorem 4.1.4]{BeLo76}, and Theorem
\ref{thm:interp-sobolev-morrey}, we find that both
\begin{align}\label{20260507.1830}
T:\dot{W}^1\mathcal{M}^p_q(\Omega,w)=
\left[\dot{W}^1\mathcal{M}^{p_0}_{q_0}(\Omega,w),
\dot{W}^1\mathcal{M}^{p_1}_{q_1}(\Omega,w)\right]^\theta\longrightarrow
\left[Q^{p_0}_{q_0}(\Gamma),Q^{p_1}_{q_1}(\Gamma)\right]^\theta
\end{align}
and
\begin{align}\label{20260507.1832}
E:\left[Q^{p_0}_{q_0}(\Gamma),Q^{p_1}_{q_1}(\Gamma)\right]^\theta\longrightarrow
\left[\dot{W}^1\mathcal{M}^{p_0}_{q_0}(\Omega,w),
\dot{W}^1\mathcal{M}^{p_1}_{q_1}(\Omega,w)\right]^\theta=
\dot{W}^1\mathcal{M}^p_q(\Omega,w)
\end{align}
are bounded. Note that
$$T:\dot{W}^1\mathcal{M}^p_q(\Omega,w)\longrightarrow
Q_q^p(\Gamma)$$
is surjective (which can be deduced from Theorem
\ref{thm:Eg-membership}). This, together with \eqref{20260507.1830}, further
implies that $Q_q^p(\Gamma)\subset[Q^{p_0}_{q_0}(\Gamma),
Q^{p_1}_{q_1}(\Gamma)]^\theta$. Conversely, from \eqref{20260507.1832} and
Theorem \ref{thm:Eg-membership}, it follows that, for any
$g\in[Q^{p_0}_{q_0}(\Gamma),Q^{p_1}_{q_1}(\Gamma)]^\theta$, $Eg\in
\dot{W}^1\mathcal{M}^p_q(\Omega,w)$ and hence $(T\circ E)g\in
Q_q^p(\Gamma)$. Therefore, $[Q^{p_0}_{q_0}(\Gamma),Q^{p_1}_{q_1}(\Gamma)]^\theta
\subset Q_q^p(\Gamma)$. This finishes the proof of
Corollary \ref{cor:interp-Qpq}.
\end{proof}

\section{Applications to Divergence-Form Degenerate Second-Order Elliptic Equations}
\label{section5}

Let $w$ be as in \eqref{dw}.
In this section, we denote the weighted measure $w(x)\,dx$ by $dm$.
This section is devoted to applications of weighted Sobolev--Morrey spaces in
the study of divergence-form degenerate second-order elliptic
equations by two subsections. Specifically, in Subsection \ref{subsection5.1},
we recall the concept of
solutions and establish the reverse H\"older inequalities for local solutions.
Applying this, in Subsection \ref{subsection5.2} we establish weighted a
priori estimates for solutions to the Dirichlet problem of divergence-form degenerate second-order elliptic equations, which further allow us to derive regularity results within the
Morrey type framework. We begin by recalling the concepts of function spaces in
addition to $\dot{W}^1\mathcal{M}_q^p(\Omega,w)$.
\begin{itemize}
\item \emph{Weighted Sobolev spaces $\dot{W}^{1,p}(\Omega,\omega)$}. Let
$p\in[1,\infty)$ and $\omega$ be a nonnegative locally integrable function
on $\mathbb{R}^n$. The space $\dot{W}^{1,p}(\Omega,\omega)$
is defined to be the set of all $f\in L^1_\mathrm{loc}(\Omega)$, modulo
constant functions, such that $\nabla f \in L^p(\Omega,\omega)$, equipped with
the norm $\|f\|_{\dot{W}^{1,p}(\Omega,\omega)}:=
\|\nabla f\|_{L^p(\Omega,\omega)}$ (see Subsection \ref{subsection4.1}).

\item The \emph{localized version of $\dot{W}^{1,2}(\Omega,w)$}. Let
$E\subset\Omega$ be an open set. The set $W_r(E,w)$ is defined by setting
\begin{align*}
W_r(E,w):=\left\{f\in L_\mathrm{loc}^1(E):\varphi f\in\dot{W}^{1,2}(\Omega,w)
\mbox{\ for\ any\ }\varphi\in C_\mathrm{c}^\infty(E)\right\}.
\end{align*}

\item \emph{Weighted Sobolev spaces $\dot{W}^{1,p}_0(\Omega,\omega)$
with zero traces}. Let $p\in[1,\infty)$ and $\omega$ be a nonnegative locally
integrable function on $\mathbb{R}^n$. The space
$\dot{W}^{1,p}_0(\Omega,\omega)$ is defined by setting
\begin{align*}
\dot{W}^{1,p}_0(\Omega,\omega):=\left\{f\in\dot{W}^{1,p}(\Omega,\omega):
Tf=0\right\}.
\end{align*}
The dual space of $\dot{W}^{1,p}_0(\Omega,\omega)$ is denoted by
$[\dot{W}^{1,p}_0(\Omega,\omega)]^*$.
\end{itemize}

\subsection{Definition of Solutions and Reverse H\"older Estimates}
\label{subsection5.1}

In this subsection, we recall the concept of solutions and establish the
reverse H\"older estimates for local solutions.
Let $w$ be as in \eqref{dw} and
$L:=-\mathrm{div}(A\nabla\cdot)$ be a divergence-form
degenerate second-order elliptic operator on $\Omega$, where $A$
is a real-valued $n\times n$ matrix of measurable functions on $\Omega$ satisfying the \emph{degenerate
elliptic condition}; i.e., there exists a constant $C_1\in[1,\infty)$ such
that, for any $x\in\Omega$ and $\xi,\eta\in\mathbb{R}^n$,
\begin{align}\label{eq:weighted-ellipticity}
|A(x)\xi\cdot\eta|\leq C_1 w(x)|\xi|\,|\eta|
\mbox{\ \ and\ \ }A(x)\xi\cdot\xi\geq C_1^{-1}w(x)|\xi|^2.
\end{align}
Denote the matrix $w^{-1}A$ by $\mathcal{A}$. Then the matrix $\mathcal{A}$
satisfies the \emph{uniformly elliptic condition}; i.e.,
for any $x\in\Omega$ and $\xi,\eta\in\mathbb{R}^n$,
\begin{align*}
|\mathcal{A}(x)\xi\cdot\eta|\leq C_1 |\xi|\,|\eta|\mbox{\ \ and\ \ }
\mathcal{A}(x)\xi\cdot\xi\geq C_1^{-1}|\xi|^2.
\end{align*}
Let $E\subset\Omega$ be an open set. Recall that a function $u\in
W_r(E,w)$ is called a \emph{weak solution} to the divergence-form
degenerate second-order elliptic equation
\begin{align}\label{eq:elliptic}
Lu=-\mathrm{div}(A(x)\nabla u)=0\mbox{\ \ in\ \ }E
\end{align}
if, for any $\varphi\in C_\mathrm{c}^\infty(E)$,
\begin{align*}
\int_\Omega A(x)\nabla u(x)\cdot\nabla\varphi(x)\,dx=
\int_\Omega\mathcal{A}(x)\nabla u(x)\cdot\nabla\varphi(x)\,dm(x)=0.
\end{align*}

As established in \cite[Lemmas 8.6 and 8.11]{DFM21}, local solutions to
\eqref{eq:elliptic} satisfy interior/boundary Caccioppoli inequalities.

\begin{lemma}\label{lem:8.6}
Let $x\in\mathbb{R}^n$, $r\in(0,\infty)$,
$w$ be as in \eqref{dw}, and $u\in W_r(\Omega(x,2r),w)$ be a
solution to \eqref{eq:elliptic} in $\Omega(x,2r)$. If one of
the following two conditions holds:
\begin{enumerate}
\item[\rm(i)] $B(x,2r)\subset\Omega$,

\item[\rm(ii)] $x\in\Gamma$ and $Tu=0$ in the sense of $\mathcal{H}^d$-almost everywhere on
$\Gamma(x,2r)$,
\end{enumerate}
then
\begin{align*}
\int_{B(x,r)}|\nabla u(y)|^2\,dm(y)\lesssim
r^{-2}\int_{B(x,2r)}|u(y)|^2\,dm(y),
\end{align*}
where the implicit positive constant depends only on $n$, $d$,
and $C_1$ in \eqref{eq:weighted-ellipticity}.
\end{lemma}

Applying the above Caccioppoli inequality, we obtain the interior/boundary
reverse H\"older estimate for local solutions to \eqref{eq:elliptic}.
\begin{lemma}\label{lem:7.1}
Let $x\in\mathbb{R}^n$, $r\in(0,\infty)$,
$w$ be as in \eqref{dw},
and $u\in W_r(\Omega(x,4r),w)$ be a
solution to \eqref{eq:elliptic} in $\Omega(x,4r)$. If one of
the following two conditions holds:
\begin{enumerate}
\item[\rm(i)] $B(x,4r)\subset\Omega$,

\item[\rm(ii)] $x\in\Gamma$ and $Tu=0$ in the sense of $\mathcal{H}^d$-almost everywhere on
$\Gamma(x,4r)$,
\end{enumerate}
then there exists $\varepsilon\in(0,\infty)$, depending only on $n$, $d$,
the structural constants of $w$, and $C_1$ in \eqref{eq:weighted-ellipticity},
such that
\begin{align*}
\left[\fint_{B(x,r)}|\nabla u(y)|^{2+\varepsilon}\,dm(y)\right]
^\frac1{2+\varepsilon}\lesssim\fint_{B(x,2r)}|\nabla u(y)|\,dm(y),
\end{align*}
where the implicit positive constant is independent of $x$, $r$, and $u$.
\end{lemma}

\begin{proof}
From Gehring's lemma (see, for instance, \cite[Theorem 3.22]{bb-book2011}),
it remains to show that there exists $p\in(1,2)$ such that
\begin{align}\label{20260513.1928}
\left[\fint_{B(x,r)}|\nabla u(y)|^2\,dm(y)\right]^\frac12
\lesssim\left[\fint_{B(x,2r)}|\nabla u(y)|^p\,dm(y)\right]^\frac1p,
\end{align}
where the implicit positive constant is independent of $x$, $r$, and $u$.

Now, we prove \eqref{20260513.1928}. Fix $p\in(1,2)$ such that
$p^*:=\frac{np}{n-p}>2$. From Proposition \ref{prop:Poincare2},
we deduce that $U_{B(x,2r)}:=\fint_{B(x,2r)}u(y)\,dm(y)$ exists and
\begin{align}\label{20260514.2214}
\left[\fint_{B(x,2r)}\left|u(y)-u_{B(x,2r)}\right|^2\,dm(y)\right]^\frac12
&\sim\left[\fint_{B(x,2r)}\left|u(y)-U_{B(x,2r)}\right|^2\,dm(y)\right]
^\frac12\notag\\
&\lesssim r\left[\fint_{B(x,2r)}|\nabla u(y)|^p\,dm(y)\right]^\frac1p.
\end{align}
Moreover, $u-U_{B(x,2r)}\in W_r(B(x,4r),w)$ and is also a solution to
\eqref{eq:elliptic} in $\Omega\cap B(x,4r)$. If the assumption (i) holds, then,
using the case (i) of Lemma \ref{lem:8.6} with $u$ therein replaced by
$u-U_{B(x,2r)}$ and \eqref{20260514.2214}, we find that
\begin{align*}
\left[\fint_{B(x,r)}|\nabla u(y)|^2\,dm(y)\right]^\frac12
&\lesssim r^{-1}\left[\fint_{B(x,2r)}\left|u(y)-U_{B(x,2r)}\right|
^2\,dm(y)\right]^\frac12\\
&\lesssim\left[\fint_{B(x,2r)}|\nabla u(y)|^p\,dm(y)\right]^\frac1p,
\end{align*}
and hence \eqref{20260513.1928} holds. If the assumption (ii) holds, then,
from the case (ii) of Lemma \ref{lem:8.6} with $u$ therein replaced by
$u-u_{B(x,2r)}$, \eqref{20260514.2214}, Proposition \ref{prop:Poincare},
Lemma \ref{lem:250104-11}(ii), and H\"older's inequality, we infer that
\begin{align*}
\left[\fint_{B(x,r)}|\nabla u(y)|^2\,dm(y)\right]^\frac12
&\lesssim r^{-1}\left[\fint_{B(x,2r)}\left|u(y)-u_{B(x,2r)}\right|
^2\,dm(y)\right]^\frac12+\fint_{B(x,2r)}|u(y)|\,dy\\
&\lesssim\left[\fint_{B(x,2r)}|\nabla u(y)|^p\,dm(y)\right]^\frac1p
+\frac{1}{r^{d+1}}\int_{B(x,2r)}|u(y)|\,dm(y)\\
&\lesssim\left[\fint_{B(x,2r)}|\nabla u(y)|^p\,dm(y)\right]^\frac1p,
\end{align*}
and hence \eqref{20260513.1928} holds. In conclusion, \eqref{20260513.1928}
holds. This finishes the proof of Lemma \ref{lem:7.1}.
\end{proof}

\subsection{Weighted a Priori Estimates of Solutions}
\label{subsection5.2}

In this subsection, we present weighted a priori estimates in a sharp range
for solutions to the Dirichlet problem
\begin{align}\label{eq:weighted-dirichlet}
\begin{cases}
Lu=-\mathrm{div}(A\nabla u)=F &\mbox{in}\ \Omega,\\
u=g &\mbox{on}\ \Gamma,
\end{cases}
\end{align}
where $A$ is a real-valued $n\times n$ matrix of measurable functions on
$\Omega$ satisfying \eqref{eq:weighted-ellipticity}. It was established in
\cite[Lemma 9.1]{DFM21} that, if $F\in[\dot{W}^{1,2}_0(\Omega,w)]^*$ and
$g\in Q_2^2(\Gamma)$, then there exists a unique weak solution $u\in
\dot{W}^{1,2}(\Omega,w)$ to \eqref{eq:weighted-dirichlet}. Applying the results
in Sections \ref{section2}--\ref{section4}, we extend these results to
both the more general weighted Sobolev setting and the weighted Sobolev--Morrey
setting. We first consider the case where $g\equiv0$ and we begin by recalling
the Muckenhoupt class $A_p(w)$ and the reverse H\"older class $RH_p(w)$
with respect to the measure $w(x)\,dx$.

\begin{definition}
Let $p\in[1,\infty)$, $q\in(1,\infty]$, and $w$ be as in \eqref{dw}.
\begin{enumerate}
\item[\rm(i)] The \emph{Muckenhoupt class $A_p(w)$} is defined to
be the set of all nonnegative locally integrable functions
$\omega$ on $\mathbb{R}^n$ such that
\begin{align*}
[\omega]_{A_p(w)}:=
\begin{cases}
\displaystyle
\sup_B\fint_B\omega(x)\,dm(x)\left[\mathop\mathrm{ess\,inf}_{x\in B}
\omega(x)\right]^{-1}<\infty &\mathrm{if}\ p=1,\\
\displaystyle
\sup_B\fint_B\omega(x)\,dm(x)\left\{\fint_B[\omega(x)]^\frac{1}{1-p}
\,dm(x)\right\}^{p-1}<\infty &\mathrm{if}\ p\in(1,\infty),
\end{cases}
\end{align*}
where the suprema are taken over all balls $B\subset\mathbb{R}^n$.

\item[(ii)] The \emph{reverse H\"older class $RH_q(w)$} is defined
to be the set of all nonnegative locally integrable functions
$\omega$ on $\mathbb{R}^n$ such that
\begin{align*}
[\omega]_{RH_q(w)}:=
\begin{cases}
\displaystyle
\sup_B\left\{\fint_B[\omega(x)]^q\,dm(x)\right\}^\frac1q
\left[\fint_B\omega(x)\,dm(x)\right]^{-1}<\infty &\mathrm{if}\ q\in(1,\infty),\\
\displaystyle
\sup_B\left[\mathop\mathrm{ess\,sup}_{x\in B}\omega(x)\right]
\left[\fint_B\omega(x)\,dm(x)\right]^{-1}<\infty &\mathrm{if}\ q=\infty,\\
\end{cases}
\end{align*}
where the suprema are taken over all balls $B\subset\mathbb{R}^n$.
\end{enumerate}
\end{definition}

\begin{remark}\label{rmk:20260520.2150}
Let $p\in[1,\infty)$, $w$ be as in \eqref{dw},
and $\omega\in A_p(w)$. Then it is easy to show that
$\omega w\in A_p(\mathbb{R}^n)$ and $[\omega w]_{A_p(\mathbb{R}^n)}
\leq[\omega]_{A_p(w)}[w]_{A_p(\mathbb{R}^n)}$.
\end{remark}

In what follows, for any $p\in[1,\infty)$ (resp. $q\in(1,\infty]$),
$\omega\in A_p(w)$ [resp. $\omega\in RH_q(w)$], and any measurable subset
$E\subset\mathbb{R}^n$, let $\omega(E):=\int_E\omega(x)\,dm(x)$.
Applying the Calder\'on--Zygmund decomposition, the weighted boundedness
of the Hardy--Littlewood maximal operator, and an argument similar
to that used in the proof of Shen \cite[Theorem 2.1]{Shen23}
(see also \cite{Shen18, Shen05, YYY21}), we establish the following
real-variable lemma of Gehring type. We omit the details here.

\begin{lemma}\label{lem:Shen}
Let $B_0\subset\mathbb{R}^n$ be a ball, $1\leq p_0<q<p_1<\infty$,
$w$ be as in \eqref{dw},
$G\in L^{p_0}(B_0)$, and $f\in L^q(B_0)$. Suppose that the
following condition holds: there exist $C_0,C_1,\eta\in(0,\infty)$
and $0<\beta_1<1<\beta_2<\infty$, independent of $G$ and $f$, such
that, for any ball $B:=B(x_B,r_B)\subset\mathbb{R}^n$ satisfying
$|B|\leq\beta_1|B_0|$ and either $2B\subset B_0$ or $x_B\in\partial B_0$,
there exist two measurable functions $G_B$ and $R_B$ on $2B$
such that $|G|\leq|G_B|+|R_B|$ on $2B\cap B_0$,
\begin{align*}
&\left[\fint_{2B\cap B_0}|G_B(x)|^{p_0}\,dm(x)\right]^\frac{1}{p_0}\\
&\quad\leq C_0\sup_{B\subset\widetilde{B}}\left[\fint_{\widetilde{B}\cap B_0}
|f(x)|^{p_0}\,dm(x)\right]^\frac{1}{p_0}
+\eta\left[\fint_{\beta_2 B\cap B_0}|G(x)|^{p_0}\,dm(x)\right]^\frac{1}{p_0},
\end{align*}
and
\begin{align*}
&\left[\fint_{2B\cap B_0}|R_B(x)|^{p_1}\,dm(x)\right]^\frac{1}{p_1}\\
&\quad\leq C_1\left\{\left[\fint_{\beta_2 B\cap B_0}
|G(x)|^{p_0}\,dm(x)\right]^\frac{1}{p_0}
+\sup_{B\subset\widetilde{B}}\left[\fint_{\widetilde{B}\cap B_0}
|f(x)|^{p_0}\,dm(x)\right]^\frac{1}{p_0}\right\},
\end{align*}
where the suprema are taken over all balls $\widetilde{B}\supset B$.
Then, for any $\omega\in A_\frac{q}{p_0}(w)\cap RH_s(w)$ with
$s\in((\frac{p_1}{q})',\infty]$, there exist positive constants
$C$ and $\eta_0$, depending only on $C_0$, $C_1$, $n$, $d$, $p_0$,
$q$, $p_1$, $\beta_1$, $\beta_2$, $[\omega]_{A_\frac{q}{p_0}(w)}$, and
$[\omega]_{RH_s(w)}$, such that, if $\eta\in[0,\eta_0)$, then
\begin{align*}
&\left[\frac{1}{\omega(B_0)}\int_{B_0}|G(x)|^q\omega(x)\,dm(x)\right]^\frac1q\\
&\quad\leq C\left\{\left[\fint_{B_0}|G(x)|^{p_0}\,dm(x)\right]
^\frac{1}{p_0}+\left[\frac{1}{\omega(B_0)}\int_{B_0}|f(x)|^q\omega(x)\,dm(x)
\right]^\frac{1}{q}\right\}.
\end{align*}
\end{lemma}

Let $p\in(1,\infty)$ and $\omega\in A_1(\mathbb{R}^n)$. It is well known that
the \emph{weak divergence operator $\mathrm{div}_\omega$} with respect to the
measure $\omega(x)\,dx$ acts as an isomorphism between the weighted Lebesgue
space $L^p(\Omega,\omega)$ and the dual space
$[\dot{W}^{1,p'}_0(\Omega,\omega)]^*$
of the homogeneous weighted Sobolev space $\dot{W}^{1,p'}_0(\Omega,\omega)$
with zero trace. Combining this and Lemmas \ref{lem:7.1} and \ref{lem:Shen}, we
obtain weighted a priori estimates for \eqref{eq:weighted-dirichlet} with
$F\in[\dot{W}^{1,p}_0(\Omega,w)]^*$ for $p$ near 2 and $g\equiv0$ therein.

\begin{theorem}\label{thm:weighted-dirichlet}
Let $\varepsilon$ be as in Lemma \ref{lem:7.1},
$p\in(\max\{2-\varepsilon,\,1\},2+\varepsilon)$,
$r_0\in(\frac{2+\varepsilon}{2+\varepsilon-p},\infty]$,
$w$ be as in \eqref{dw},
and $\omega\in A_1(w)\cap RH_{r_0}(w)$.
Then, for any given $f:=(f^{(1)},\dots,f^{(n)})\in L^p(\Omega,\omega w)$,
there exists a unique weak solution $u\in\dot{W}^{1,p}_0(\Omega,\omega w)$ to
the Dirichlet problem
\begin{align}\label{eq:weighted-dirichlet2}
\begin{cases}
Lu=-\mathrm{div}(A\nabla u)=\mathrm{div}_w(f) &\mbox{in}\ \Omega,\\
u=0 &\mbox{on}\ \Gamma
\end{cases}
\end{align}
such that
\begin{align}\label{eq:weighted-dirichlet-est}
\|\nabla u\|_{L^p(\Omega,\omega w)}\lesssim\|f\|_{L^p(\Omega,\omega w)},
\end{align}
where the implicit positive constant depends only on $n$, $d$, $p$,
$C_1$ in \eqref{eq:weighted-ellipticity}, $[\omega]_{A_1(w)}$,
and $[\omega]_{RH_{r_0}}(w)$.
\end{theorem}

\begin{proof}
Let $f:=(f^{(1)},\dots,f^{(n)})\in L^p(\Omega,\omega w)$. We divide the proof
of Theorem \ref{thm:weighted-dirichlet} into the following four steps.

\emph{Step 1.} We first consider the case $p\in(2,2+\varepsilon)$ and assume
that $f\in L^p(\Omega,\omega w)\cap L^2(\Omega,w)$. In this case, from
\cite[Lemma 9.1]{DFM21}, we deduce that there exists a unique weak solution
$u\in\dot{W}^{1,2}(\Omega,w)$ to \eqref{eq:weighted-dirichlet2}.
We prove \eqref{eq:weighted-dirichlet-est} for $u$.
Indeed, fix $x_0\in\Gamma$ and, for any $k\in\mathbb{N}$, define
$B_k:=B(x_0,2^k)$ and $f_k:=f\boldsymbol{1}_{B_k}$. By \cite[Lemma 9.1]{DFM21}
again, we find that, for any $k\in\mathbb{N}$, there exists a unique weak
solution $u_k\in\dot{W}^{1,2}(\Omega,w)$ to the Dirichlet problem
\begin{align*}
\begin{cases}
-\mathrm{div}(A\nabla u)=\mathrm{div}_w(f_k) &\mbox{in}\ \Omega,\\
u=0 &\mbox{on}\ \Gamma
\end{cases}
\end{align*}
with
\begin{align}\label{20260519.1329}
\|u_k\|_{\dot{W}^{1,2}(\Omega,w)}=\|\nabla u_k\|_{L^2(\Omega,w)}
\lesssim\|\mathrm{div}_w(f_k)\|_{[\dot{W}^{1,2}_0(\Omega,w)]^*}
\sim\|f_k\|_{L^2(\Omega,w)}.
\end{align}

Now, fix $k\in\mathbb{N}$. Let $B:=B(x_B,r_B)\subset\mathbb{R}^n$ be a ball
satisfying $r_B\in(0,\frac{2^k}{8})$ and either $2B\subset B_k$ or $x_B\in
\partial B_k$. Also let $\varphi\in C_\mathrm{c}^\infty(\mathbb{R}^n)$ satisfy
$\boldsymbol{1}_{8B}\leq\varphi\leq\boldsymbol{1}_{16B}$.
From \cite[Lemma 9.1]{DFM21}, we infer that there exist $u_{k,1},u_{k,2}\in
\dot{W}^{1,2}(\Omega,w)$ such that
\begin{align*}
\begin{cases}
-\mathrm{div}(A\nabla u_{k,1})=\mathrm{div}_w(\varphi f_k)
&\mbox{in}\ \Omega,\\
u_{k,1}=0 &\mbox{on}\ \Gamma
\end{cases}
\end{align*}
and
\begin{align*}
\begin{cases}
-\mathrm{div}(A\nabla u_{k,2})=\mathrm{div}_w(f_k-\varphi f_k)
&\mbox{in}\ \Omega,\\
u_{k,2}=0 &\mbox{on}\ \Gamma
\end{cases}
\end{align*}
with
\begin{align}\label{20260519.1402}
\|\nabla u_{k,1}\|_{L^2(\Omega,w)}\lesssim\|\mathrm{div}_w(\varphi f_k)\|
_{[\dot{W}^{1,2}_0(\Omega,w)]^*}\sim\|\varphi f_k\|_{L^2(\Omega,w)}
\end{align}
and
\begin{align*}
\|\nabla u_{k,2}\|_{L^2(\Omega,w)}\lesssim\|\mathrm{div}_w(f_k-\varphi f_k)\|
_{[\dot{W}^{1,2}_0(\Omega,w)]^*}\sim\|f_k-\varphi f_k\|_{L^2(\Omega,w)}.
\end{align*}
Moreover, $u_k=u_{k,1}+u_{k,2}$. Furthermore, we can prove that
\begin{align}\label{20260519.1404}
\frac{w(16B)}{w(2B\cap B_k)}\lesssim1.
\end{align}
Indeed, if $2B\subset B_k$, then, using the doubling property of $w$ [see
Remark \ref{rmk:20260508.2141}(ii)], we obtain
\begin{align*}
\frac{w(16B)}{w(2B\cap B_k)}\leq\frac{w(16B)}{w(2B)}\lesssim
\frac{|16B|}{|2B|}\sim1
\end{align*}
and hence \eqref{20260519.1404} holds. If $x_B\in\partial B_k$, then there
exists a ball $U\subset2B\cap B_k$ with radius $\frac{r_B}{2}$,
which, together with the doubling property of $w$, further implies that
\begin{align*}
\frac{w(16B)}{w(2B\cap B_k)}\leq\frac{w(16B)}{w(U)}\lesssim
\frac{|16B|}{|U|}\sim1,
\end{align*}
and hence \eqref{20260519.1404} holds. In conclusion, \eqref{20260519.1404}
holds. Let $G:=\nabla u_k$, $G_B:=\nabla u_{k,1}$, and
$R_B:=\nabla u_{k,2}$. Then $|G|\leq|G_B|+|R_B|$ on $2B\cap B_k$.
In addition, from the fact that $\operatorname{supp}(\varphi f_k)\subset 16B\cap B_k$,
\eqref{20260519.1402}, and \eqref{20260519.1404}, it follows that
\begin{align}\label{20260519.1433}
\left[\fint_{2B\cap B_k}|G_B(x)|^2\,dm(x)\right]^\frac12&\leq
\left[\frac{1}{w(2B\cap B_k)}\int_\Omega|G_B(x)|^2\,dm(x)\right]^\frac12\notag\\
&\lesssim\left[\frac{1}{w(2B\cap B_k)}
\int_\Omega|\varphi(x)f_k(x)|^2\,dm(x)\right]^\frac12\notag\\
&\leq\left[\frac{w(16B\cap B_k)}{w(2B\cap B_k)}
\fint_{16B\cap B_k}|f_k(x)|^2\,dm(x)\right]^\frac12\notag\\
&\lesssim\left[\fint_{16B\cap B_k}|f_k(x)|^2\,dm(x)\right]^\frac12.
\end{align}
Moreover, note that $f_k-\varphi f_k\equiv0$ in $8B$. Combining this,
\eqref{20260519.1404}, Lemma \ref{lem:7.1}, H\"older's inequality, the fact
that $\operatorname{supp}(u_{k,2})\subset B_k$, \eqref{20260519.1402}, and
\eqref{20260519.1433}, we conclude that
\begin{align*}
\left[\fint_{2B\cap B_k}|R_B(x)|^{2+\varepsilon}\,dm(x)\right]
^\frac{1}{2+\varepsilon}&\leq\left[\frac{w(2B)}{w(2B\cap B_k)}
\fint_{2B}|R_B(x)|^{2+\varepsilon}\,dm(x)\right]^\frac{1}{2+\varepsilon}\\
&\lesssim\fint_{4B}|R_B(x)|\,dm(x)\leq
\left[\fint_{4B}|R_B(x)|^2\,dm(x)\right]^\frac12\\
&\leq\left[\fint_{4B\cap B_k}\left|\nabla u_k(x)-\nabla u_{k,1}(x)\right|
^2\,dm(x)\right]^\frac12\\
&\leq\left[\fint_{4B\cap B_k}\left|G(x)\right|^2\,dm(x)\right]^\frac12
+\left[\fint_{4B\cap B_k}\left|G_B(x)\right|^2\,dm(x)\right]^\frac12\\
&\lesssim\left[\fint_{4B\cap B_k}\left|G(x)\right|^2\,dm(x)\right]^\frac12
+\left[\fint_{16B\cap B_k}|f_k(x)|^2\,dm(x)\right]^\frac12.
\end{align*}
By this, \eqref{20260519.1433}, and Lemma \ref{lem:Shen} with $p_1:=2$,
$p_2:=2+\varepsilon$, $\eta:=0$, $\beta_1:=\frac{1}{8^n}$, $\beta_2:=4$,
$s:=r_0$ therein, we obtain
\begin{align}\label{20260519.1832}
&\left[\frac{1}{\omega(B_k)}\int_{B_k}
|\nabla u_k(x)|^p\omega(x)\,dm(x)\right]^\frac1p\notag\\
&\quad\lesssim\left[\frac{1}{w(B_k)}\int_{B_k}
|\nabla u_k(x)|^2\,dm(x)\right]^\frac12
+\left[\frac{1}{\omega(B_k)}
\int_{B_k}|f_k(x)|^p\omega(x)\,dm(x)\right]^\frac1p.
\end{align}
Moreover, from \eqref{20260519.1329}, the fact that
$\operatorname{supp}(f_k)\subset B_k$,
H\"older's inequality, and the assumption that $\omega\in A_1(w)$,
we deduce that
\begin{align*}
\left[\frac{1}{w(B_k)}\int_{B_k}|\nabla u_k(x)|^2\,dm(x)\right]^\frac12
&\lesssim\left[\frac{1}{w(B_k)}\int_{B_k}|f_k(x)|^2\,dm(x)\right]^\frac12\\
&\leq\left[\frac{1}{w(B_k)}\int_{B_k}|f_k(x)|^p\,dm(x)\right]^\frac1p\\
&\lesssim\left[\frac{1}{\omega(B_k)}\int_{B_k}|f_k(x)|^p\omega(x)\,dm(x)
\right]^\frac1p,
\end{align*}
which, together with \eqref{20260519.1832}, further implies that
\begin{align}\label{20260519.1842}
\|\nabla u_k\|_{L^p(\Omega,\omega w)}\lesssim\|f_k\|_{L^p(\Omega,\omega w)}.
\end{align}

Next, using \cite[Lemma 9.1]{DFM21} again, we find that, for any
$k\in\mathbb{N}$, $u-u_k$ is the unique weak solution to the Dirichlet problem
\begin{align*}
\begin{cases}
-\mathrm{div}(A\nabla u)=\mathrm{div}_w(f-f_k)
&\mbox{in}\ \Omega,\\
u=0 &\mbox{on}\ \Gamma
\end{cases}
\end{align*}
and $\|u-u_k\|_{L^2(\Omega,w)}\lesssim\|f-f_k\|_{L^2(\Omega,w)}$. Since
$\|f-f_k\|_{L^2(\Omega,w)}\to0$ as $k\to\infty$, it follows that there exists
a subsequence of $\{\nabla u_k\}_{k\in\mathbb{N}}$, still denoted by
$\{\nabla u_k\}_{k\in\mathbb{N}}$, such that, for almost every $x\in\Omega$,
$\lim_{k\to\infty}\nabla u_k(x)=\nabla u(x)$. Combining this, Fatou's lemma,
and \eqref{20260519.1842}, we conclude that
\begin{align*}
\|\nabla u\|_{L^p(\Omega,\omega w)}&\leq\liminf_{k\to\infty}
\left[\int_{B_k}|\nabla u_k(x)|^p\omega(x)\,dm(x)\right]^\frac1p\\
&\lesssim\liminf_{k\to\infty}\left[\int_{B_k}|f_k(x)|^p\omega(x)
\,dm(x)\right]^\frac1p\leq\|f\|_{L^p(\Omega,\omega w)},
\end{align*}
and hence \eqref{eq:weighted-dirichlet-est} holds.

\emph{Step 2.} We remove the assumption that $f\in L^2(\Omega,w)$ and consider
directly the general case $f\in L^p(\Omega,\omega w)$ for
$p\in(2,2+\varepsilon)$. In this case, from the density of
$L^p(\Omega,\omega w)\cap L^2(\Omega,w)$ in $L^p(\Omega,\omega w)$, we infer
that there exist $\{f_k\}_{k\in\mathbb{N}}\subset L^p(\Omega,\omega w)\cap
L^2(\Omega,w)$ such that $\|f_k-f\|_{L^p(\Omega,\omega w)}\to0$ as
$k\to\infty$. By the argument in Step 1, we find that, for any $k\in\mathbb{N}$,
there exists a unique weak solution $u_k\in\dot{W}^{1,p}(\Omega,\omega w)\cap
\dot{W}^{1,2}(\Omega, w)$ to the Dirichlet problem
\begin{align*}
\begin{cases}
-\mathrm{div}(A\nabla u)=\mathrm{div}_w(f_k) &\mbox{in}\ \Omega,\\
u=0 &\mbox{on}\ \Gamma.
\end{cases}
\end{align*}
This, combined with the argument in Step 1 again, further implies that, for any
$j,k\in\mathbb{N}$, $u_j-u_k$ is the unique weak solution to
the Dirichlet problem
\begin{align*}
\begin{cases}
-\mathrm{div}(A\nabla u)=\mathrm{div}_w\left(f_j-f_k\right)
&\mbox{in}\ \Omega,\\
u=0 &\mbox{on}\ \Gamma
\end{cases}
\end{align*}
and $\|\nabla u_j-\nabla u_k\|_{L^p(\Omega,\omega w)}\lesssim
\|f_j-f_k\|_{L^p(\Omega,\omega w)}$. Therefore, $\{\nabla u_k\}_{\mathbb{N}}$
is a Cauchy sequence in $L^p(\Omega,\omega w)$ and hence there exists
$v\in L^p(\Omega,\omega w)$ such that $\|\nabla u_k-v\|_{L^p(\Omega,\omega w)}
\to0$ as $k\to\infty$. On the other hand, from the assumption that $\omega\in
A_1(w)$, Remark \ref{rmk:20260520.2150}, and the weighted Poincar\'e inequality
(see, for instance, \cite[Corollary 1.9]{pr(tams-2019)}), it follows
that, for any $k\in\mathbb{N}$ and any ball $B\subset\mathbb{R}^n$ with radius
$r_B\in(0,\infty)$,
\begin{align*}
&\left[\frac{1}{\omega(B)}\int_B|u_k(x)-(u_k)_B|^pw(x)\,dx\right]^\frac1p\\
&\quad\leq\left[\frac{1}{\omega(B)\mathrm{ess\,inf}_B\omega}
\int_B|u_k(x)-(u_k)_B|^p\omega(x)w(x)\,dx\right]^\frac1p\\
&\quad\lesssim r_B\left[\frac{1}{\omega(B)\mathrm{ess\,inf}_B\omega}
\int_B|\nabla u_k(x)|^p\omega(x)w(x)\,dx\right]^\frac1p,
\end{align*}
which further implies that $\{u_k-(u_k)_B\}_{k\in\mathbb{N}}$ is a Cauchy
sequence in $L^p(B,w)$ and hence there exists $u^{(B)}\in L^p(B,w)$ such that
$\|u_k-(u_k)_B-u^{(B)}\|_{L^p(B,w)}\to0$ as $k\to\infty$. Thus,
up to constants, $\{u_k\}_{k\in\mathbb{N}}$ has a limit function $u$ in
$L_\mathrm{loc}^p(\Omega,w)\hookrightarrow L_\mathrm{loc}^1(\Omega)$,
which further implies that $\nabla u$ exists and $\nabla u=v$ almost
everywhere. Moreover, using \eqref{eq:weighted-ellipticity} and Remark
\ref{rmk:20260520.2150}, we obtain, for any $\varphi\in
C_\mathrm{c}^\infty(\Omega)$ and $k\in\mathbb{N}$,
\begin{align*}
\left|\int_\Omega A(x)[\nabla u(x)-\nabla u_k(x)]
\cdot\nabla\varphi(x)\,dx\right|
&\lesssim\int_\Omega|\nabla u(x)-\nabla u_k(x)|\,|\nabla\varphi(x)|w(x)\,dx\\
&\leq\|\nabla u-\nabla u_k\|_{L^p(\Omega,\omega w)}\to0
\end{align*}
and
\begin{align*}
\left|\int_\Omega[f(x)-f_k(x)]\cdot\nabla\varphi(x)\,dm(x)\right|
\lesssim\|f-f_k\|_{L^p(\Omega,\omega w)}\to0
\end{align*}
as $k\to\infty$. This, together with the fact that, for any $k\in\mathbb{N}$
and $\varphi\in C_\mathrm{c}^\infty(\Omega)$,
\begin{align*}
\int_\Omega A(x)\nabla u_k(x)\cdot\nabla\varphi(x)\,dx=
\int_\Omega f_k(x)\cdot\nabla\varphi(x)\,dm(x),
\end{align*}
further implies that, for any $\varphi\in C_\mathrm{c}^\infty(\Omega)$,
\begin{align*}
\int_\Omega A(x)\nabla u(x)\cdot\nabla\varphi(x)\,dx=
\int_\Omega f(x)\cdot\nabla\varphi(x)\,dm(x).
\end{align*}
Consequently, $u\in\dot{W}^{1,p}(\Omega,\omega w)$ and is the unique weak
solution to the Dirichlet problem \eqref{eq:weighted-dirichlet2}. Furthermore,
from Fatou's lemma and the Lebesgue dominated convergence theorem, we deduce
that \eqref{eq:weighted-dirichlet-est} holds. This finishes the proof of
Theorem \ref{thm:weighted-dirichlet} in the case $p\in(2,2+\varepsilon)$.

\emph{Step 3.} Next, we consider the case $p\in(\max\{2-\varepsilon,\,1\},2)$.
We first assume that $f\in L^p(\Omega,\omega w)\cap L^2(\Omega,
\omega w)$. In this case, by a standard duality argument as the discussion
in Step 1, we find that there exists a unique weak solution
$u\in\dot{W}^{1,p}(\Omega,\omega w)\cap\dot{W}^{1,2}(\Omega,w)$ to the
Dirichlet problem \eqref{eq:weighted-dirichlet2} such that
$\|u\|_{L^p(\Omega,\omega w)}\lesssim\|f\|_{L^p(\Omega,\omega w)}$.
For the general case $f\in L^p(\Omega,\omega w)$, we can use an argument
similar to that used in Step 2 to obtain the existence of the weak solution
to \eqref{eq:weighted-dirichlet2} and apply a standard duality argument to
establish \eqref{eq:weighted-dirichlet-est} for this solution.
This finishes the proof of Theorem
\ref{thm:weighted-dirichlet} in the case $p\in(\max\{2-\varepsilon,\,1\},2)$.

\emph{Step 4.} Finally, we consider the case $p=2$. In this case, the existence
and the uniqueness of the weak solution $u\in\dot{W}^{1,2}(\Omega,w)$ to the
Dirichlet problem \eqref{eq:weighted-dirichlet2} follow directly from the
Lax--Milgram theorem. Moreover, by Marcinkiewicz's interpolation theorem, we
establish \eqref{eq:weighted-dirichlet-est} for the solution. Combining the
argument in Step 1 through Step 4 completes the proof of Theorem
\ref{thm:weighted-dirichlet}.
\end{proof}

Applying Theorems \ref{thm:Eg-membership}, and \ref{thm:weighted-dirichlet},
we establish the weighted a priori estimates for solutions to the Dirichlet problem
\eqref{eq:weighted-dirichlet} with $F\in[\dot{W}^{1,p}_0(\Omega,w)]^*$ and
$g\in Q_p^p(\Gamma)$ for $p$ near 2 therein.

\begin{corollary}\label{thm:weighted-dirichlet-cor}
Let $\varepsilon$ be as in Lemma \ref{lem:7.1},
$p\in(\max\{2-\varepsilon,\,1\},2+\varepsilon)$,
and $w$ be as in \eqref{dw}.
Then, for any given $F\in[\dot{W}^{1,p'}_0(\Omega,w)]^*$ and
$g\in Q_p^p(\Gamma)$, there exists a unique weak solution
$u\in\dot{W}^{1,p}(\Omega,w)$ to the Dirichlet problem
\eqref{eq:weighted-dirichlet}:
\begin{align*}
\begin{cases}
Lu=-\mathrm{div}(A\nabla u)=F &\mbox{in}\ \Omega,\\
u=g &\mbox{on}\ \Gamma,
\end{cases}
\end{align*}
such that
\begin{align}\label{20260521.1344}
\|\nabla u\|_{L^p(\Omega,w)}\lesssim
\|F\|_{[\dot{W}^{1,p'}_0(\Omega,w)]^*}+\|g\|_{Q_p^p(\Gamma)},
\end{align}
where the implicit positive constant depends only on $n$, $d$, $p$, and
$C_1$ in \eqref{eq:weighted-ellipticity}.
\end{corollary}

\begin{proof}
Let $F\in[\dot{W}^{1,p'}_0(\Omega,w)]^*$ and $g\in Q_p^p(\Gamma)$.
By Theorem \ref{thm:Eg-membership}, we find that
\begin{align*}
G:=Eg\in \dot{W}^{1,p}(\Omega,w),\ TG=g,\mbox{\ and\ }
\|G\|_{\dot{W}^{1,p}(\Omega,w)}\lesssim\|g\|_{Q_p^p(\Gamma)}.
\end{align*}
We show that $LG=-\mathrm{div}(A\nabla G)\in
[\dot{W}^{1,p'}_0(\Omega,w)]^*$. Indeed, from \eqref{eq:weighted-ellipticity}
and H\"older's inequality, we infer that, for any
$\varphi\in\dot{W}^{1,p'}_0(\Omega,w)$,
\begin{align*}
\langle LG,\varphi\rangle:=\int_\Omega A(x)\nabla G(x)\nabla\varphi(x)\,dx
\end{align*}
and
\begin{align*}
|\langle LG,\varphi\rangle|&=\left|\int_\Omega A(x)
\nabla G(x)\nabla\varphi(x)\,dx\right|\lesssim
\int_\Omega A(x)|\nabla G(x)|\,|\nabla\varphi(x)|w(x)\,dx\\
&\leq\|G\|_{\dot{W}^{1,p}(\Omega,w)}\|\varphi\|_{\dot{W}^{1,p'}_0(\Omega,w)}
\lesssim\|g\|_{Q_p^p(\Gamma)}\|\varphi\|_{\dot{W}^{1,p'}_0(\Omega,w)}.
\end{align*}
Thus, $LG\in[\dot{W}^{1,p'}_0(\Omega,w)]^*$ and
$\|LG\|_{[\dot{W}^{1,p'}_0(\Omega,w)]^*}\lesssim\|g\|_{Q_p^p(\Gamma)}$.
Moreover, by the definition of $[\dot{W}^{1,p'}_0(\Omega,w)]^*$ and an argument
similar to that used in the proof of \cite[Section 1.1.15, Theorem 1]{m2011}, we conclude that there exists
$f:=(f_1,\ldots,f_n)\in[L^p(\Omega,w)]^n$ such that, for any
$\varphi\in\dot{W}^{1,p'}_0(\Omega,w)$,
$$\langle F-LG,\varphi\rangle=\int_{\Omega}f\cdot\nabla\varphi\,dx$$
and $\|F-LG\|_{[\dot{W}^{1,p'}_0(\Omega,w)]^*}\sim\|f\|_{L^p(\Omega,w)}$.
Applying this and Theorem \ref{thm:weighted-dirichlet}, we find that there
exists a unique weak solution $v\in\dot{W}^{1,p}(\Omega,w)$ to the Dirichlet problem
\begin{align*}
\begin{cases}
Lv=-\mathrm{div}(A\nabla v)=F-LG &\mbox{in}\ \Omega,\\
v=0 &\mbox{on}\ \Gamma
\end{cases}
\end{align*}
such that
$$\|v\|_{\dot{W}^{1,p}(\Omega,w)}\lesssim\|f\|_{L^p(\Omega,w)}\sim
\|F-LG\|_{[\dot{W}^{1,p'}_0(\Omega,w)]^*}.$$
Define $u:=v+G$. Then it is easy
to verify that $u$ is a weak solution to \eqref{eq:weighted-dirichlet} and
\begin{align}\label{20260521.1439}
\|\nabla u\|_{L^p(\Omega,w)}=\|u\|_{\dot{W}^{1,p}(\Omega,w)}
\leq\|v\|_{\dot{W}^{1,p}(\Omega,w)}+\|G\|_{\dot{W}^{1,p}(\Omega,w)}
\lesssim\|F\|_{[\dot{W}^{1,p'}_0(\Omega,w)]^*}+\|g\|_{Q_p^p(\Gamma)}.
\end{align}
Finally, the uniqueness of the weak solution follows immediately from the
linearity of the operator $L$ and \eqref{20260521.1439}. This finishes the
proof of Corollary \ref{thm:weighted-dirichlet-cor}.
\end{proof}

\begin{remark}\label{rmk:20260521.1442}
We use the same notation as in Corollary
\ref{thm:weighted-dirichlet-cor}.
\begin{enumerate}
\item[\rm(i)] We point out that Corollary \ref{thm:weighted-dirichlet-cor} is
an extension of \cite[Lemma 9.1]{DFM21} from $p=2$ to
$p\in(\max\{2-\epsilon,\,1\},2+\epsilon)$.

\item[\rm(ii)] It is worth noting that, for any $p\neq2$, one can construct a
Meyers-type Dirichlet problem in the form of \eqref{eq:weighted-dirichlet}
such that its solution does not satisfy the a priori estimate
\eqref{20260521.1344}, which further implies that the range
$p\in(\max\{2-\epsilon,\,1\},2+\epsilon)$ in Corollary
\ref{thm:weighted-dirichlet-cor} is sharp in some sense. Indeed,
let $n\in\mathbb{N}\cap[3,\infty)$, $p\in(2,\infty)$, and $x_0\in\Omega$.
Without loss of generality, we may assume that $x_0:=\mathbf{0}$.
For any $x\in\mathbb{R}^n$, define
\begin{align*}
A(x):=
\begin{cases}
\displaystyle
w(\mathbf{0})\left[(1+\sigma)I_n-n\sigma\frac{x\otimes x}{|x|^2}\right]
&\mbox{if\ }x\in B\left(\mathbf{0},\frac{\delta(\mathbf{0})}{2}\right),\\
w(x)I_n&\mbox{if\ }x\in\Omega\setminus
B\left(\mathbf{0},\frac{\delta(\mathbf{0})}{2}\right),
\end{cases}
\end{align*}
where $I_n$ denotes the $n\times n$ identity matrix,
$\sigma\in(0,\frac{1}{n-1})$, and $\otimes$ denotes
the tensor product of two vectors. Obviously, for any $x\in\Omega\setminus
B(\mathbf{0},\frac{\delta(\mathbf{0})}{2})$, \eqref{eq:weighted-ellipticity}
holds. In addition, for any $x\in B(\mathbf{0},\frac{\delta(\mathbf{0})}{2})$,
$(\frac23)^{n-d-1}w(\mathbf{0})\leq w(x)\leq2^{n-d-1}w(\mathbf{0})$.
From this and the Cauchy--Schwarz inequality, we deduce that, for any $x\in
B(\mathbf{0},\frac{\delta(\mathbf{0})}{2})$ and $\xi,\eta\in\mathbb{R}^n$,
\begin{align*}
A(x)\xi\cdot\xi&=w(\mathbf{0})\left[(1+\sigma)\xi
-n\sigma\frac{(x\otimes x)\xi}{|x|^2}\right]\cdot\xi\\
&=w(\mathbf{0})\left[(1+\sigma)\xi
-n\sigma\frac{(x\cdot\xi)x}{|x|^2}\right]\cdot\xi\\
&=w(\mathbf{0})\left[(1+\sigma)|\xi|^2
-n\sigma\frac{|x\cdot\xi|^2}{|x|^2}\right]\\
&\geq\left(\frac12\right)^{n-d-1}[1-(n-1)\sigma]w(x)|\xi|^2
\end{align*}
and
\begin{align*}
|A(x)\xi\cdot\eta|&=\left|w(\mathbf{0})\left[(1+\sigma)\xi
-n\sigma\frac{(x\otimes x)\xi}{|x|^2}\right]\cdot\eta\right|\\
&=\left|w(\mathbf{0})\left[(1+\sigma)(\xi\cdot\eta)
-n\sigma\frac{(x\cdot\xi)(x\cdot\eta)}{|x|^2}\right]\right|\\
&\leq\left(\frac32\right)^{n-d-1}[1+(n+1)\sigma]w(x)|\xi|\,|\eta|.
\end{align*}
Therefore, $A$ satisfies \eqref{eq:weighted-ellipticity} with
\begin{align*}
C_1:=\max\left\{\left(\frac32\right)^{n-d-1}[1+(n+1)\sigma],\,
\frac{2^{n-d-1}}{1-(n-1)\sigma}\right\}.
\end{align*}
Let $\eta\in C_\mathrm{c}^\infty(\mathbb{R}^n)$ satisfy
\begin{align*}
\begin{cases}
\eta(x)=1 &\mbox{if\ }x\in B\left(x_0,\frac{\delta(x_0)}{4}\right),\\
\eta(x)\in[0,1] &\mbox{if\ }x\in
B\left(x_0,\frac{\delta(x_0)}{2}\right)\setminus
B\left(x_0,\frac{\delta(x_0)}{4}\right),\\
\eta(x)=0 &\mbox{if\ }x\in\mathbb{R}^n\setminus
B\left(x_0,\frac{\delta(x_0)}{2}\right)
\end{cases}
\end{align*}
and define $u_0(x):=\frac{\eta(x)}{|x|^{n-2}}$ and
\begin{align*}
f(x):=
\begin{cases}
0 &\mbox{if\ }x\in B(\mathbf{0},\frac{\delta(\mathbf{0})}{4}),\\
\displaystyle
\frac{A(x)\nabla u_0(x)}{w(x)} &\mbox{if\ }x\in\Omega\setminus
B(\mathbf{0},\frac{\delta(\mathbf{0})}{4}).
\end{cases}
\end{align*}
Then $u_0$ is a weak solution to the Dirichlet problem
\begin{align}\label{20260525.1731}
\begin{cases}
Lu=-\mathrm{div}(A\nabla u)=\mathrm{div}_w(f) &\mbox{in}\ \Omega,\\
u=0 &\mbox{on}\ \Gamma,
\end{cases}
\end{align}
but the estimate \eqref{eq:weighted-dirichlet-est} fails. Indeed,
a direct calculation indicates that, for any
$x\in B(\mathbf{0},\frac{\delta(\mathbf{0})}{4})$,
$\nabla u_0(x)=-\frac{(n-2)x}{|x|^n}$. This further implies that,
for any $x\in B(\mathbf{0},\frac{\delta(\mathbf{0})}{4})$,
\begin{align*}
A(x)\nabla u_0(x)=w(\mathbf{0})(n-2)[1-(n-1)\sigma]\frac{x}{|x|^n}
\end{align*}
and hence $\mathrm{div}(A\nabla u_0)=0$ in
$B(\mathbf{0},\frac{\delta(\mathbf{0})}{4})$. Therefore, $u_0$ is a weak
solution to the Dirichlet problem \eqref{20260525.1731}. Moreover,
\begin{align*}
\|\nabla u_0\|_{L^p(\Omega,w)}^p\geq
\int_{B(\mathbf{0},\frac{\delta(\mathbf{0})}{4})}
\frac{(n-2)w(x)}{|x|^{(n-1)p}}\,dx\sim w(\mathbf{0})
\int_{B(\mathbf{0},\frac{\delta(\mathbf{0})}{4})}
\frac{1}{|x|^{(n-1)p}}\,dx=\infty
\end{align*}
and
\begin{align*}
\|f\|_{L^p(\Omega,w)}^p&=\int_{B(\mathbf{0},\frac{\delta(\mathbf{0})}{2})
\setminus B(\mathbf{0},\frac{\delta(\mathbf{0})}{4})}
\frac{|A(x)\nabla u_0(x)|^p}{[w(x)]^{p-1}}\,dx\\
&\lesssim w(\mathbf{0})\int_{B(\mathbf{0},\frac{\delta(\mathbf{0})}{2})
\setminus B(\mathbf{0},\frac{\delta(\mathbf{0})}{4})}
\left[\frac{|\eta(x)|^p}{|x|^{(n-1)p}}
+\frac{|\nabla\eta(x)|^p}{|x|^{(n-2)p}}\right]\,dx<\infty.
\end{align*}
Thus, \eqref{eq:weighted-dirichlet-est} fails. This establishes the sharpness
of the range $p\in(\max\{2-\epsilon,\,1\},2+\epsilon)$ in Corollary
\ref{thm:weighted-dirichlet-cor}.
\end{enumerate}
\end{remark}

Finally, using Theorems \ref{thm:weighted-dirichlet} and
\ref{prop:250104-11}, we establish the following weighted a priori estimates
for solutions to the Dirichlet problem \eqref{eq:weighted-dirichlet2} within
the setting of Morrey spaces.

\begin{corollary}\label{cor:7.1}
Let $\varepsilon$ be as in Lemma \ref{lem:7.1},
$q\in(\max\{2-\varepsilon,\,1\},2+\varepsilon)$,
$\theta\in(0,\frac{2+\varepsilon-q}{2+\varepsilon})$,
$p\in(q,\frac{qn}{n-(d+1)\theta}]$, and $w$ be as in \eqref{dw}.
Then, for any given $f:=(f^{(1)},\dots,f^{(n)})\in\mathcal{M}^p_q(\Omega,w)$,
there exists a unique weak solution
$u\in\dot{W}^1\mathcal{M}^p_q(\Omega,w)$ to the Dirichlet problem
\begin{align}\label{20260521.2141}
\begin{cases}
Lu=-\mathrm{div}(A\nabla u)=\mathrm{div}_w(f) &\mbox{in}\ \Omega,\\
u=0 &\mbox{on}\ \Gamma
\end{cases}
\end{align}
such that
\begin{equation}\label{eq:weighted-dirichlet-est-Morrey}
\|\nabla u\|_{\mathcal{M}^p_q(\Omega,w)}
\lesssim\|f\|_{\mathcal{M}^p_q(\Omega,w)},
\end{equation}
where the implicit positive constant depends only on $n$, $d$, $q$, $p$, and
$C_1$ in \eqref{eq:weighted-ellipticity}.
\end{corollary}

\begin{proof}
Let $f:=(f^{(1)},\dots,f^{(n)})\in\mathcal{M}^p_q(\Omega,w)$.
By Theorem \ref{prop:250104-11} and the assumption that
$p\in(q,\frac{qn}{n-(d+1)\theta}]$, we find that
\begin{equation}\label{eq:7.14}
\sup_{x\in\Omega,\,r\in(0,\infty)}r^{\frac{n}{p}-\frac{n}{q}}
\left\{\int_{\Omega}|f(y)|^q\left[M_w(\boldsymbol{1}_{\Omega(x,r)})(y)\right]
^\theta w(y)\,dy\right\}^\frac1q\sim\|f\|_{\mathcal{M}^p_q(\Omega,w)},
\end{equation}
which further implies that, for any $x\in\mathbb{R}^n$ and $r\in(0,\infty)$,
\begin{align*}
\|f\|_{L^q(\Omega,r^{n(\frac{q}{p}-1)}[M_w(\boldsymbol{1}_{\Omega(x,r)})]
^\theta w)}\lesssim\|f\|_{\mathcal{M}^p_q(\Omega,w)}.
\end{align*}
For any given $x\in\Omega$ and $r\in(0,\infty)$, let
$$\omega:=r^{n(\frac{q}{p}-1)}\left[M_w(\boldsymbol{1}_{\Omega(x,r)})\right]^\theta.$$
Then an argument similar to that used in the proof of
\cite[Lemma 4.6]{am07} yields $\omega \in A_1(w)\cap RH_s(w)$
for any $s\in(1,\frac{1}{\theta})$. Applying Theorem
\ref{thm:weighted-dirichlet} and the assumption that
$\theta\in(0,\frac{2+\varepsilon-q}{2+\varepsilon})$,
we conclude that there exists a unique solution $u$ to the Dirichlet problem
\eqref{20260521.2141} such that
\begin{align*}
\|\nabla u\|_{L^q(\Omega,r^{n(\frac{q}{p}-1)}
[M_w(\boldsymbol{1}_{\Omega(x,r)})]^\theta w)}\lesssim
\|f\|_{L^q(\Omega,r^{n(\frac{q}{p}-1)}
[M_w(\boldsymbol{1}_{\Omega(x,r)})]^\theta w)},
\end{align*}
where the implicit positive constant is independent of $f$, $x$, and $r$.
From this, Theorem \ref{prop:250104-11}, and \eqref{eq:7.14}, we deduce that
\begin{align*}
\|u\|_{\dot{W}^1\mathcal{M}^p_q(\Omega,w)}
&\sim\sup_{x\in\Omega,\,r\in(0,\infty)}
\|\nabla u\|_{L^q(\Omega,r^{n(\frac{q}{p}-1)}
[M_w(\boldsymbol{1}_{\Omega(x,r)})]^\theta w)}\\
&\lesssim\sup_{x\in\Omega,\,r\in(0,\infty)}
\|f\|_{L^q(\Omega,r^{n(\frac{q}{p}-1)}
[M_w(\boldsymbol{1}_{\Omega(x,r)})]^\theta w)}\lesssim
\|f\|_{\mathcal{M}^p_q(\Omega,w)}
\end{align*}
and hence \eqref{eq:weighted-dirichlet-est-Morrey} holds.
This finishes the proof of Corollary \ref{cor:7.1}.
\end{proof}

\smallskip
\noindent\textbf{Acknowledgements}\qquad
Yoshihiro Sawano would like to express his deep thanks to Professor Tadele Mengesha
for his helpful discussion about the Dirichlet problem on Lipschitz domains.

\bigskip

\smallskip

\noindent Weiyi Kong, Dachun Yang (Corresponding author) and Wen Yuan

\smallskip

\noindent Laboratory of Mathematics and Complex Systems (Ministry of Education
of China), School of Mathematical Sciences, Institute for Advanced Study,
Beijing Normal University,
Beijing 100875, The People's Republic of China

\smallskip

\noindent{\it E-mails:} \texttt{weiyikong@mail.bnu.edu.cn} (W. Kong)

\noindent\phantom{{\it E-mails:} }\texttt{dcyang@bnu.edu.cn} (D. Yang)

\noindent\phantom{{\it E-mails:} }\texttt{wenyuan@bnu.edu.cn} (W. Yuan)

\bigskip

\noindent Yoshihiro Sawano

\smallskip

\noindent Department of Mathematics, Chuo University, Tokyo, 112-8551, Japan

\smallskip

\noindent{\it E-mail:} \texttt{yoshihiro-sawano@celery.ocn.ne.jp} (Y. Sawano)

\bigskip

\noindent Sibei Yang

\smallskip

\noindent School of Mathematics and Statistics, Lanzhou University, Lanzhou 730000,
The People's Republic of China

\smallskip

\noindent{\it E-mail:} \texttt{yangsb@lzu.edu.cn} (S. Yang)

\end{document}